\documentclass[leqno,11pt]{amsart}
\usepackage{amssymb,amsmath,latexsym,amsfonts,amsbsy,amsthm,mathtools,graphicx,color,subeqnarray,cases,tikz} 
\usepackage{amsfonts}
\usepackage{color}
\usepackage{amssymb,amsthm,amsmath,hyperref}
\usepackage{epsfig}
\usepackage{graphicx,float}
\usepackage{CJK}
\usepackage{setspace}
\usepackage{cancel}
\usepackage{verbatim}
\usepackage{ulem}
\newtheorem{thm}{Theorem}[section]
\newtheorem{coro}[thm]{Corollary}
\newtheorem{lem}[thm]{Lemma}
\newtheorem{rem}[thm]{Remark}
\newtheorem{prop}[thm]{Proposition}

\newcommand{\beq}{\begin{eqnarray}}
\newcommand{\eeq}{\end{eqnarray}}
\newcommand{\beqno}{\begin{eqnarray*}}
\newcommand{\eeqno}{\end{eqnarray*}}
\newcommand{\be}{\begin{equation}}
\newcommand{\ee}{\end{equation}}
\newcommand{\beno}{\begin{equation*}}
\newcommand{\eeno}{\end{equation*}}

\newcommand{\D}{\displaystyle}

\newcommand{\R}{\mathbb{R}}
\newcommand\Z{\mathbb{Z}}

\newcommand\eps{\varepsilon}

\newcommand\fr{\frac}

\def\eps{\varepsilon}
\def\e{{\rm e}}
\def\dd{{\rm d}}

\def\R {\mathbb{R}}

\def\T {{\mathbb T}}
\def\D {{\mathbb D}}

\newcommand\les{\lesssim}

\newcommand\Om{\Omega}
\newcommand\om{\omega}

\newcommand\pr{\partial}

\numberwithin{equation}{section}  % 使 equation 环境的编号随 section 递增
\allowdisplaybreaks[4]

\begin{document}
\title[2D Navier-Stokes Boussinesq equations around the Couette flow]{Stability threshold of the Couette flow in 2D Navier-Stokes Boussinesq equations on $\mathbb{T} \times [-1, 1]$ and $\mathbb{R} \times [-1, 1]$}

\author{Nangao Zhang }
\address[N. Zhang]{Research Center for Applied Mathematics and Interdisciplinary Sciences,  School of Mathematics and Statistics, Wuhan Textile University, Wuhan 430200, P. R. China.}
\email{mathngzhang@163.com}

\author{Gengwei Zhong}
\address[G. Zhong]{School of Mathematics and Statistics,  Central China Normal University, Wuhan,  430079,  P. R. China.}
\email{gwzhong@mails.ccnu.edu.cn}

\author{Ruizhao Zi}
\address[R. Zi]{School of Mathematics and Statistics, and Key Laboratory of Nonlinear Analysis \& Applications (Ministry of Education), Central China Normal University, Wuhan,  430079,  P. R. China.}
\email{rzz@ccnu.edu.cn}

\begin{abstract}
In this paper, we study the transition threshold problem for the 2D Navier-Stokes Boussinesq equations around the Couette flow with a constant background temperature in two domains: $\mathbb{T} \times [-1, 1]$ and $\mathbb{R} \times [-1, 1]$. 
 It was proved that if the sizes of initial perturbations of the vorticity and the temperature in Sobolev spaces are $\nu^{\fr13}$ and $\nu^{\fr56}$, respectively,  then this system admits a globally stable solution. 
 Besides, we derive the enhanced dissipation estimate for the
vorticity and temperature, as well as the inviscid damping estimate for the velocity.
The coupling in the equations poses some difficulties to the analysis. We overcome these challenges and establish the results
via the quasi-linear method and energy estimates.

\end{abstract}  

    
\maketitle

\tableofcontents

\section{Introduction and main results}

In this paper, we consider the two-dimensional Navier-Stokes Boussinesq equations in the domain $\T \times [-1,1]$ or $\R \times [-1,1]$:
\begin{equation}\label{1.1}
    \begin{cases}
        \partial_t v+v \cdot \nabla v+\nabla \Pi-\nu \Delta v=\vartheta e_2, \\
        \partial_t \vartheta+v \cdot \nabla \vartheta-\mu\Delta \vartheta=0, \\
        \nabla \cdot v=0, \\ 
         v^{(2)}(t,x,\pm1)=0,\ \ \partial_yv^{(1)}(t,x,\pm1)=1, \ \ \vartheta(t,x,\pm1)=1,\\
    v(0,x,y)=v_{in}(x,y),\ \ \vartheta(0,x,y)=\vartheta_{in}(x,y),
    \end{cases}
\end{equation}
where $v=\left(v^{(1)}(t, x, y), v^{(2)}(t, x, y)\right)^T$is the velocity field, $\vartheta=\vartheta(t, x, y)$ is the temperature, and $\Pi=\Pi(t, x, y)$ is the pressure. 
Moreover, $e_2=(0,1)^{T}$ is the unit vector in the vertical direction. 
In both domains, we impose Navier-slip boundary conditions on the velocity at $y=\pm1$.
The constants $\nu$ and $\mu$ represent the viscosity coefficient and the thermal diffusivity coefficient respectively.
The first two equations in \eqref{1.1} are the incompressible Navier–Stokes equations with the buoyancy forcing term in the vertical direction, and the third equation describes the balance of temperature convection and diffusion.
Throughout this paper, we focus on the case $\nu=\mu$ for simplicity. 

 It is clear that for any $c\in\mathbb{R}$,
\begin{equation}\label{1.2}
v_s = (y, 0), \quad \vartheta_s = 1, \quad p_s = y + c
\end{equation}
is a steady state solution to the system \eqref{1.1}.
Let's introduce the perturbations as $U= v- (y, 0)$, $P=\Pi-p_s$ and $\Theta = \vartheta-1$, then $(U, P, \Theta)$ satisfies
\begin{equation}\label{1.3}
\left\{
\begin{aligned}
&\partial_t U + y\partial_x U + 
\begin{pmatrix}
U^{(2)} \\ 0
\end{pmatrix}
+ U \cdot \nabla U - \nu \Delta U + \nabla P = 
\begin{pmatrix}
0 \\ \Theta
\end{pmatrix}, \\
&\partial_t \Theta + y\partial_x \Theta + U \cdot \nabla \Theta - \nu \Delta \Theta = 0, \quad \nabla \cdot U = 0, \\
&U^{(2)}(t, x, \pm 1)=\partial_y U^{(1)}(t, x, \pm 1) = 0, \; \Theta(t, x, \pm 1) = 0, \\
&U(0, x, y) = U_{in}(x, y), \; \Theta(0, x, y) = \Theta_{in}(x, y).
\end{aligned}
\right.
\end{equation}
Defining the vorticity of $U$ as $\Omega = \nabla \times U = \partial_y U^{(1)} - \partial_x U^{(2)}$, then for $(x,y)\in \T \times [-1,1]$ or $\R \times [-1,1]$, it is possible to remove the pressure term and to rewrite the system \eqref{1.3} as
\begin{equation}\label{eq:1.4}
\left\{
\begin{aligned}
&\partial_t \Omega + y\partial_x \Omega + U \cdot \nabla \Omega - \nu \Delta \Omega = -\partial_x \Theta, \\
&\partial_t \Theta + y\partial_x \Theta + U \cdot \nabla \Theta - \nu \Delta \Theta = 0, \\
&U = \nabla^\perp \Phi = (\partial_y \Phi, -\partial_x \Phi), \quad \Delta \Phi = \Omega,\\
&\Omega(0,x,y)=\Omega_{in}(x,y),\ \ \ \Theta(0,x,y)=\Theta_{in}(x,y),\\ 
&\Omega(t,x,\pm1)=0,\ \ \Theta(t,x,\pm1)=0. 
\end{aligned}
\right.
\end{equation}
 
In this paper, we study the following {\it stability threshold problem} for \eqref{eq:1.4}:

 Given norms $\|\cdot\|_{Y_1}$ and $\|\cdot\|_{Y_2}$, find the smallest $\alpha=\alpha\left(Y_1, Y_2\right)$ and $\beta=\beta\left(Y_1, Y_2\right)$ such that

$$
\begin{aligned}
    & \left\|U_{in}\right\|_{Y_1} \leq \nu^\alpha \text { and }\left\|\Theta_{in}\right\|_{Y_2} \leq \nu^\beta \Rightarrow \text { stability; } \\
    & \left\|U_{in}\right\|_{Y_1} \gg \nu^\alpha \text { or }\left\|\Theta_{in}\right\|_{Y_2} \gg \nu^\beta \Rightarrow \text { instability. }
\end{aligned}
$$
The exponents $\alpha$ and $\beta$ are referred to as the transition thresholds.

\subsection{Background}
The Navier-Stokes Boussinesq (NSB) equations describes the dynamic behavior of incompressible and inhomogeneous fluid flow driven by buoyancy. This system is usually used to simulate atmospheric science, ocean circulation and mantle convection, and to model Rayleigh-Baynard convection in laminar and turbulent state (cf. \cite{Doering1996}). 
In addition, the 2D NSB system retains some key characteristics of the three-dimensional Euler equation and Navier-Stokes equations, such as the vortex stretching mechanism \cite{Majda2002}.
Due to the wide application of the NSB equations in physics and its significance in mathematics, the theory of this equation has received extensive attention. Depending on different research fields, the studies on this equation can be roughly classified into the following categories: 

(I) Constant equilibrium, $v_s =\vartheta_s=\mbox{constant}$. Chae \cite{Chae} obtained
 the global-in-time regularity with either the zero diffusivity or the zero viscosity, Cao and Wu \cite{Cao2013} proved the global regularity of this system  with vertical
dissipation. 
 For other related results, see \cite{Danchin,Doering2018,Li2016,Tao2020}.

 (II) Shear flow with hydrostatic equilibrium, $v_s = (v_s(y),0), \ \vartheta_s=\vartheta_s(y)$.
In the region $\mathbb{T} \times \mathbb{R}$, Tao and Wu~\cite{Tao2019} obtained the linear stability around Couette flow.
Bedrossian et al.~\cite{Bedrossian20231} established the nonlinear shear-buoyancy instability until the nonlinear timescale $t \approx \varepsilon^{-2}$. 
 Ren and Wei~\cite{Ren} proved the stability threshold of $\frac{1}{3}$ near Couette flow with 
Richardson number $\gamma^2>\frac{1}{4}$.
 By constructing some Fourier multipliers, Deng et al.~\cite{Deng2020} proved that if the initial perturbation satisfies that for $b>1$,
\[
\|\Omega_{in}\|_{H^b} \ll \nu^{\frac{1}{2}},\quad \|\Theta_{in}\|_{H^b} \ll \nu,\quad \||D_x|^{\frac{1}{3}}\Theta_{in}\|_{H^b} \ll \nu^{\frac{5}{6}},
\]
the Boussinesq system with full dissipation is nonlinear stability near Couette flow. 
For the case with only vertical dissipation, they also obtained the nonlinear stability results when
\[
\|\Omega_{in}\|_{H^b} \ll \nu^{\frac{2}{3}},\quad \|\Theta_{in}\|_{H^b} \ll \nu,\quad \||D_x|^{\frac{1}{3}}\Theta_{in}\|_{H^b} \ll \nu ,
\] 
with $b>\frac{4}{3}$.
Later, Zhang and Zi \cite{ZhangZi} obtained the nonlinear stability of Couette flow for the Boussinesq equations with vertical dissipation under the condition that
\[
\|U_{in}\|_{H^{b+1}} + \nu^{-\frac{1}{2}}\|\Theta_{in}\|_{H^b} + \nu^{-\frac{1}{3}}\||D_x|^{\frac{1}{3}} \Theta_{in}\|_{H^b} \ll \nu^{\frac{1}{3}}, 
\]
with $b > 7$, thereby improving the results in \cite{Deng2020}.
Recently, by using the idea of quasi-linearization, Niu and Zhao \cite{Niu} obtained the nonlinear stability of Couette flow for the Boussinesq equations with full dissipation provided that
\[
\|U_{in}\|_{H^{b+1}} + \nu^{-\frac{1}{3}}\|\Theta_{in}\|_{H^b} + \nu^{-\frac{1}{3}}\|\left \langle D_x \right \rangle  \Theta_{in}\|_{H^b}\ll \nu^{\frac{1}{3}}, \quad b > 5.
\]
In the region $\mathbb{R}^2$, Chen et al. \cite{Chen-Wang}  investigated the nonlinear stability of the Couette flow, and obtained the stability threshold $\{\frac{1}{3}+,\frac{2}{3}+\}$. 

 For the region $\mathbb{T} \times [-1, 1]$, Liang et al. \cite{Liang2025} recently studied the nonlinear enhanced dissipation and the inviscid damping effect when the steady state solution is near-Couette shear flow. 
When the steady state solution is $v_s = (y, 0)$ and $\vartheta_s = 1$, by using the resolvent estimate, Masmoudi et al.~\cite{Masmoudi2021} proved that if the initial perturbation satisfies
\begin{equation}\notag
\|U_{in}\|_{H^2} \ll \nu^{\frac{1}{2}}, \quad \| \Theta_{in} \|_{H^1} + \left\| \langle \partial_x \rangle^{\frac{1}{6}} \Theta_{in} \right\|_{H^1} \ll \nu^{\frac{11}{12}},
\end{equation}
then this system with  the non-slip boundary condition is asymptotically stable.
When the steady state solution is the monotone shear flow and a constant background temperature, $v_s = (U(t,y), 0)$ and $\vartheta_s = 1$,
Chen and Li \cite{Chen2026} obtained the nonlinear stability of the NSB equations with non-slip boundary condition under the condition
\begin{equation}\notag
\|U_{in}\|_{H^2} \ll \nu^{\frac{1}{2}}, \quad \left\| \langle \partial_x \rangle\Theta_{in} \right\| \ll \nu^{\frac{5}{6}}.
\end{equation}
For other results related to \eqref{1.1}, refer to \cite{Bianchini2025,Knobel,Liang2026,MSZ,zhai2022}.

From the above introduction, it can be seen that some progress has been made regarding the results of this system, but there are still some shortcomings. For instance, in region $\mathbb{T} \times [-1, 1]$ with Navier-slip boundary condition, the optimal transition threshold has not yet been determined,  and there are few results in region $\mathbb{R} \times [-1, 1]$ as far as we know.
This is precisely the issue that this paper is concerned about.
It is worth noting that when we ignore the effect of temperature, \eqref{1.1} reduces to the two-dimensional incompressible Navier-Stokes system. 
The stability threshold of Couette flow for this equation in region $\mathbb{T} \times [-1, 1]$ has made considerable progress. 
For the non-slip boundary condition, Chen et al. \cite{Chen2020} proved the nonlinear stability provided the initial perturbation of velocity satisfies $\|U_{in}\|_{H^2} \ll \nu^{\frac{1}{2}}$. 
For the Navier-slip boundary condition, Bedrossian et al. \cite{BHW} obtained the same threshold of nearly-Couette Shear flows. 
Recently,  Wei and Zhang \cite{wei-zhang2026} proved the
optimal threshold $\alpha=\frac{1}{3}$ by using the quasi-linear method.
For the region $\mathbb{R} \times [-1, 1]$, Arbon and Bedrossian \cite{Arbon2025} studied the stability
problem in three domains: $\mathbb{R} \times \mathbb{R}$, $\mathbb{R} \times [0,\infty)$ and $\mathbb{R} \times [-1, 1]$ with Navier-slip boundary conditions, showing that the threshold is $\alpha\leq\frac{1}{2}+$. Also for Navier-slip boundary conditions, Chen et al. \cite{miao2025} obtained the optimal threshold $\alpha=\frac{1}{3}$ by using the quasi-linear method and dyadic decomposition.    
For other results, see \cite{BPM1,BPM2,BPM,BMV,BVW,Coti, Kelvin, Lihui2026}.

\subsection{Notations and conventions}
\begin{enumerate}
    \item[1.] We define $D:=\mathbb{D} \times I$, where $\mathbb{D}:=\mathbb{T}$ or $\mathbb{R}$, and $I:=[-1,1]$. For a function $f=f(t,x, y): D \rightarrow \mathbb{R}$,  we denote the Fourier transform with respect to the variable $x$ as
\begin{align}\notag
&f_k(t,y)=\frac{1}{2\pi}\int_{\D} f(t,x,y)\e^{-ikx}\dd x,\notag
\end{align}
here $k\in \mathbb{Z}$ or $k\in \mathbb{R}$. In particular, for $k\in \mathbb{Z}$, $f_{\neq}=f-f_0$.

\item[2.] For a given function $f$, when $1\leq p,q \leq \infty$ and $T_0\leq t\leq T_1$, we denote
\begin{align}\notag
& \|f\|_{L^q_{x}}=\|f\|_{L^q_{x}(\mathbb{D})},\ \ \ \ \ \ \  \|f\|_{L^q_{y}}=\|f\|_{L^q_{y}(I)},\notag\\
&\|f\|_{L^q}=\|f\|_{L^q_{x,y}(D)},\ \ \ \ \ \|f\|_{L^pL^q}=\|f\|_{L^p(T_0,T_1)L^q_{x,y}(D)}.\notag
\end{align}
The $H^q$-norm can be defined analogously.
To simplify the notation, we also use $\|f\|=\|f\|_{L^2}=\|f\|_{L^2_{x,y}(D)}$.
\item[3.] We use the notation $a \lesssim b$  to indicate that  there exists a constant $C>0$ (independent of $\nu, t$) such that $a \leq C b$ (we define $a \gtrsim b$ analogously). 
In addition, we use the notation $a \approx b$ when there exists $C>0$ such that $C^{-1} b \leq a \leq C b$.
\item[4.] For given functions $f,g$, $1\leq q \leq \infty$, we denote
\begin{equation}\notag
 \|(f,g)\|_{L^q}=\|f\|_{L^q}+\|g\|_{L^q}, \ \ \ \  \langle f\rangle=\left(1+|f|^2\right)^{\frac{1}{2}},
 \end{equation}
 and use the notation $\langle f,g\rangle$ to denote the inner product of $f$ and $g$ in $L^2(D)$.
\end{enumerate}

\subsection{ Main results}

\begin{thm}\label{thm1.1}
Suppose that $(\Omega_{in},\Theta_{in})$ is the initial data for \eqref{eq:1.4} with $\mathbb{D}=\mathbb{T}$. There exists a small constant $\varepsilon_0>0$ independent of $\nu$ such that if
\begin{equation}\label{1.4}
    \left\|\Omega_{in}\right\|_{H^3_{x,y}}+\nu^{-\frac{1}{2}}\left\|\Theta_{in}\right\|_{H^4_{x,y}}\leq \varepsilon_0\nu^{\fr13},
\end{equation}
then for $0<\delta\ll 1$ (depending only on universal constants), $t\geq 0$ and $0<\nu\leq 1$,  the corresponding solution $(\Omega,\Theta)$ to \eqref{eq:1.4} satisfies
\begin{equation}\label{1.6}
\left\|\Omega_{\neq}(t)\right\|_{L^2_{x,y}}+\left\|\Theta_{\neq}(t)\right\|_{L^2_{x,y}}\lesssim \varepsilon_0{\rm e}^{-\delta\nu^{\frac13} t},   
\end{equation}
\begin{equation}\label{1.7}
    \left\|U_{\neq}(t)\right\|_{L^2_{x,y}}\lesssim \varepsilon_0(1+t)^{-1}{\rm e}^{-\frac{\delta}{2}\nu^{\frac13} t}.
\end{equation}
\end{thm}

\begin{rem}
 Compared with the results in Zhang and Zi \cite{ZhangZi}, we successfully extend the whole-space case ($y\in \mathbb{R}$) to the case with boundaries, and consider two different types of domains
separately.
 According to \cite{wei-zhang2023,wei-zhang2026}, the stability threshold for $\Omega_{in}$ should be optimal. 
Due to the presence of the buoyancy term $ -\partial_x \Theta$ in the vorticity equation, and since the multiplier method cannot be applied as in the whole-space case, the stability threshold for the temperature is larger.
\end{rem}

\begin{thm}\label{thm1.2}
Suppose that $(\Omega_{in},\Theta_{in})$ is the initial data for \eqref{eq:1.4} with $\mathbb{D}=\mathbb{R}$. There exists a small constant $\varepsilon_0>0$ independent of $\nu$ such that if
\begin{equation}\label{1.8}
    \|\Omega_{in}\|_{H^3_{x,y}\cap L^1_xH^3_{y}}+\nu^{-\frac{1}{2}}\|\Theta_{in}\|_{H^4_{x,y}\cap L^1_xH^4_{y}}\leq \varepsilon_0\nu^{\fr13},
\end{equation}
then for $0<\delta\ll 1$, $t\geq 0$ and $0<\nu\leq 1$,  the solution $(\Omega,\Theta)$ to \eqref{eq:1.4} satisfies
\begin{equation}\label{1.9}
\left\|(1+\nu^{\fr13}t\mathcal{M}_1)^\delta\Omega(t)\right\|_{L^2_{x,y}}+\nu^{-\frac{1}{2}}\left\|(1+\nu^{\fr13}t\mathcal{M}_1)^\delta\Theta(t)\right\|_{L^2_{x,y}}\lesssim \varepsilon_0,   
\end{equation}
\begin{equation}\label{1.10}
    \left\|(1+\nu^{\fr13}t\mathcal{M}_1)^\delta \mathcal{M}_2 U(t)\right\|_{L^2_{t}L^2_{x,y}}\lesssim \varepsilon_0.
\end{equation}
Here $\mathcal{M}_1=\mathcal{M}_1(\partial_x)$ and $\mathcal{M}_2(\partial_x)$ are the Fourier multipliers, which are defined as (see \cite{miao2025})
\begin{align}\label{M}
\mathcal{M}_1(k)=|k|^{\frac23}\chi+1-\chi, \quad \mathcal{M}_2 (k)=|k|^{\frac12}+|k|,
\end{align}
where $\chi$ is a cut-off function satisfying $\chi(k)=\begin{cases}
1 & |k|\leq \frac12, \\
-|k|^2-\frac{|k|}{2}+\frac{3}{2} & \frac12<|k|< 1, \\
0 & |k|\geq 1.
\end{cases}$ 
\end{thm}

Now, let us briefly point out the differences between this paper and previous results, as well as the novel contributions of this work.
In \cite{Masmoudi2021}, the authors established the stability of solutions via resolvent estimates and space-time estimates on $\mathbb{T} \times [-1, 1]$ with the non-slip or Navier-slip boundary condition. 
These methods are effective for dealing with the initial-boundary value problem with boundary conditions, but it seems that the optimal transition threshold of $\Omega$ and $\Theta$ with the Navier-slip boundary condition have not been obtained yet.
Unlike their approach, in the domain $\mathbb{T} \times [-1, 1]$, this paper mainly uses quasi-linearization, energy estimate, operator semigroup estimate and other methods.
More precisely, inspired by \cite{wei-zhang2026}, we first subtract the zero mode of the linearized part of system \eqref{eq:1.4}, and then construct an approximate solution $(\omega_{a},\theta_{a})$ for the remaining equations \eqref{2.6} (see \eqref{2.1}--\eqref{2.8} for more details).
It is worth noting that the initial value of the approximate equations \eqref{2.7} coincides with that of the original equations \eqref{2.6}. 
Due to the relatively simple structure of the approximate equations, we can derive the explicit expression for the solution, thereby obtaining some estimates through direct calculation (see Proposition \ref{p2.1}).
Subsequently, by perturbing the solution of system \eqref{2.6} around the approximate solution $(\omega_{a},\theta_{a})$, we derive a new system \eqref{2.9}.
Although the initial data for this system are zero, the system is more challenging to handle because of the presence of numerous nonlinear terms.
We divide the time interval into two parts, $0<t\leq \nu^{-\frac{1}{6}}$ and $t>\nu^{-\frac{1}{6}}$ , for separate consideration. For $0<t\leq \nu^{-\frac{1}{6}}$, thanks to the favorable transport structure, we can avoid directly estimating the most troublesome nonlinear terms, namely $u_{ob} \cdot \nabla \omega$ and $u_{ob} \cdot \nabla \theta$, by directly establishing the zero-order energy estimate. Unfortunately, there is a buoyancy term $-\partial_x\theta$ in the vorticity equation, this term is the first derivative of temperature.
Raising the regularity of the solution would break the transport structure, and some boundary terms would become difficult to handle. Consequently, the energy estimate would fail to close. (see \eqref{123}--\eqref{456}). 
This presents a significant difference compared to the Navier-Stokes equations.
In order to overcome this difficulty, we first derive some estimates for the solution to a new linear equation (see Proposition \ref{p4.1}), and then use the semigroup estimate to control $ \|\theta\|_{X_1}$ (see Proposition \ref{p4.7}). 
Note that this norm involves the term $\nu^{\frac{1}{2}} \| \partial_x \theta\|_{L^1 L^2}$, which overcomes the difficulty caused by the buoyancy term. 
But it also leads to a loss of $\nu^{\frac{1}{2}}$ in temperature, which seems inevitable.
This is precisely why the scale requirement of temperature is $\nu^{\frac{5}{6}}$.
In view of the above analysis, we can get the energy estimate in the short time scale  (see Proposition \ref{p4.8}).
For the long time scale $t>\nu^{-\frac{1}{6}}$, our strategy is to decompose system \eqref{2.9} into two subsystems, system \eqref{aaa}--\eqref{bbb} and system \eqref{aaaa}--\eqref{bbbb} according to the better estimate for the good derivative $\partial_y+t\partial_x$. 
For system \eqref{aaaa}--\eqref{bbbb}, we first establish the decay properties of the solution to the linearized equation via energy estimates, and then derive estimates for the nonlinear equation by using operator semigroup estimates (see Propositions \ref{p4.8a} and \ref{prop:2.6}). 
For system \eqref{aaa}--\eqref{bbb}, we obtain the desired estimates through a decomposition similar to that in \cite{wei-zhang2026} (see Proposition \ref{p4.12a}). 
Based on the above analysis, we are able to derive estimates on the long-time scale (see \eqref{4.43}--\eqref{4.49}). 
Combining the estimates on the short time scale, we can finally get Theorem \ref{thm1.1} (see \eqref{2.16}--\eqref{2.19}).
For the region $\mathbb{R} \times [-1, 1]$, the decay rate of the perturbation variable changes from exponential to algebraic. Compared with the previous region, we mainly focus on the estimates for $0<|k|<1$ (the case $|k|\geq1$ is similar to before). Inspired by \cite{miao2025}, we finally establish Theorem \ref{thm1.2} via the quasi-linear method and dyadic decomposition.

The rest of this paper is organized as follows. 
In Section \ref{s2}, we give the outline of the proof of the main theorems. In Section \ref{s3}, we establish the estimates for the approximate solutions $\omega_a$ and $\theta_a$, and prove Propositions \ref{p2.1} and \ref{linear}. In Sections \ref{s4} and \ref{s5}, we derive the estimates for the perturbation variables $\omega$ and $\theta$ in the regions 
$\mathbb{T} \times I$ and $\mathbb{R} \times I$, respectively, and prove Propositions \ref{p2.2} and \ref{ap3.2}. In Section \ref{s6}, we discuss further improvements on the threshold for the temperature. The appendix provides the proof of the main results in Section \ref{s6} and a key inequality.

\section{Outline of the proof}\label{s2}

In this section, we give the outline of the proofs of Theorems \ref{thm1.1} and \ref{thm1.2}. 
For convenience, we first introduce some notation regarding the initial data.
For $k\in \mathbb{Z}$ or $k\in \mathbb{R}$, $y\in [-1,1]$, with $\Omega_{in}$ and $\Theta_{in}$ given in \eqref{eq:1.4}, we denote
\begin{align}\label{aaa3.1}
&E_{1,j}(k):=\|(k, \partial_y)^j\Omega_{in,k}\|_{L^2_y}\ (j=0,1,2,3),\notag\\ 
&E_{2,j}(k):=\|(k, \partial_y)^j\Theta_{in,k}\|_{L^2_y}\ (j=0,1,2,3,4),\notag\\
& E_{1}(k):=\sum_{j=0}^3E_{1,j}(k),\ \ \ \ \ \ E_{2}(k):=\sum_{j=0}^4E_{2,j}(k).
\end{align}
In particular, for $k\in \mathbb{R}$, we denote
\begin{align}\label{aaa3.2}
  &\mathcal{E}_{1,in}:=\|\Omega_{in}\|_{H^3_{x,y}}+\|\Omega_{in}\|_{L^1_xH^3_{y}},  \ \ \ \ \ \ \ \mathcal{E}_{2,in}:=\|\Theta_{in}\|_{H^4_{x,y}}+\|\Theta_{in}\|_{L^1_xH^4_{y}},\notag\\ &\mathcal{E}_{in}:=\mathcal{E}_{1,in}+\nu^{-\frac{1}{2}}\mathcal{E}_{2,in}.
\end{align}
Thus, one has that
\begin{align}\label{a3.2}
&\|E_{1}(k)\|_{L^2_k}\lesssim\sum_{j=0}^3\|E_{1,j}(k)\|_{L^2_k} \lesssim \|\Omega_{in}\|_{H^3_{x,y}}\lesssim \mathcal{E}_{1,in},\notag\\ 
&\|E_{2}(k)\|_{L^2_k}\lesssim\sum_{j=0}^4\|E_{2,j}(k)\|_{L^2_k} \lesssim \|\Theta_{in}\|_{H^4_{x,y}}\lesssim\mathcal{E}_{2,in}.
\end{align}

\subsection{For Theorem \ref{thm1.1}}

Note that $k\in \Z$ at this time.
Let $ W_{1,in}(y)=P_0 \Omega_{in}$, $ W_{2,in}(y)=P_0 \Theta_{in}$, then it holds
\[
\Omega_{in}(x,y) = W_{1,in}(y) + P_{\neq}\Omega_{in}(x,y), \ \ \Theta_{in}(x,y) = W_{2,in}(y) + P_{\neq}\Theta_{in}(x,y).
\]
Similar to \cite{wei-zhang2026}, we define the heat extension with homogeneous Dirichlet boundary condition as
\[
W_1(t,y) = {\rm e}^{\nu t\partial_y^2} W_{1,in}(y),\ \ \ \ \ W_2(t,y) =  {\rm e}^{\nu t\partial_y^2} W_{2,in}(y).
\]
Then we define the following shear flows:
\begin{equation}\label{2.1}
V_{1}(t,y) = y - \partial_y(-\partial_y^2)^{-1} W_1(t,y), \ \ \ \ V_{2}(t,y)=1+W_2(t,y).
\end{equation}
Thanks to $W_{1,in}(\pm 1) =W_{2,in}(\pm 1) =0$, it holds for $t > 0$ that
\begin{align}
& \|V_1 - y\|_{C^3} = \|\partial_y(-\partial_y^2)^{-1} W_1\|_{C^3} \lesssim \|W_1\|_{C^2} \leq \|W_{1,in}\|_{C^2} \lesssim \|W_{1,in}\|_{H^3} \leq \|\Omega_{in}\|_{H^3}, \label{2.2}\\
& \|V_{1}' - 1\|_{C^2} + \|V_{1}''\|_{C^1} + \|V_{1}'''\|_{L^\infty} \lesssim \|V_{1} - y\|_{C^3} \lesssim \|\Omega_{in}\|_{H^3},\label{2.3}\\
&\|V_2 - 1\|_{C^2} + \|V_2'\|_{C^1} + \|V_2''\|_{L^\infty} \lesssim \|W_2\|_{C^2} \lesssim \|W_{2,in}\|_{C^2}\lesssim \|\Theta_{in}\|_{H^3}.\label{2.4}
\end{align}
We denote
\begin{align}\label{2.5}
&u_{ob}(t,x,y) = v(t,x,y) - (V_1(t,y), 0),\ \ \ \omega_{ob}(t,x,y) = \Omega(t,x,y) - W_1(t,y), \notag\\
&\theta_{ob}(t,x,y)=\Theta(t,x,y)-W_2(t,y).
\end{align}
Then we have 
\begin{eqnarray}\label{2.6}
\left\{
\begin{aligned}
&\partial_t \omega_{ob} + V_{1} \partial_x \omega_{ob}- \nu \Delta \omega_{ob}  - V_{1}'' \partial_x \phi_{ob} + u_{ob} \cdot \nabla \omega_{ob} = -\partial_x \theta_{ob}, \\
&\partial_t \theta_{ob}  + V_{1} \partial_x \theta_{ob}- \nu \Delta \theta_{ob}-V_{2}' \partial_x \phi_{ob} + u_{ob} \cdot \nabla \theta_{ob} =0, \\
&u_{ob} = (\partial_y, -\partial_x)\phi_{ob}, \quad \phi_{ob} = \Delta^{-1} \omega_{ob}, \\
&\omega_{ob}(t,x,\pm 1) = 0, \quad \theta_{ob}(t,x,\pm 1) = 0, \\ 
&\omega_{ob}|_{t=0} = P_{\neq}\Omega_{in},\ \ \  \theta_{ob}|_{t=0} =P_{\neq}\Theta_{in}.
\end{aligned}
\right.
\end{eqnarray}
Next, we construct approximate solutions $\omega_{a}$ and $\theta_{a}$ of $\omega_{ob}$ and $\theta_{ob}$ respectively. 
\begin{eqnarray}\label{2.7}
\left\{\begin{aligned}
&\partial_t \omega_{a} + V_1 \partial_x \omega_{a}- \nu t^2 \partial_x^2 \omega_{a}  = -\partial_x \theta_{a}, \\
&\partial_t \theta_{a}  + V_1 \partial_x \theta_{a}- \nu t^2 \partial_x^2 \theta_{a} =0, \ \ \ \ \ \ (x,y)\in \mathbb{T}\times [-1,1],\\ 
&\omega_{a}(t,x,\pm 1) = 0, \quad \theta_a(t,x,\pm 1) = 0, \\ 
&\omega_{a}|_{t=0} = P_{\neq}\Omega_{in},\ \ \ \ \theta_{a}|_{t=0} = P_{\neq}\Theta_{in}.
\end{aligned}\right.
\end{eqnarray}
One can verify that $P_0 \omega_{a} =P_0 \theta_{a}= 0$. Let
\[
\phi_{a} = \Delta^{-1} \omega_{a}, \quad u_{a} = (\partial_y, -\partial_x) \phi_{a},
\]
and
\begin{align}\label{2.8}
&Er_{\omega}^1 = \partial_t \omega_{a}  + V_1 \partial_x \omega_{a}- \nu \Delta \omega_{a}+\partial_x\theta_{a}, \quad Er^1 = Er_{\omega}^1 - V_1'' \partial_x \phi_{a} + u_{a} \cdot \nabla \omega_{a},\notag\\ 
&Er_{\theta}^2 = \partial_t \theta_{a} + V_1 \partial_x \theta_{a}- \nu \Delta \theta_{a} , \quad Er^2 = Er_{\theta}^2 -V_2' \partial_x \phi_{a} + u_{a} \cdot \nabla \theta_{a}, 
\end{align}
and define the perturbations as
\[
\omega = \omega_{ob} - \omega_{a}, \quad u = u_{ob} - u_{a}, \quad \phi = \phi_{ob} - \phi_{a},\quad \theta = \theta_{ob} - \theta_{a}, 
\]
we then deduce that
\begin{equation}\label{2.9}
\left\{
\begin{aligned}
&\partial_t \omega + V_1 \partial_x \omega- \nu \Delta \omega  - V_1'' \partial_x \phi + u_{ob} \cdot \nabla \omega + u \cdot \nabla \omega_{a} + Er^1 =-\partial_x \theta , \\
&\partial_t \theta + V_1 \partial_x \theta- \nu \Delta \theta -V_2' \partial_x \phi + u_{ob} \cdot \nabla \theta + u \cdot \nabla \theta_{a} + Er^2 =0, \\
&u = (\partial_y, -\partial_x) \phi, \quad \phi = \Delta^{-1} \omega, \\
&\omega(t,x,\pm 1) =\omega|_{t=0} = 0,\quad \theta(t,x,\pm 1) =\theta|_{t=0} = 0.
\end{aligned}
\right.
\end{equation}

As for system \eqref{2.7}, we have the following conclusions, which will be proved in Section \ref{s3}.
\begin{prop}\label{p2.1}
Under the conditions of Theorem \ref{thm1.1}, the solution $(\omega_{a},\theta_{a})$ of \eqref{2.7}  satisfies that for $t \geq 0$,
\begin{align}
& \|\partial_x \theta_{a}(t)\|_{L^\infty} + \|(\partial_y + t\partial_x)\theta_{a}(t)\|_{L^\infty} \lesssim {\rm e}^{-2\nu^{\frac{1}{3}} t} \|\Theta_{in}\|_{H^4}, \label{2.10}\\
& \|\partial_x \omega_{a}(t)\|_{L^\infty} + \|(\partial_y + t\partial_x)\omega_{a}(t)\|_{L^\infty} \lesssim {\rm e}^{-2\nu^{\frac{1}{3}} t} (\|\Omega_{in}\|_{H^3}+\nu^{-\frac{1}{3}}\|\Theta_{in}\|_{H^4}), \label{2.11}\\
& \|\partial_x^j u_{a}^{(2)}(t)\| + \|\partial_x^j(u_{a}^{(1)} - t u_{a}^{(2)})(t)\| \notag\\ &\ \ \lesssim (1+t)^{-2} {\rm e}^{-2\nu^{\frac{1}{3}} t} (\|\Omega_{in}\|_{H^3}+\nu^{-\frac{1}{3}}\|\Theta_{in}\|_{H^4}),\ j=0,1,2,3, 4,\label{2.12}\\
& (1+t) \|\partial_x^ju_{a}^{(2)}(t)\|_{L^\infty} + \|\partial_x^ju_{a}^{(1)}(t)\|_{L^\infty} \notag\\ 
&\ \ \lesssim (1+t)^{-1} {\rm e}^{-2\nu^{\frac{1}{3}} t}(\|\Omega_{in}\|_{H^3}+\nu^{-\frac{1}{3}}\|\Theta_{in}\|_{H^4}), \ j=0,1,2,\label{2.13}\\
& \|Er^1(t)\| \lesssim \nu (1+t){\rm e}^{-2\nu^{\frac{1}{3}} t} (\|\Omega_{in}\|_{H^3}+\nu^{-\frac{1}{3}}\|\Theta_{in}\|_{H^4}) \notag\\
&\ \ \ \ \ \ \ \ \ \ \ \ \ \ \ \ \ +  (1+t)^{-2} {\rm e}^{-2\nu^{\frac{1}{3}} t} (\|\Omega_{in}\|_{H^3}+\nu^{-\frac{1}{3}}\|\Theta_{in}\|_{H^4})^2,\label{2.14}\\
& \|Er^2(t)\|\lesssim  \nu (1+t) {\rm e}^{-2\nu^{\frac{1}{3}} t} \|\Theta_{in}\|_{H^4} \notag\\
&\ \ \ \ \ \ \ \ \ \ \ \ \ \ \ \ \  +  (1+t)^{-2} {\rm e}^{-2\nu^{\frac{1}{3}} t}\|\Theta_{in}\|_{H^4}(\|\Omega_{in}\|_{H^3}+\nu^{-\frac{1}{3}}\|\Theta_{in}\|_{H^4}).\label{2.15}
\end{align}
\end{prop}

With the estimates of $(\omega_{a},\theta_{a})$, to prove Theorem~\ref{thm1.1}, we need to derive the estimates for $(\omega,\theta)$, which are given as follows.
\begin{prop}\label{p2.2}
Under the conditions of Theorem \ref{thm1.1}, the solution $(\omega,\theta)$ to \eqref{2.9} satisfies
that for $ t \geq 0 $,
\begin{equation}\label{222}
\|\omega_{\neq}(t)\|+\nu^{-\frac{1}{2}}\|\theta_{\neq}(t)\| \lesssim \e^{-\delta \nu^{\frac{1}{3}} t} (\nu^{\frac{1}{3}} \|\Omega_{in}\|_{H^3}+\nu^{-\frac{1}{6}}\|\Theta_{in}\|_{H^3}).
\end{equation}
\end{prop}

The detailed proof of Proposition \ref{p2.2} will be given in Section \ref{s4}. 
Based on the above two propositions, we are ready to prove Theorem~\ref{thm1.1}.

\begin{proof}[\textbf{Proof of Theorem~\ref{thm1.1}}]
By using \eqref{2.12}, one has
\begin{align*}
\|u_{a}(t)\| &\leq \|u_{a}^{(1)}(t)\| + \|u_{a}^{(2)}(t)\| \\
&\leq (1 + t) \|u_{a}^{(2)}(t)\| + \|(u_{a}^{(1)} - t u_{a}^{(2)})(t)\| \\
&\lesssim(1 + t)^{-1} \e^{-2\nu^{\frac{1}{3}} t} (\|\Omega_{in}\|_{H^3}+\nu^{-\frac{1}{3}}\|\Theta_{in}\|_{H^4}).
\end{align*}
By applying \eqref{2.10}--\eqref{2.11} and $P_0 \theta_{a}=P_0 \omega_{a} = 0$, we get
\begin{equation}\notag
\|\theta_{a}(t)\| \leq \|\partial_x \theta_{a}(t)\| \lesssim \|\partial_x \theta_{a}(t)\|_{L^\infty} \lesssim \e^{-2\nu^{\frac{1}{3}} t} \|\Theta_{in}\|_{H^4},
\end{equation}
\begin{equation}\notag
\|\omega_{a}(t)\| \leq \|\partial_x \omega_{a}(t)\| \lesssim \|\partial_x \omega_{a}(t)\|_{L^\infty} \lesssim \e^{-2\nu^{\frac{1}{3}} t} (\|\Omega_{in}\|_{H^3}+\nu^{-\frac{1}{3}}\|\Theta_{in}\|_{H^4}).
\end{equation}
Notice that $u = (\partial_y, -\partial_x) \Delta^{-1} \omega$. Then we use Proposition~\ref{p2.2} to get
\begin{equation}\notag
\|\theta_{\neq}(t)\| \lesssim \e^{-\delta \nu^{\frac{1}{3}} t} (\nu^{\frac{1}{3}} \|\Omega_{in}\|_{H^3}+\nu^{-\frac{1}{6}}\|\Theta_{in}\|_{H^3}),
\end{equation}
\begin{equation}\notag
\|u_{\neq}(t)\| \lesssim \|\omega_{\neq}(t)\| \lesssim \e^{-\delta \nu^{\frac{1}{3}} t} (\nu^{\frac{1}{3}} \|\Omega_{in}\|_{H^3}+\nu^{-\frac{1}{6}}\|\Theta_{in}\|_{H^3}).
\end{equation}
Thanks to \eqref{2.5}, we have  
\begin{equation}\label{2.16}
 U_{\neq} = P_{\neq} u_{ob},\   \Omega_{\neq} = P_{\neq} \omega_{ob}, \ \Theta_{\neq} = P_{\neq} \theta_{ob}. 
\end{equation}
Thus, it holds that
\begin{align}\label{2.17}
\|\Theta_{\neq}(t)\| &= \|P_{\neq} \theta_{ob}(t)\| \leq \|\theta_{a}(t)\| + \|P_{\neq} \theta(t)\| \notag\\
&\lesssim \e^{-2\nu^{\frac{1}{3}} t} \|\Theta_{in}\|_{H^4}+\e^{-\delta \nu^{\frac{1}{3}} t} (\nu^{\frac{1}{3}} \|\Omega_{in}\|_{H^3}+\nu^{-\frac{1}{6}}\|\Theta_{in}\|_{H^3})\notag\\
&\lesssim \e^{-\delta \nu^{\frac{1}{3}} t}(\nu^{\frac{1}{3}}\|\Omega_{in}\|_{H^3}+\nu^{-\frac{1}{6}}\|\Theta_{in}\|_{H^4}),
\end{align}
\begin{align}\label{2.18}
\|\Omega_{\neq}(t)\| &= \|P_{\neq} \omega_{ob}(t)\| \leq \|\omega_{a}(t)\| + \|P_{\neq} \omega(t)\| \notag\\
&\lesssim \e^{-2\nu^{\frac{1}{3}} t} (\|\Omega_{in}\|_{H^3}+\nu^{-\frac{1}{3}}\|\Theta_{in}\|_{H^4})+\e^{-\delta \nu^{\frac{1}{3}} t} (\nu^{\frac{1}{3}} \|\Omega_{in}\|_{H^3}+\nu^{-\frac{1}{6}}\|\Theta_{in}\|_{H^3})\notag\\
&\lesssim \e^{-\delta \nu^{\frac{1}{3}} t}(\|\Omega_{in}\|_{H^3}+\nu^{-\frac{1}{3}}\|\Theta_{in}\|_{H^4}),
\end{align}
and
\begin{align}\label{2.19}
\|U_{\neq}(t)\| &= \|P_{\neq} u_{ob}(t)\| \leq \|u_{a}(t)\| + \|P_{\neq} u(t)\| \notag\\
&\lesssim(1 + t)^{-1} \e^{-2\nu^{\frac{1}{3}} t} \|\Omega_{in}\|_{H^3} +  \e^{-\delta \nu^{\frac{1}{3}} t} \nu^{\frac{1}{3}} (\|\Omega_{in}\|_{H^3}+\nu^{-\frac{1}{2}}\|\Theta_{in}\|_{H^4}) \notag\\
&\lesssim(1 + t)^{-1} \e^{-\delta \nu^{\frac{1}{3}} t / 2} (\|\Omega_{in}\|_{H^3}+\nu^{-\frac{1}{2}}\|\Theta_{in}\|_{H^4}).
\end{align}
By combining \eqref{2.17}--\eqref{2.19}, we get \eqref{1.6}--\eqref{1.7}, and thus complete the proof of Theorem~\ref{thm1.1}.
\end{proof}

\subsection{For Theorem \ref{thm1.2}}

Note that $k\in \R$. Unlike before, we directly construct approximate solutions $\omega_{a}$ and $\theta_{a}$ of $\Omega$ and $\Theta$.
Inspired by \cite{miao2025}, let's consider the following system:
\begin{eqnarray}\label{a3.3}
\left\{\begin{aligned}
&\partial_t \omega_{a} + y \partial_x \omega_{a}+\nu (1-t^2\partial_x^2 ) \omega_{a}  = -\partial_x \theta_{a}, \\
&\partial_t \theta_{a}  +y \partial_x \theta_{a}+\nu (1-t^2\partial_x^2 )\theta_{a} =0,\\ 
&\omega_{a}(t,x,\pm 1) = 0, \quad \theta_{a}(t,x,\pm 1) = 0, \\ 
&\omega_{a}(0,x,y)= \Omega_{in}(x,y),\ \ \ \ \theta_{a}(0,x,y)=\Theta_{in}(x,y).
\end{aligned}\right.
\end{eqnarray}
 Let
\[
\phi_{a} = \Delta^{-1} \omega_{a}, \quad u_{a} = (\partial_y, -\partial_x) \phi_{a},
\]
and
\begin{align}\label{a3.4}
&Er_{\omega}^1 = \partial_t \omega_{a}  + y \partial_x \omega_{a}- \nu \Delta \omega_{a}+\partial_x\theta_{a}, \quad Er^1 = Er_{\omega}^1 + u_{a} \cdot \nabla \omega_{a},\notag\\ 
&Er_{\theta}^2 = \partial_t \theta_{a} + y \partial_x \theta_{a}- \nu \Delta \theta_{a} , \quad Er^2 = Er_{\theta}^2  + u_{a} \cdot \nabla \theta_{a},
\end{align}
and denote
\[
\omega = \Omega - \omega_{a}, \quad u = U - u_{a}, \quad \phi = \Phi - \phi_{a},\quad \theta = \Theta - \theta_{a}, 
\]
then we derive
\begin{equation}\label{a3.5}
\left\{
\begin{aligned}
&\partial_t \omega + y \partial_x \omega- \nu \Delta \omega + U \cdot \nabla \omega + u \cdot \nabla \omega_{a} + Er^1 =-\partial_x \theta , \\
&\partial_t \theta + y \partial_x \theta- \nu \Delta \theta + U \cdot \nabla \theta + u \cdot \nabla \theta_{a} + Er^2 =0, \\
&u = (\partial_y, -\partial_x) \phi, \quad \phi = \Delta^{-1} \omega, \\
&\omega|_{y=\pm 1} =\omega|_{t=0} = 0,\quad \theta|_{y=\pm 1} =\theta|_{t=0} = 0.
\end{aligned}
\right.
\end{equation}

For system \eqref{a3.3}, we have the following estimates, which will be proved in Section \ref{s3}.
\begin{prop}\label{linear}
Under the conditions of Theorem \ref{thm1.2}, the solution $(\omega_{a},\theta_{a})$ to \eqref{a3.3}  satisfies for $m\geq 0$ that
\begin{align}
&\|(1+ \nu^{\frac13}t\mathcal{M}_1)^{m}[\theta_{a},(\partial_y+t\partial_x)\theta_{a}](t)\|\lesssim_m (1+\nu t^3)^{-\frac14}{\rm e}^{-\nu t}\mathcal{E}_{2,in}, \label{a3.6}\\
&\|(1+ \nu^{\frac13}t\mathcal{M}_1)^{m}\partial_x\theta_{a}(t)\|\lesssim_m (1+\nu t^3)^{-\frac34}{\rm e}^{-\nu t}\mathcal{E}_{2,in},\label{a3.7}\\
&\|(1+ \nu^{\frac13}t\mathcal{M}_1)^{m}\partial_x\theta_{a}(t)\|_{L^\infty}\lesssim_m (1+\nu t^3)^{-1}{\rm e}^{-\nu t}\mathcal{E}_{2,in}, \label{a3.8}\\
&\|(1+ \nu^{\frac13}t\mathcal{M}_1)^{m}(\partial_y+t\partial_x)\theta_{a}(t)\|_{L^\infty}\lesssim_m (1+\nu t^3)^{-\frac12}{\rm e}^{-\nu t}\mathcal{E}_{2,in}, \label{a3.9}\\
&\|(1+ \nu^{\frac13}t\mathcal{M}_1)^{m}\partial^2_x\theta_{a}(t)\|_{L^\infty}\lesssim_m (\nu t^3)^{-\frac13}(1+\nu t^3)^{-\frac76}{\rm e}^{-\nu t}\mathcal{E}_{2,in}, \label{a3.10}
\end{align}
\begin{align}
&\|(1+ \nu^{\frac13}t\mathcal{M}_1)^{m}[\omega_{a},(\partial_y+t\partial_x)\omega_{a}](t)\|\lesssim_m (1+\nu t^3)^{-\frac14}{\rm e}^{-\nu t}\mathcal{E}_{in}, \label{paxpayomL L2}\\
&\|(1+ \nu^{\frac13}t\mathcal{M}_1)^{m}\partial_x\omega_{a}(t)\|\lesssim_m (1+\nu t^3)^{-\frac34}{\rm e}^{-\nu t}\mathcal{E}_{in},\label{paxomL2}\\
&\|(1+ \nu^{\frac13}t\mathcal{M}_1)^{m}\partial_x\omega_{a}(t)\|_{L^\infty}\lesssim_m (1+\nu t^3)^{-1}{\rm e}^{-\nu t}\mathcal{E}_{in}, \label{omL Linfty}\\
&\|(1+ \nu^{\frac13}t\mathcal{M}_1)^{m}(\partial_y+t\partial_x)\omega_{a}(t)\|_{L^\infty}\lesssim_m (1+\nu t^3)^{-\frac12}{\rm e}^{-\nu t}\mathcal{E}_{in}, \label{paxomL Linfty}\\
&\|(1+ \nu^{\frac13}t\mathcal{M}_1)^{m}\partial^2_x\omega_{a}(t)\|_{L^\infty}\lesssim_m (\nu t^3)^{-\frac13}(1+\nu t^3)^{-\frac76}{\rm e}^{-\nu t}\mathcal{E}_{in}, \label{paxxxomL Linfty}
\end{align}
and 
\begin{align}
&\|(1+ \nu^{\frac13}t\mathcal{M}_1)^{m}[u_a^{(1)},(u_a^{(1)}-tu_a^{(2)})](t)\|_{L^\infty}\lesssim_m (1+t)^{-1}{\rm e}^{-\nu t}\mathcal{E}_{in},\label{u1LLinfty}\\
&\|(1+ \nu^{\frac13}t\mathcal{M}_1)^{m}u_a^{(2)}(t)\|_{L^\infty}\lesssim_m(1+t)^{-2}\ln(1+t){\rm e}^{-\nu t}\mathcal{E}_{in},\label{u12L22}
\end{align}
\begin{align} \label{a3.18}
&\|(1+ \nu^{\frac13}t\mathcal{M}_1)^{m}E_r^2(t)\|\notag\\
\lesssim_m& \Big((1+t)^{-1}(1+\nu t^3)^{-\frac14} \mathcal{E}_{in} +\nu(1+\nu t^3)^{-\frac14}+\nu^{\frac23}(1+\nu t^3)^{-\frac{5}{12}}\Big){\rm e}^{-\nu t}\mathcal{E}_{2,in},
\end{align}
\begin{align} \label{ErL2}
&\|(1+ \nu^{\frac13}t\mathcal{M}_1)^{m}E_r^1(t)\|\notag\\
\lesssim_m& \Big((1+t)^{-1}(1+\nu t^3)^{-\frac14} \mathcal{E}_{in} +\nu(1+\nu t^3)^{-\frac14}+\nu^{\frac23}(1+\nu t^3)^{-\frac{5}{12}}\Big){\rm e}^{-\nu t}\mathcal{E}_{in}.
\end{align}
\end{prop} 

With the estimates of $\omega_{a}$ and $\theta_{a}$ in hand, the next step is to investigate the estimates of $\omega$ and $\theta$, which are given as follows.
\begin{prop}\label{ap3.2}
Under the conditions of Theorem \ref{thm1.2}, the solution $(\omega,\theta)$ to \eqref{a3.4}--\eqref{a3.5}  satisfies for $0\leq t\leq \nu^{-\frac16}$,
\begin{align}\label{a3.20}
\|(1+ \nu^{\frac13}t\mathcal{M}_1)^{\delta}\om(t)\|+\nu^{-\frac12}\|(1+ \nu^{\frac13}t\mathcal{M}_1)^{\delta}\theta(t)\|\lesssim \nu^{\frac13}\ln(1+\nu^{-1})\mathcal{E}_{in}.
\end{align}
For $t\geq \nu^{-\frac16}$,
\begin{equation}\label{a3.20a}
    \|(1+ \nu^{\frac13}t\mathcal{M}_1)^{\delta}\om(t)\|+\nu^{-\frac12}\|(1+ \nu^{\frac13}t\mathcal{M}_1)^{\delta}\theta(t)\|+\|(1+ \nu^{\frac13}t\mathcal{M}_1)^{\delta}\mathcal{M}_2u(t)\|_{L^2L^2}\lesssim \mathcal{E}_{in}.
\end{equation}
\end{prop}

The detailed proof of Proposition \ref{ap3.2} will be given in Section \ref{s5}. 
Based on the above two propositions, we are going to prove Theorem \ref{thm1.2}.

\begin{proof}[\textbf{Proof of Theorem~\ref{thm1.2}}]
Due to the fact that $\Om=\om_a+\om$ and $\Theta=\theta_{a}+\theta$, we can conclude
\begin{align*}
\|(1+ \nu^{\frac13}t\mathcal{M}_1)^{\delta}\Om\|&\leq \|(1+ \nu^{\frac13}t\mathcal{M}_1)^{\delta}\om_a\|+\|(1+ \nu^{\frac13}t\mathcal{M}_1)^{\delta}\om\|,\\
\|(1+ \nu^{\frac13}t\mathcal{M}_1)^{\delta}\Theta\|&\leq \|(1+ \nu^{\frac13}t\mathcal{M}_1)^{\delta}\theta_{a}\|+\|(1+ \nu^{\frac13}t\mathcal{M}_1)^{\delta}\theta\|.
\end{align*}
By taking $m=\delta$ in \eqref{a3.6} and \eqref{paxpayomL L2} respectively, we obtain
\begin{align*}
\|(1+ \nu^{\frac13}t\mathcal{M}_1)^{\delta}\om_a\|\lesssim\mathcal{E}_{in},\ \ \|(1+ \nu^{\frac13}t\mathcal{M}_1)^{\delta}\theta_{a}\|\lesssim\mathcal{E}_{2,in}.
\end{align*}
By Proposition \ref{ap3.2}, we obtain
\begin{align*}
\|(1+ \nu^{\frac13}t\mathcal{M}_1)^{\delta}\om\|+\nu^{-\frac12}\|(1+ \nu^{\frac13}t\mathcal{M}_1)^{\delta}\theta\|\lesssim \mathcal{E}_{in}.
\end{align*}
Notice that $\mathcal{E}_{in}=\mathcal{E}_{1,in}+\nu^{-\frac{1}{2}}\mathcal{E}_{2,in}$, we can deduce \eqref{1.9}.

For the inviscid damping estimate \eqref{1.10}, we first use $U=u_{a}+u$ to obtain
\begin{align}\label{esti u}
&\|(1+ \nu^{\frac13}t\mathcal{M}_1)^{\delta}\mathcal{M}_2U\|_{L^2L^2}\notag\\
\leq &\|(1+ \nu^{\frac13}   t\mathcal{M}_1)^{\delta}\mathcal{M}_2u_{a}\|_{L^2L^2}+\|(1+ \nu^{\frac13} t\mathcal{M}_1)^{\delta}\mathcal{M}_2u\|_{L^2L^2}.
\end{align}
Thanks to Plancherel's formula, $\mathcal{M}_1(k)\leq 3|k|^{\frac23}$, \eqref{k>1, psiL2}--\eqref{xi>1 xi4pay+xitpsiL2}, \eqref{esti4 psiL2}--\eqref{xi<1 pay+xitpsiL2}, \eqref{a3.26}, we get
\begin{align}\label{uLneq}
&\|(1+ \nu^{\frac13}t\mathcal{M}_1)^{\delta}\mathcal{M}_2u_{a}\|^2_{L^2L^2}\notag\\
 \lesssim &\int^{\infty}_0\int_{|k|\geq 1}(1+ \nu^{\frac13}|k|^{\frac23} t)^{2\delta}|k|^2\big((|k|+|k|t)\|\phi_{a}\|_{L^2_y}+\|(\partial_y+ikt)\phi_{a}\|_{L^2_y}\big)^2 \dd k\dd t\notag\\
 &+\int^{\infty}_0\int_{|k|\leq 1}(1+ \nu^{\frac13} |k|^{\frac23} t)^{2\delta}|k|((|k|+|k|t)\|\phi_{a}\|_{L^2_y}+\|(\partial_y+ikt)\phi_{a}\|_{L^2_y})^2 \dd k\dd t\notag\\
\lesssim &\mathcal{E}_{in}^2\int^{\infty}_0 (1+t)^{-\frac32} \dd t\lesssim \mathcal{E}_{in}^2,
\end{align}
where we used
\begin{align*}
\int_{|k|\leq 1}\frac{|k|}{(1+|kt|)^2}\dd k\lesssim (1+t)^{-1}\int_{|k|\leq 1}\frac{1}{1+|kt|}\dd k\lesssim (1+t)^{-2}\ln(1+t)\lesssim(1+t)^{-\frac32}.
\end{align*}
Moreover, it follows from the Poincaré inequality and Proposition \ref{ap3.2} that
\begin{equation}\notag
\begin{aligned}
&\|(1+ \nu^{\frac13}t\mathcal{M}_1)^{\delta}\mathcal{M}_2u\|_{L^2L^2}\\
\leq&\|(1+ \nu^{\frac13}t\mathcal{M}_1)^{\delta}\mathcal{M}_2u\|_{L^2(0,\nu^{-\frac16})L^2}+\|(1+ \nu^{\frac13}t\mathcal{M}_1)^{\delta}\mathcal{M}_2u\|_{L^2(\nu^{-\frac16},\infty)L^2}\\
\lesssim & \nu^{-\frac{1}{12}}\|(1+ \nu^{\frac13}t\mathcal{M}_1)^{\delta}\omega\|_{L^\infty(0,\nu^{-\frac16}) L^2}+ \mathcal{E}_{in}\lesssim \mathcal{E}_{in},
\end{aligned}
\end{equation}
which together with \eqref{esti u}--\eqref{uLneq} implies \eqref{1.10}. The proof of Theorem~\ref{thm1.2} is completed.
\end{proof}

\section{Proof of Propositions \ref{p2.1} and \ref{linear}}\label{s3}

In this section, we are committed to the proof of Propositions \ref{p2.1} and \ref{linear}. 

For Proposition \ref{p2.1},  taking the Fourier transform of \eqref{2.7} with respect to $x$ yields
\begin{equation}\label{3.2}
\left\{
\begin{aligned}
&\partial_t \theta_{a,k} + \nu t^2 k^2 \theta_{a,k} + ik V_1 \theta_{a,k} = 0,\\
&\partial_t \omega_{a,k} + \nu t^2 k^2 \omega_{a,k} + ik V_1 \omega_{a,k} = -ik\theta_{a,k},\ \  k \in \mathbb{Z},\ y\in I, \\
&\omega_{a,k}|_{t=0} = \Omega_{in,k},\ \theta_{a,k}|_{t=0} = \Theta_{in,k},\; k \neq 0; \quad \omega_{a,0}=\theta_{a,0} = 0.
\end{aligned}
\right.
\end{equation}

We present the following lemmas, which will be used to prove the proposition.

\begin{lem}\label{l3.1}
Suppose that $(\omega_{a,k},\theta_{a,k})$ solves system \eqref{3.2} and $V_1$ satisfies \eqref{2.2}--\eqref{2.3}. Then for $k \in \mathbb{Z}$ and $ k\neq 0$, $t \geq 0$, $j = 0,1,2,3$, it holds that
\begin{align}
&\|(k, \partial_y)^j (\e^{ikyt} \theta_{a,k})(t)\|_{L^2_y} \lesssim |k|^{j-3} \e^{-2\nu^{\frac{1}{3}} |k|^{\frac{2}{3}} t} E_{2}(k),\label{3.3}\\
&\|(k, \partial_y)^j (\e^{ikyt} \omega_{a,k})(t)\|_{L^2_y}\lesssim |k|^{j-3}\e^{-2\nu^{\frac{1}{3}} |k|^{\frac{2}{3}} t} \left(E_{1}(k)+ \nu^{-\frac{1}{3}}E_{2}(k)\right).\label{3.4}
\end{align}
For $k \in \mathbb{Z}$ and $ k\neq 0$, $t \geq 0$, $j = 0,1,2$, there holds
\begin{align}
&\|(k, \partial_y)^j (\e^{ikyt} \theta_{a,k})(t)\|_{L^\infty_y} \lesssim |k|^{j-\frac{5}{2}} \e^{-2\nu^{\frac{1}{3}} |k|^{\frac{2}{3}} t} E_{2}(k),\label{3.5}\\
&\|(k, \partial_y)^j (\e^{ikyt} \omega_{a,k})(t)\|_{L^\infty_y} \lesssim |k|^{j - \frac{5}{2}} \e^{-2\nu^{\frac{1}{3}} |k|^{\frac{2}{3}} t}\left(E_{1}(k)+ \nu^{-\frac{1}{3}}E_{2}(k)\right).\label{3.6}
\end{align}
\end{lem}
\begin{proof}
By direct calculation, $\theta_{a,k}$ is given by
\begin{equation}\label{3.7}
\theta_{a,k}(t,y) = \Theta_{in,k}(y) \e^{-ikV_1^{\star}(t,y) - \nu k^2 t^3/3}, 
\end{equation}
where $ V_1^{\star}(t,y) = \int_0^t V_1(\tau,y)\,\dd \tau$.
By applying \eqref{2.2}, we have
\begin{equation}\label{3.8}
\|V_1^{\star} - ty\|_{C^3} \lesssim t \|\Omega_{in}\|_{H^3} \leq \nu^{\frac{1}{3}} t.
\end{equation}
With the expression of $\theta_{a,k}$, by solving non-homogeneous ordinary differential equation \eqref{3.2}$_2$, we can get
\begin{align}\label{3.9}
\e^{ikyt} \omega_{a,k}(t,y) =& \exp\!\left( -\frac{\nu k^2 t^3}{3} - ik (V_1^{\star}(t,y) - yt)\right) \notag\\
&\cdot\left(\Omega_{in,k}(y) - ik\int_0^t \exp\!\left( \frac{\nu k^2s^3}{3}+ik(V_1^{\star}(s,y) - ys)\right) \e^{ikys} \theta_{a,k}(s,y)\dd \tau\right) \notag\\
=&\exp\!\left( -\frac{\nu k^2 t^3}{3} - ik (V_1^{\star}(t,y) - yt)\right)\left(\Omega_{in,k}(y) - ikt\Theta_{in,k}(y)\right).
\end{align}
Let $h(t) = -k(V_1^{\star}(t,y) - ty)$. Thanks to \eqref{3.8}, we deduce $|\partial_y^j h| \lesssim \nu^{\frac{1}{3}} |kt|$ and $|\partial_y^j(\e^{ih})| \lesssim(1+\nu^{\frac{1}{3}} |kt|)^j$ for $j = 0,1,2,3$, then it holds that
\begin{align*}
|\partial_y^j(\e^{ih - \nu k^2 t^3/3})| &\lesssim (1 + \nu^{\frac{1}{3}} |kt|)^j \e^{-\nu k^2 t^3/3} \lesssim |k|^{\frac{j}{3}}(1 + \nu k^2 t^3)^{\frac{j}{3}} \e^{-\nu k^2 t^3/3} \\
&\lesssim |k|^{\frac{j}{3}} \e^{-3\nu^{\frac{1}{3}} |k|^{\frac{2}{3}} t},
\end{align*}
\begin{equation}\label{3.10}
   |(k, \partial_y)^j(\e^{ih - \nu k^2 t^3/3})| \lesssim |k|^j \e^{-3\nu^{\frac{1}{3}} |k|^{\frac{2}{3}} t}. 
\end{equation}
Thus, for $j = 0,1,2,3$, we deduce that
\begin{align}
\|(k, \partial_y)^j (\e^{ikyt} \theta_{a,k})\|_{L^2_y} &= \|(k, \partial_y)^j \big(\Theta_{in,k} \e^{ih - \nu k^2 t^3/3}\big)\|_{L^2_y} \notag\\
&\lesssim \sum_{l=0}^j \|(k, \partial_y)^l (\Theta_{in,k}(y))\|_{L^2_y} \|(k, \partial_y)^{j-l} (\e^{ih - \nu k^2 t^3/3})\|_{L^\infty_y} \notag\\
&\lesssim \sum_{l=0}^j |k|^{l-3}E_{2}(k)  |k|^{j-l} \e^{-3\nu^{\frac{1}{3}} |k|^{\frac{2}{3}} t} \notag\\
&\lesssim |k|^{j-3} \e^{-2\nu^{\frac{1}{3}} |k|^{\frac{2}{3}} t}E_{2}(k).\notag
\end{align}
In a similar way, we derive
\begin{align*}
\|(k, \partial_y)^j (\e^{ikyt} \omega_{a,k})\|_{L^2_y} 
\lesssim& |k|^{j-3}\e^{-3\nu^{\frac{1}{3}} |k|^{\frac{2}{3}} t}(E_{1}(k)+tE_{2}(k))\\
\lesssim& |k|^{j-3}\e^{-2\nu^{\frac{1}{3}} |k|^{\frac{2}{3}} t}(E_{1}(k)+ \nu^{-\frac{1}{3}}E_{2}(k)).
\end{align*}
For $j = 0,1,2$, there holds
\begin{align*}
\|(k, \partial_y)^j (\e^{ikyt} \theta_{a,k})\|_{L^\infty_y} &\lesssim \|(k, \partial_y)^j (\e^{ikyt} \theta_{a,k})\|_{L^2_y}^{\frac{1}{2}}\|(k, \partial_y)^{j+1} (\e^{ikyt} \theta_{a,k})\|_{L^2_y}^{\frac{1}{2}}  \\
&\lesssim |k|^{j - \frac{5}{2}} \e^{-2\nu^{\frac{1}{3}} |k|^{\frac{2}{3}} t}E_{2}(k),
\end{align*}
\begin{align*}
\|(k, \partial_y)^j (\e^{ikyt} \omega_{a,k})\|_{L^\infty_y} &\lesssim \|(k, \partial_y)^j (\e^{ikyt} \omega_{a,k})\|_{L^2_y}^{\frac{1}{2}}\|(k, \partial_y)^{j+1} (\e^{ikyt} \omega_{a,k})\|_{L^2_y}^{\frac{1}{2}}  \\
&\lesssim |k|^{j - \frac{5}{2}}\e^{-2\nu^{\frac{1}{3}} |k|^{\frac{2}{3}} t} \left(E_{1}(k)+ \nu^{-\frac{1}{3}}E_{2}(k)\right).
\end{align*}
The proof of Lemma \ref{l3.1} is completed.
\end{proof}
\begin{lem}\label{l3.2}
  Suppose that $(\omega_{a},\theta_{a})$ solves system \eqref{2.7} and $V_1$ satisfies \eqref{2.2}--\eqref{2.3}, then we have
   \begin{align}
      &t \|\partial_x \theta_{a}(t)\|_{L^\infty_y}+t \|\partial_x^2 \theta_{a}(t)\|_{L^\infty_y}  \lesssim \nu^{-\frac{1}{3}} \e^{-\nu^{\frac{1}{3}} t} \|\Theta_{in}\|_{H^4},\notag \\
      &t \|\partial_x \omega_{a}(t)\|_{L^\infty_y}+t \|\partial_x^2 \omega_{a}(t)\|_{L^\infty_y}  \lesssim \nu^{-\frac{1}{3}} \e^{-\nu^{\frac{1}{3}} t} (\|\Omega_{in}\|_{H^3}+\nu^{-\frac{1}{3}} \|\Theta_{in}\|_{H^4}).\notag
   \end{align}
\end{lem}
\begin{proof}
The proof of this lemma is similar to Lemma 3.3 in \cite{wei-zhang2026}, and is therefore omitted.
\end{proof}
\begin{lem}[Lemmas 3.4--3.6 in \cite{wei-zhang2026}]\label{l3.3}
If $k \in \mathbb{Z}$ and $k \neq 0$, $\omega \in H^3(-1,1)$, $\phi = -\Delta_k^{-1} \omega$, $\omega(\pm1)=0$. Then it holds for $t \geq 0$ that
\[
k^2 \|(k, \partial_y) \phi(t)\|_{L^2_y} \lesssim (1+t)^{-1} \|(k, \partial_y)(\e^{ikyt} \omega)(t)\|_{L^2_y},
\]
\[
|k|^4 \|\phi(t)\|_{L^2_y} \lesssim (1+t)^{-2} \|(k, \partial_y)^2 (\e^{ikyt} \omega)(t)\|_{L^2_y},
\]
\[
|k|^2 \|(k, \partial_y)(\partial_y \phi + ikt \phi)(t)\|_{L^2_y} \lesssim (1+t)^{-1} \|(k, \partial_y)^2 (\e^{ikyt} \omega)(t)\|_{L^2_y},
\]
\[
|k|^4 \|(\partial_y \phi + ikt \phi)(t)\|_{L^2_y} \lesssim (1+t)^{-2} \|(k, \partial_y)^3 (\e^{ikyt} \omega)(t)\|_{L^2_y}.
\]
\end{lem}

With Lemmas \ref{l3.1}--\ref{l3.3} in hand, we are now ready to prove Proposition \ref{p2.1}.
\begin{proof}[\textbf{Proof of Proposition \ref{p2.1}}]
For $(\theta_{a},\omega_{a})$,
due to the fact that
\[
\theta_{a}(t,x,y) = \sum_{k \neq 0} \theta_{a,k}(t,y) \e^{ikx}, \quad
\partial_x \theta_{a}(t,x,y) = \sum_{k \neq 0} ik \theta_{a,k}(t,y) \e^{ikx},
\]
\[
(\partial_y + t\partial_x)\theta_{a}(t,x,y) = \sum_{k \neq 0} (\partial_y + ikt) \theta_{a,k}(t,y) \e^{ikx}
= \sum_{k \neq 0} \partial_y(\e^{ikyt} \theta_{a,k})(t,y) \e^{ik(x - yt)}.
\]
Then it follows from \eqref{aaa3.1} and Lemma~\ref{l3.1} that
\begin{align}\label{3.12}
&\|\partial_x \theta_{a}(t)\|_{L^\infty} + \|(\partial_y + t\partial_x)\theta_{a}(t)\|_{L^\infty}\notag\\
\lesssim&  \sum_{k \neq 0} \|(k, \partial_y)(\e^{ikyt} \theta_{a,k})(t)\|_{L^\infty_y}
\lesssim \sum_{k \neq 0} |k|^{-\frac{3}{2}} E_{2}(k) \e^{-2\nu^{\frac{1}{3}} |k|^{\frac{2}{3}} t}\notag\\
\lesssim&  \left( \sum_{k \neq 0 } |k|^{-3}\right)^{\frac{1}{2}}\left( \sum_{k \neq 0 } (E_{2}(k))^2 \right)^{\frac{1}{2}} \e^{-2\nu^{\frac{1}{3}} t}
\lesssim \e^{-2\nu^{\frac{1}{3}} t} \|\Theta_{in}\|_{H^4}.
\end{align}
Similarly, we can also prove that
\begin{align}\label{3.13}
&\|\partial_x \omega_{a}(t)\|_{L^\infty} + \|(\partial_y + t\partial_x)\omega_{a}(t)\|_{L^\infty}\notag\\
\lesssim&  \sum_{k \neq 0 } \|(k, \partial_y)(\e^{ikyt} \omega_{a,k})(t)\|_{L^\infty_y}
\lesssim \sum_{k \neq 0} |k|^{-\frac{3}{2}} \e^{-2\nu^{\frac{1}{3}} |k|^{\frac{2}{3}} t}(E_{1}(k)+\nu^{-\frac{1}{3}}E_{2}(k))\notag\\
\lesssim&{\rm e}^{-2\nu^{\frac{1}{3}} t} (\|\Omega_{in}\|_{H^3}+\nu^{-\frac{1}{3}}\|\Theta_{in}\|_{H^4}).
\end{align}
Thus, we obtain \eqref{2.10}--\eqref{2.11}.

 For \eqref{2.12}--\eqref{2.13}, the proof is similar to that of (2.6)--(2.7) in \cite{wei-zhang2026}, here we omit the details.

Finally,
note that
\begin{align}
&Er_{\theta}^2 = \partial_t \theta_{a} - \nu \Delta \theta_{a} + V_1 \partial_x \theta_{a}, \quad Er^2 = Er_{\theta}^2 -V_2' \partial_x \phi_{a} + u_{a} \cdot \nabla \theta_{a},\notag\\ 
&Er_{\omega}^1 = \partial_t \omega_{a}  + V_1 \partial_x \omega_{a}- \nu \Delta \omega_{a}+\partial_x\theta_{a}, \quad Er^1 = Er_{\omega}^1 - V_1'' \partial_x \phi_{a} + u_{a} \cdot \nabla \omega_{a}.\notag
\end{align}
From $\partial_x \phi_{a} = -u_{a}^2$, \eqref{2.4} and \eqref{2.12}, we derive
\[
\|V_2'\partial_x \phi_{a}(t)\| \leq \|V_2'\|_{L^\infty} \|u_{a}^2(t)\| \lesssim(1+t)^{-2} \e^{-2\nu^{\frac{1}{3}} t}\|\Theta_{in}\|_{H^3} (\|\Omega_{in}\|_{H^3}+\nu^{-\frac{1}{3}}\|\Theta_{in}\|_{H^4}).
\]
Moreover, 
\[
u_{a} \cdot \nabla \theta_{a} = u_{a}^1 \partial_x \theta_{a} + u_{a}^2 \partial_y \theta_{a} = (u_{a}^1 - t u_{a}^2) \partial_x \theta_{a} + u_{a}^2 (\partial_y + t \partial_x) \theta_{a}.
\]
Then it follows from \eqref{2.10} and \eqref{2.12} that
\begin{align*}
\|u_{a} \cdot \nabla \theta_{a}(t)\| &\leq \|(u_{a}^1 - t u_{a}^2)(t)\| \|\partial_x \theta_{a}(t)\|_{L^\infty} + \|u_{a}^2(t)\| \|(\partial_y + t \partial_x)\theta_{a}(t)\|_{L^\infty} \\
&\lesssim(1+t)^{-2} \e^{-2\nu^{\frac{1}{3}} t} \|\Theta_{in}\|_{H^4}(\|\Omega_{in}\|_{H^3}+\nu^{-\frac{1}{3}}\|\Theta_{in}\|_{H^4}).
\end{align*}
Notice that $(\partial_t + V_1 \partial_x- \nu t^2 \partial_x^2 )\theta_{a} = 0$ and $Er^2_{\theta} = (\partial_t + V_1 \partial_x- \nu \Delta)\theta_{a}$, we obtain
\[
Er^2_{\theta} = (\nu t^2 \partial_x^2 - \nu \Delta)\theta_{a}= \sum_{k \neq 0} \bar\Theta_k(t,y) \e^{ikx},
\]
here
\[
\bar\Theta_k(t,y) = -\nu (\partial_y^2 - k^2 + k^2 t^2) \theta_{a,k}(t,y).
\]
Thus, we arrive at
\[
\e^{ikyt} \bar\Theta_k = -\nu \left((\partial_y - ikt)^2 - k^2 + k^2 t^2\right)(\e^{ikyt} \theta_{a,k})
= -\nu (\partial_y^2 - 2ikt \partial_y - k^2)(\e^{ikyt} \theta_{a,k}).
\]
We use Lemma~\ref{l3.1} and \eqref{aaa3.1} to obtain 
\begin{align*}
\|\bar\Theta_k(t)\|_{L^2_y} &=\|\e^{ikyt}\bar\Theta_k(t)\|_{L^2_y}= \nu\| (\partial_y^2 - 2ikt \partial_y - k^2)(\e^{ikyt} \theta_{a,k})(t)\|_{L^2_y} \\
&\lesssim \nu (1+t) \|(k, \partial_y)^2 (\e^{ikyt} \theta_{a,k})(t)\|_{L^2_y} \\
&\lesssim \nu (1+t) |k|^{-1} \e^{-2\nu^{\frac{1}{3}} |k|^{\frac{2}{3}} t} E_{2}(k) \lesssim \nu (1+t) \e^{-2\nu^{\frac{1}{3}} t} E_{2}(k),
\end{align*}
and then
\begin{align*}
\|Er^2_{\theta}(t)\|^2 &\lesssim \sum_{k \neq 0} \|\bar\Theta_k(t)\|_{L^2_y}^2\lesssim \sum_{k \neq 0} \nu^2 (1+t)^2 \e^{-4\nu^{\frac{1}{3}} t} (E_{2}(k))^2 \\
&\lesssim \nu^2 (1+t)^2 \e^{-4\nu^{\frac{1}{3}} t} \|\Theta_{in}\|_{H^4}^2.
\end{align*}
Summing up, there holds
\begin{align*}
\|Er^2(t)\| &\leq \|Er^2_{\theta}(t)\| + \|V_2' \partial_x \phi_{a}(t)\| + \|u_{a} \cdot \nabla \theta_{a}(t)\| \\
&\lesssim \nu (1+t) \e^{-2\nu^{\frac{1}{3}} t}\|\Theta_{in}\|_{H^4} +(1+t)^{-2} \e^{-2\nu^{\frac{1}{3}} t}\|\Theta_{in}\|_{H^4}(\|\Omega_{in}\|_{H^3}+\nu^{-\frac{1}{3}}\|\Theta_{in}\|_{H^4}).
\end{align*}
Similarly, we can also prove \eqref{2.14}, and then finish the proof of Proposition \ref{p2.1}.
\end{proof}

We now turn to the proof of Proposition \ref{linear}.
Firstly, by taking the Fourier transform of \eqref{a3.3} with respect to $x$, we obtain
\begin{equation}\label{a3.24}
\left\{
\begin{aligned}
&\partial_t \theta_{a,k} + ik y \theta_{a,k}+ \nu (1+t^2 k^2) \theta_{a,k}  = 0,\\
&\partial_t \omega_{a,k} + iky \omega_{a,k}+ \nu (1+t^2 k^2) \omega_{a,k}  = -ik\theta_{a,k},\ \  k \neq 0,\ y\in I, \\
&\omega_{a,k}|_{t=0} = \Omega_{in,k},\ \theta_{a,k}|_{t=0} = \Theta_{in,k}.
\end{aligned}
\right.
\end{equation}

\begin{lem}\label{al3.3}
For $k \in \mathbb{R} \setminus \{0\}$, $t \geq 0$, $j = 0,1,2,3$, it holds that
\begin{align}
&\|(k, \partial_y)^j (\e^{ikyt} \theta_{a,k})(t)\|_{L^2_y} \lesssim \e^{-\nu t- \nu k^2 t^3/4}E_{2}(k),\label{a3.25}\\
&\|(k, \partial_y)^j (\e^{ikyt} \omega_{a,k})(t)\|_{L^2_y}\lesssim \e^{-\nu t- \nu k^2 t^3/4}\left(E_{1}(k)+\nu^{-\frac{1}{3}}E_{2}(k)\right).\label{a3.26}
\end{align}
For $k \in \mathbb{R} \setminus \{0\}$, $t \geq 0$, $j = 0,1,2$, it holds that
\begin{align}
&\|(k, \partial_y)^j (\e^{ikyt} \theta_{a,k})(t)\|_{L^\infty_y} \lesssim \e^{-\nu t- \nu k^2 t^3/4}E_{2}(k),\label{a3.27}\\
&\|(k, \partial_y)^j (\e^{ikyt} \omega_{a,k})(t)\|_{L^\infty_y} \lesssim \e^{-\nu t- \nu k^2 t^3/4}\left(E_{1}(k)+\nu^{-\frac{1}{3}}E_{2}(k)\right).\label{a3.28}
\end{align}
\end{lem}
\begin{proof}
According to \eqref{a3.24}$_1$, we know that 
\begin{equation}\label{a3.29}
\theta_{a,k}(t,y) = \Theta_{in,k}(y) \e^{-\nu t-ikyt - \nu k^2 t^3/3}. 
\end{equation}
Then, we can deduce 
\begin{align}\label{a3.30}
\e^{ikyt} \omega_{a,k}(t,y) =& \exp\!\left( -\nu t-\frac{\nu k^2 t^3}{3} \right)\left(\Omega_{in,k}(y) - ik\int_0^t \exp\!\left(\nu s+ \frac{\nu k^2s^3}{3}\right) \e^{ikys} \theta_{a,k}(s,y)\dd s\right) \notag\\
=&\exp\!\left(-\nu t -\frac{\nu k^2 t^3}{3} \right)\left(\Omega_{in,k}(y) - ikt\Theta_{in,k}(y)\right).
\end{align}
Thus, similar to Lemma \ref{l3.1}, we can obtain \eqref{a3.25}--\eqref{a3.28}. 
 This completes the proof of this lemma.
\end{proof}

\begin{lem}\label{al3.4}
Assume that $k \in \mathbb{R} \setminus \{0\}$, $\omega \in H^3(-1,1)$, $\phi = -\Delta_k^{-1} \omega$, $\omega(\pm1)=0$. Then it holds that for $|k|\geq 1$,
\begin{align}
|k|^2\|(k, \partial_y)\phi\|_{L^2_y}\lesssim& (1+t)^{-1}\|(k, \partial_y) ({\rm e}^{ik yt} \omega)\|_{L^2_y},, \label{psiL2k>1}\\
|k|^2 \|(k, \partial_y)(\partial_y \phi+ik t \phi)\|_{L^2_y}\lesssim&(1+t)^{-1}\|(k, \partial_y)^2 ({\rm e}^{ik yt} \omega)\|_{L^2_y}, \label{paypay+kpsiL2}\\
|k|^4\|\phi\|_{L^2_y}\lesssim&(1+t)^{-2}\|(k, \partial_y)^2 ({\rm e}^{ik yt} \omega)\|_{L^2_y},\label{k>1, psiL2}\\
| k|^4\|(\partial_y+i k t)\phi\|_{L^2_y}\lesssim&(1+t)^{-2}\|(k, \partial_y)^3 ({\rm e}^{ik yt} \omega)\|_{L^2_y},\label{xi>1 xi4pay+xitpsiL2}
\end{align}
and for $|k|< 1$,
\begin{align}
\|(1, \partial_y)\phi\|_{L^2_y}\lesssim& (1+|kt|)^{-1}\|(1, \partial_y) ({\rm e}^{ik yt} \omega)\|_{L^2_y},  \label{psiL2k<1}\\
\|(1, \partial_y)(\partial_y\phi+ik t\phi)\|_{L^2_y}\lesssim&(1+|kt|)^{-1}\|(1, \partial_y)^2 ({\rm e}^{ik yt} \omega)\|_{L^2_y},\label{esti3 paypaypsi+xiytpsi}\\
\|\phi\|_{L^2_y}\lesssim&(1+|kt|)^{-2}\|(1, \partial_y)^2 ({\rm e}^{ik yt} \omega)\|_{L^2_y}, \label{esti4 psiL2}\\
\|(\partial_y+i k t)\phi\|_{L^2_y}\lesssim&(1+| k t|)^{-2}\|(1, \partial_y)^3 ({\rm e}^{ik yt} \omega)\|_{L^2_y}. \label{xi<1 pay+xitpsiL2}
\end{align}
\end{lem}
\begin{proof}
For details of the proof, see Lemma A.4 in  \cite{miao2025}, we omit the proof here.
\end{proof}

With the above preparations, we now proceed to prove Proposition \ref{linear}. 
\begin{proof}[\textbf{Proof of Proposition \ref{linear}}]

We only prove \eqref{u1LLinfty}--\eqref{u12L22}, and the proofs of other formulas are similar to those in \cite{miao2025}. 
 
Define $(\widetilde{\theta},\widetilde{\om})(t,x,y)=(\theta_{a},\omega_{a})(t,x+ty,y)$. We have from \eqref{a3.24} that
\begin{equation}\notag
\left\{
\begin{aligned}
&\partial_t \tilde{\theta} + \nu (1-t^2 \partial_x^2) \tilde{\theta}  = 0,\\
&\partial_t \tilde{\om} + \nu (1-t^2 \partial_x^2) \tilde{\om}  = -\partial_x\tilde{\theta},\ \  x\in \mathbb{R},\ y\in I, \\
&\omega_{a}|_{t=0} = \Omega_{in},\ \theta_{a}|_{t=0} = \Theta_{in},\\ 
&\tilde{\om}(t,x\mp t,\pm1)=\tilde{\theta}(t,x\mp t,\pm1)=0.
\end{aligned}
\right.
\end{equation}
Then we can solve the above system as 
\begin{equation}\notag
\tilde{\theta} = \e^{-\nu t} \e^{\frac{\nu t^3}{3}\partial_x^2} \Theta_{in}, \ \ \ \tilde{\omega} = \e^{-\nu t} \e^{\frac{\nu t^3}{3}\partial_x^2} \Omega_{in} - \int_0^t \e^{-\nu(t-s)} \partial_x \e^{A(t,s)\partial_x^2}  \tilde{\theta}(s)\dd s,
\end{equation}
where $A(t, s) = \nu \int_s^t \tau^2 \dd\tau = \frac{\nu}{3} (t^3 - s^3).$ Using $\mathcal{M}_1 \leq 3|k|^{\frac{2}{3}}$ and the heat kernel estimate, we obtain
\begin{align}\label{omLxy}
\| (1 + \nu^{\frac{1}{3}} t \mathcal{M}_1)^{m} \tilde{\theta} \|_{L_x^1 L_y^\infty} 
\lesssim_m \e^{-\nu t} \| \Theta_{in} \|_{L_x^1 L_y^\infty},
\end{align}
and then by $1 + \nu^{\frac{1}{3}} t \mathcal{M}_1 \leq (1 + \nu^{\frac{1}{3}} s \mathcal{M}_1)(1 + \nu^{\frac{1}{3}} (t-s) \mathcal{M}_1)$ and $\nu^{\frac{1}{3}} (t-s)\leq (\nu(t^3-s^3))^{\frac{1}{3}}$,
\begin{align}\label{omLxy1}
\| (1 + \nu^{\frac{1}{3}} t \mathcal{M}_1)^{m} \tilde{\omega} \|_{L_x^1 L_y^\infty} 
\lesssim_m& \e^{-\nu t} \| \Omega_{in} \|_{L_x^1 L_y^\infty}+ \int_0^t \e^{-\nu(t-s)} A(t,s)^{-\frac{1}{2}} \|(1 + \nu^{\frac{1}{3}} s \mathcal{M}_1)^{m}\tilde{\theta}(s)\|_{L_x^1 L_y^\infty} \dd s\notag\\
\lesssim_m&\e^{-\nu t} \| \Omega_{in} \|_{L_x^1 L_y^\infty}+ \nu^{-\frac{1}{2}}\e^{-\nu t}\| \Theta_{in} \|_{L_x^1 L_y^\infty}\int_0^t (t^3 - s^3)^{-\frac{1}{2}}\dd s\notag\\
\lesssim_m&\e^{-\nu t} \| \Omega_{in} \|_{L_x^1 L_y^\infty}+ \nu^{-\frac{1}{2}}t^{-\frac{1}{2}}\e^{-\nu t}\| \Theta_{in} \|_{L_x^1 L_y^\infty}.
\end{align}
Define $(1+ \nu^{\frac13}t\mathcal{M}_1)^{m}\phi_{a,1}$ be the solution to $\partial^2_y (1+\nu^{\frac13}t\mathcal{M}_1)^{m} \phi_{a,1}=(1+ \nu^{\frac13}t\mathcal{M}_1)^{m}\omega_{a}$ with $\phi_{a,1}(\pm1)=0.$ 
By \eqref{omLxy}--\eqref{omLxy1}, we can deduce that 
\begin{align}\label{a3.43}
&\|\partial_y(1+ \nu^{\frac13}t\mathcal{M}_1)^{m}\phi_{a,1}\|_{L^\infty_y}\leq \|(1+ \nu^{\frac13}t\mathcal{M}_1)^{m}\omega_{a}\|_{L^1_y}\notag\\ 
&= \int^{1}_{-1}|(1+ \nu^{\frac13}t\mathcal{M}_1)^{m}(\widetilde{\om}(t,x-ty,y))|{\rm d}y\lesssim_mt^{-1}\int_{\R}|(1+\nu^{\frac13}t\mathcal{M}_1)^{m}\widetilde{\om}(t,s,y)|{\rm d}s\notag\\
 &\lesssim_m t^{-1}\e^{-\nu t}(\| \Omega_{in} \|_{L_x^1 L_y^\infty}+ \nu^{-\frac{1}{2}}t^{-\frac{1}{2}}\e^{-\nu t}\| \Theta_{in} \|_{L_x^1 L_y^\infty}).
\end{align}
Moreover, it follows from $\Delta \phi_{a}=\omega_{a}$ that
\begin{align*}
\partial^2_y (1+ \nu^{\frac13}t\mathcal{M}_1)^{m}(\phi_{a,1}-\phi_{a})=(1+ \nu^{\frac13}t\mathcal{M}_1)^{m}\partial^2_x\phi_{a}.
\end{align*}
We then apply \eqref{k>1, psiL2}, \eqref{esti4 psiL2}, \eqref{a3.26} and $\mathcal{M}_1(k)\leq 3|k|^{\frac23}$ to obtain
\begin{equation*}
\begin{aligned}
&\|\partial_y(1+ \nu^{\frac13}t\mathcal{M}_1)^{m}(\phi_{a,1}-\phi_{a})\|_{L^\infty_y}
\leq  \|(1+ \nu^{\frac13}t\mathcal{M}_1)^{m}\partial^2_x\phi_{a}\|_{L^1_y}\\
\lesssim &\int^1_{-1}\Big|\int_{\R} (1+ \nu^{\frac13}|k|^{\frac23} t)^{m}{\rm e}^{ikx} k^2\phi_{a,k} {\rm d}k\Big|{\rm d}y\lesssim\int_{\R} (1+ \nu^{\frac13}|k|^{\frac23} t)^{m}|k|^2\|\phi_{a,k}\|_{L^2_y}{\rm d}k\\
\lesssim_m&\int_{|k|\leq 1} \frac{|k|^2{\rm e}^{-\nu t}}{(1+kt)^2}(E_1+\nu^{-\frac13}E_2){\rm d}k+\int_{|k|\geq 1}(1+t)^{-2}{\rm e}^{-\nu t}|k|^{-2}(E_1+\nu^{-\frac13}E_2){\rm d}k\\
\lesssim_m&(1+t)^{-2}{\rm e}^{-\nu t}(\|E_1(k)\|_{L^2_k}+\nu^{-\frac13}\|E_2(k)\|_{L^2_k})\lesssim_m(1+t)^{-2}{\rm e}^{-\nu t}\mathcal{E}_{in},
\end{aligned}
\end{equation*}
which together with \eqref{a3.43} implies
\begin{align}\label{payphiLinfty}
\|(1+ \nu^{\frac13}t\mathcal{M}_1)^{m}\partial_y \phi_{a}\|_{L^\infty_{x,y}}\lesssim_m&t^{-1}\e^{-\nu t}(\mathcal{E}_{in}+\| \Omega_{in} \|_{L_x^1 L_y^\infty}+ \nu^{-\frac{1}{2}}t^{-\frac{1}{2}}\| \Theta_{in} \|_{L_x^1 L_y^\infty})\notag\\ 
\lesssim_m&t^{-1}\e^{-\nu t}(\mathcal{E}_{in}+t^{-\frac{1}{2}}\mathcal{E}_{in}).
\end{align}
In addition, using \eqref{psiL2k>1}, \eqref{paypay+kpsiL2}, \eqref{psiL2k<1}, \eqref{esti3 paypaypsi+xiytpsi} and \eqref{a3.26}, we deduce
\begin{align*}
\|(1+ \nu^{\frac13}t\mathcal{M}_1)^{m}\partial_y\phi_{a}\|_{L^\infty_{x,y}}\lesssim & \int_{\R} (1+ \nu^{\frac13}|k|^{\frac23} t)^{m}\|{\rm e}^{ikty}\partial_y\phi_{a,k}\|_{L^\infty_y}{\rm d}k\\
\lesssim &  \int_{\R} (1+ \nu^{\frac13}|k|^{\frac23} t)^{m}\|\partial_y(\partial_y+ikt)\phi_{a,k}\|^{\frac12}_{L^2_y}\|\partial_y\phi_{a,k}\|^{\frac12}_{L^2_y}{\rm d}k\\
\lesssim_m&\int_{|k|\leq 1}{\rm e}^{-\nu t} (E_1+\nu^{-\frac13}E_2) {\rm d}k+\int_{|k|\geq 1} {\rm e}^{-\nu t}|k|^{-2}(E_1+\nu^{-\frac13}E_2) {\rm d}k\\ 
\lesssim_m&{\rm e}^{-\nu t}(\|E_1(k)\|_{L^2_k}+\nu^{-\frac13}\|E_2(k)\|_{L^2_k})\lesssim_m {\rm e}^{-\nu t}\mathcal{E}_{in},
\end{align*}
which together with \eqref{payphiLinfty} leads to
\begin{align}\label{payphi}
\|(1+ \nu^{\frac13}t\mathcal{M}_1)^{m}\partial_y \phi_{a}\|_{L^\infty}\lesssim_m(1+t)^{-1}{\rm e}^{-\nu t}\mathcal{E}_{in}.
\end{align}
Similarly, by taking $\partial^2_y (1+ \nu^{\frac13}t\mathcal{M}_1)^{m} (\partial_y+t\partial_x)\phi_{a,1}=(1+ \nu^{\frac13}t\mathcal{M}_1)^{m}(\partial_y+t\partial_x)\omega_{a}$, we can obtain the estimate of $\|(1+ \nu^{\frac13}t\mathcal{M}_1)^{m}(\partial_y+t\partial_x)\phi_{a}\|_{L^\infty}$.
Next, 
\begin{align*}
\|(1+ \nu^{\frac13}t\mathcal{M}_1)^{m}\partial_x\phi_{a}\|_{L^\infty_y}
\lesssim& \int_{\R} (1+ \nu^{\frac13}|k|^{\frac23} t)^{m}|k|\|{\rm e}^{ikty}\phi_{a,k}\|_{L^\infty_y}{\rm d}k\\
\lesssim& \int_{\R} (1+ \nu^{\frac13}|k|^{\frac23} t)^{m}|k|\|(1,\partial_y)({\rm e}^{ikty}\phi_{a,k})\|_{L^2_y}{\rm d}k,
\end{align*}
then applying \eqref{xi>1 xi4pay+xitpsiL2}, \eqref{xi<1 pay+xitpsiL2} and \eqref{a3.26} gives
\begin{align*}
 &\int_{\R} (1+ \nu^{\frac13}|k|^{\frac23} t)^{m}|k|\|(1,\partial_y)({\rm e}^{ikty}\phi_{a,k})\|_{L^2_y}{\rm d}k\\
\lesssim_m&{\rm e}^{-\nu t}\Big(\int_{|k|\leq 1} \frac{|k|}{(1+|kt|)^2}(E_1+\nu^{-\frac13}E_2){\rm d}k+\int_{|k|\geq 1}(1+t)^{-2}|k|^{-3}(E_1+\nu^{-\frac13}E_2){\rm d}k\Big)\\
\lesssim_m&(1+t)^{-2}\ln(1+t){\rm e}^{-\nu t}\mathcal{E}_{in},
\end{align*}
which implies
\begin{align}\notag
\|(1+ \nu^{\frac13}t\mathcal{M}_1)^{m}\partial_x\phi_{a}\|_{L^\infty}\lesssim_m(1+t)^{-2}\ln(1+t){\rm e}^{-\nu t}\mathcal{E}_{in}.
\end{align}
Thus, one gets \eqref{u1LLinfty}--\eqref{u12L22}.

\end{proof}

\section{Proof of Proposition \ref{p2.2} }\label{s4}
In this section, we focus on the proof of Proposition \ref{p2.2}. According to \eqref{2.2} and \eqref{1.4}, we deduce that for $t \geq 0$, 
\begin{equation}\label{eq:4.1}
    \|V_1(t,y) - y\|_{C^3} \leq C\varepsilon_0\ll1.
\end{equation}
Then, inspired by \cite{Villani,wei-zhang2026}, we establish corresponding estimates on the short time scale $0\leq t\leq \nu^{-\frac{1}{6}}$ and the long time scale $t\geq \nu^{-\frac{1}{6}}$, respectively, thereby proving Proposition \ref{p2.2}. 
For this purpose, we first provide some estimates, which will be used later.

\subsection{Space-time estimates for linearized equation}

%\subsubsection{Linear enhanced dissipation.}

\begin{prop}\label{p4.1}
Suppose that $V_1(t,y)$ satisfies \eqref{eq:4.1}, $u_{a}$ is defined as in \eqref{2.7}--\eqref{2.8} and $f$ solves
\begin{equation}\label{4.2}
\left\{
\begin{aligned}
&\partial_t f - \nu \Delta f + V_1 \partial_x f+u_{a}\cdot\nabla f = 0, \ \ \ \  t>0,\\
&f|_{t=0}=f_{in}(x,y),\quad f(t,x,\pm 1) = 0, \quad \nabla\cdot u_{a}=0,
\end{aligned}
\right.
\end{equation}
then for $0<t\leq \nu^{-\frac{1}{6}}$, we have
\begin{align}
\|\partial_x f(t)\| &\les (\nu t)^{-\frac{1}{2}} (1+t)^{-1} \|f_{in}\|, \label{aa1}\\
\|f_{\ne}(t)\| &\les (1 + \nu t^3)^{-\frac{1}{2}} \|f_{in}\|.\label{aa2}
\end{align}
\end{prop}

For clarity, we divide the proof of Proposition \ref{p4.1} into several steps. The first step is about some basic energy estimates of \eqref{4.2}.
\begin{lem}\label{lem:4.1}
Under the hypothesis of the Proposition \ref{p4.1}, we have
\begin{align}
&\frac{\rm d}{{\rm d}t} \|f\|^2+2\nu \|\nabla f\|^2=0, \label{4.4}\\
&\frac{\rm d}{{\rm d}t} \|\nabla f\|^2 + 2\nu \|\Delta f\|^2 \leq 3 \|\partial_xf\| \|\partial_y f\|+C\varepsilon_0\nu^{\frac{1}{3}}\|\nabla f\|^2, \label{4.5}\\
&\frac{\rm d}{{\rm d}t} \langle \partial_xf, \partial_y f \rangle + \|\sqrt{V_1'+\partial_yu_{a}^1} \partial_xf\|^2 \leq 2\nu \|\partial_{xy} f\| \|\Delta f\|+C\varepsilon_0\nu^{\frac{1}{3}}(1+t)^{-2}\|\partial_yf\|^2.\label{4.6}
\end{align}
\end{lem}

\begin{proof}
Firstly, by taking the inner product of both sides of \eqref{4.2}$_1$ with $f$ and using $\nabla \cdot u_{a}=0$, we obtain \eqref{4.4}.

Next, taking the inner product of \eqref{4.2}$_1$ with $\Delta f$ yields
\begin{align}
\frac{\rm d}{{\rm d}t} \|\nabla f\|^2 &= -2 \langle \Delta f, \partial_t f \rangle = -2\nu \|\Delta f\|^2 + 2  \langle \Delta f, (V_1+u_{a}^1) \partial_xf+u_{a}^2\partial_yf \rangle.\notag
\end{align}
It can be directly calculated that
\begin{align}
\langle \Delta f,  V_1 \partial_xf \rangle =&  \langle \partial_{xx} f,  V_1\partial_xf \rangle -  \langle \partial_y f, \partial_y (V_1 \partial_xf) \rangle = - \langle \partial_y f,  V_1'\partial_xf  \rangle\notag\\ 
\leq&\|V_1'\|_{L^\infty}\|\partial_xf\| \|\partial_y f\|\leq \frac{5}{4}\|\partial_xf\| \|\partial_y f\|,\notag
\end{align}
\begin{align}
&\langle \Delta f,  u_{a}^1 \partial_xf+u_{a}^2\partial_yf \rangle= \langle \partial_{xx} f+\partial_{yy} f, u_{a}^1 \partial_xf+u_{a}^2\partial_yf \rangle \notag\\ 
=&-\frac{1}{2}\langle |\partial_{x} f|^2, \partial_{x}u_{a}^1\rangle-\frac{1}{2}\langle |\partial_{y} f|^2,  \partial_{y}u_{a}^2\rangle-\langle \partial_{x} f,\partial_{x}u_{a}^2\partial_yf \rangle-\langle \partial_{x} f,u_{a}^2\partial_{xy}f \rangle\notag\\ 
&-\langle \partial_{y} f,\partial_{y}u_{a}^1\partial_xf \rangle- \langle \partial_{y} f,u_{a}^1\partial_{xy}f \rangle  \notag\\
=&\frac{1}{2}\langle |\partial_{y} f|^2-|\partial_{x} f|^2, \partial_{x}u_{a}^1-\partial_{y}u_{a}^2\rangle-\langle \partial_{x} f,\partial_{x}u_{a}^2\partial_yf \rangle-\langle \partial_{y} f,\partial_{y}u_{a}^1\partial_xf \rangle\notag\\
\leq& \|\nabla u_{a}\|_{L^\infty}(\|\partial_xf\|^2+\|\partial_yf\|^2)+2\|\nabla u_{a}\|_{L^\infty}\|\partial_xf\| \|\partial_y f\|\notag\\ 
\leq& \frac{1}{4}\|\partial_xf\| \|\partial_y f\|+C\varepsilon_0\nu^{\frac{1}{3}}\|\nabla f\|^2.\notag
\end{align}
Here we used
\[
\|V_1'\|_{L^\infty} \leq 1 + \|V_1' - 1\|_{L^\infty} \leq 1 + C\varepsilon_0 < \frac{5}{4},\ \ \|\nabla u_{a}\|_{L^\infty}\leq  C\varepsilon_0\nu^{\frac{1}{3}}< \frac{1}{4}.
\]
Combining the above four formulas, we can get \eqref{4.5}.

At last, by direct calculation, we have
\begin{align}
\frac{\rm d}{{\rm d}t} \langle \partial_xf, \partial_y f \rangle =& \langle \partial_x (\nu \Delta f -  (V_1+u_{a}^1) \partial_xf-u_{a}^2\partial_y f), \partial_y f \rangle  \notag\\
&+ \langle \partial_xf, \partial_y (\nu \Delta f -(V_1+u_{a}^1) \partial_xf-u_{a}^2\partial_y f ) \rangle \notag\\
=&-\|\sqrt{V_1'+\partial_yu_{a}^1} \partial_xf\|^2 +\nu \langle   \partial_x\Delta f, \partial_y f \rangle +\nu \langle \partial_xf,  \partial_y \Delta f \rangle\notag\\ 
&- \langle \partial_x (u_{a}^2\partial_y f), \partial_y f \rangle -  \langle \partial_xf, \partial_y (u_{a}^2\partial_y f) \rangle. \notag
\end{align}
By using \eqref{2.12}, we have
\begin{align}
|\nu\langle   \partial_x\Delta f, \partial_y f \rangle +\nu \langle \partial_xf,  \partial_y \Delta f \rangle| \leq 2\nu \|\partial_{xy} f\| \|\Delta f\|,\notag
\end{align}
\begin{align}
| \langle \partial_x (u_{a}^2\partial_y f), \partial_y f \rangle +  \langle \partial_xf, \partial_y (u_{a}^2\partial_y f) \rangle|\leq& 2\|\partial_xu_{a}^2\|_{L^\infty} \|\partial_{y} f\|^2\notag\\\leq &C\varepsilon_0\nu^{\frac{1}{3}}(1+t)^{-2}\|\partial_yf\|^2.\notag
\end{align}
Combining the above three formulas implies \eqref{4.6}. Up to now, the proof of this lemma is completed.
\end{proof}

\begin{lem}\label{l4.4}
Under the hypothesis of Proposition \ref{p4.1}, if we take $a_1 = \frac{1}{900}$, $a_2 = 16a_1$, $a_3 = 6a_1$ and
\[
\Phi_k(t) = (1 + a_1 \nu |k|^2 t^3) \|f_k\|^2 + a_2 \nu t \|(k, \partial_y) f_k\|^2 + a_3 \nu t^2 \mathrm{Re} \langle ik f_k, \partial_y f_k \rangle,
\]
where $f$ satisfies \eqref{4.2} and $f_k$ denotes the Fourier transform of $f$ with respect to $x$, then for $0\leq t\leq \nu^{-\frac{1}{6}}$, we have
\begin{equation}\label{4.7}
  \Phi_k(t) \geq \left(1 + \frac{a_1 \nu |k|^2 t^3}{4} \right) \|f_k\|^2 + \frac{a_2}{4} \nu t \|(k, \partial_y) f_k\|^2,  
\end{equation}
\begin{equation}\label{4.8}
  \Phi_k'(t)+\frac{a_1}{2}\left[\nu |k|^2 t^2 \|f_k\|^2+\nu \|(k, \partial_y) f_k\|^2+\nu^2 t \|\Delta_k f_k\|^2+\nu^2 |k|^4 t^3 \|f_k\|^2\right]\leq 0.  
\end{equation}
\end{lem}

\begin{proof}
Notice that $a_2 = 16a_1 > 0$, $a_3 = 6a_1$, we have
\begin{align*}
|a_3 \nu t^2 \mathrm{Re} \langle ik f_k, \partial_y f_k \rangle| &\leq \frac{a_3 }{8}\nu (t^3 |k|^2 \|f_k\|^2 + 16t \|\partial_y f_k\|^2) \\
&= \frac{3}{4}\nu (a_1 t^3 |k|^2 \|f_k\|^2 + a_2 t \|\partial_y f_k\|^2),
\end{align*}
which gives \eqref{4.7}.

Next, by applying Lemma~\ref{lem:4.1} and $|\mathrm{Re} \langle ik f_k, \partial_y f_k \rangle| \leq |k| \|f_k\| \|\partial_y f_k\|$, we can prove
\begin{align*}
\Phi_k'(t)&=3a_1 \nu |k|^2 t^2 \|f_k\|^2 + a_2 \nu  \|(k, \partial_y) f_k\|^2 + 2a_3 \nu t \mathrm{Re} \langle ik f_k, \partial_y f_k \rangle\\
&\quad+(1 + a_1 \nu |k|^2 t^3) \frac{\rm d}{{\rm d}t}\|f_k\|^2 + a_2 \nu t \frac{\rm d}{{\rm d}t}\|(k, \partial_y) f_k\|^2 + a_3 \nu t^2 \frac{\rm d}{{\rm d}t}\mathrm{Re} \langle ik f_k, \partial_y f_k \rangle\\ 
&\leq 3a_1 \nu t^2 |k|^2 \|f_k\|^2 - 2\nu (1 + a_1 \nu |k|^2 t^3) \|(k, \partial_y) f_k\|^2 +Ca_2\varepsilon_0\nu t\nu^{\frac{1}{3}}\|(k, \partial_y) f_k\|^2 \\
&\quad+ a_2 \nu \|(k, \partial_y) f_k\|^2 - 2a_2 \nu^2 t \|\Delta_k f_k\|^2 + 3a_2 \nu t |k| \|f_k\| \|\partial_y f_k\| + 2a_3 \nu t |k| \|f_k\| \|\partial_y f_k\| \\
&\quad - a_3 \nu t^2 |k|^2 \|\sqrt{V_1'+\partial_yu_{a}^1} f_k\|^2 + 2a_3 \nu^2 t^2 |k| \|\partial_y f_k\| \|\Delta_k f_k\|+Ca_3 \varepsilon_0\nu^{\frac{1}{3}}\nu\|\partial_yf\|^2 \\
&= \nu |k|^2 t^2 \left(4a_1 \|f_k\|^2 - a_3 \|\sqrt{V_1'+\partial_yu_{a}^1} f_k\|^2\right) \\ 
&\quad- \nu (1 - a_2-Ca_2\varepsilon_0 t\nu^{\frac{1}{3}}-Ca_3\varepsilon_0 \nu^{\frac{1}{3}}) \|(k, \partial_y) f_k\|^2 \\
&\quad - \left[\nu |k|^2 t^2 a_1 \|f_k\|^2 - (3a_2 + 2a_3) \nu t |k| \|f_k\| \|\partial_y f_k\| + \nu \|(k, \partial_y) f_k\|^2\right] \\
&\quad - \left[2a_2 \nu^2 t \|\Delta_k f_k\|^2 - 2a_3 \nu^2 t^2 |k| \|\partial_y f_k\| \|\Delta_k f_k\| + 2a_1 \nu^2 |k|^2 t^3 \|(k, \partial_y) f_k\|^2\right] \\
&:= I_1 + I_2 + I_3 + I_4.
\end{align*}
From Proposition \ref{p2.1}, \eqref{eq:4.1} and $a_3 = 6a_1 > 0$, we have 
\[
V_1'+\partial_yu_{a}^1 \geq 1 - \|V_1' - 1\|_{L^\infty}-\|\partial_yu_{a}^1\|_{L^\infty} \geq 1 - \frac{1}{4} >\frac{3}{4}, 
\]
\begin{equation}\notag
a_3 \|\sqrt{V_1'+\partial_yu_{a}^1} f_k\|^2 \geq \frac{3a_3}{4} \|f_k\|^2 = \frac{9a_1}{2} \|f_k\|^2,
\end{equation}
this implies 
\begin{equation}\notag
 I_1\leq -\frac{a_1}{2}\nu |k|^2 t^2 \|f_k\|^2 .  
\end{equation}
For $I_2$, due to $a_2 = \frac{4}{225}$, $\varepsilon_0\ll 1$ and $t\leq \nu^{-\frac{1}{6}}$, we can deduce 
\begin{equation}\notag
    I_2\leq  -\frac{1}{2}\nu \|(k, \partial_y) f_k\|^2.
\end{equation}
Finally, note that  $a_1 = \frac{1}{900}$, $a_2 = 16a_1$, $a_3 = 6a_1$, then $3a_2 + 2a_3=60a_1 $, and
\begin{align*}
I_3 &= -[\nu |k|^2 t^2 a_1 \|f_k\|^2 - (3a_2 + 2a_3) \nu t |k| \|f_k\| \|\partial_y f_k\| + \nu \|(k, \partial_y) f_k\|^2] \\
&= -\nu a_1 (|k|t \|f_k\| - 30 \|(k, \partial_y) f_k\|)^2 \leq 0,
\end{align*}
\begin{align}
    I_4 =& -[2a_2 \nu^2 t \|\Delta_k f_k\|^2 - 2a_3 \nu^2 t^2 |k| \|\partial_y f_k\| \|\Delta_k f_k\| + 2a_1 \nu^2 |k|^2 t^3 \|(k, \partial_y) f_k\|^2]\notag\\ 
    &-[14a_1 \nu^2 t \|\Delta_k f_k\|^2 + 2a_1 \nu^2 t (3\|\Delta_k f_k\| - |k|t \|\partial_y f_k\|)^2 + 2a_1 \nu^2 |k|^4 t^3 \|f_k\|^2].\notag
\end{align}
Summing up, we arrive at \eqref{4.8}.
\end{proof}

Let's begin to prove Proposition \ref{p4.1} now.

\begin{proof}[\textbf{Proof of Proposition~\ref{p4.1}}]
   According to \eqref{4.7}--\eqref{4.8}, we derive for $k\in \mathbb{Z}$ that
   \begin{align}
      \|f_{in,k}\|^2=& \Phi_k(0)\geq  \Phi_k(t) \geq \left(1 + \frac{a_1 \nu |k|^2 t^3}{4} \right) \|f_k\|^2 + \frac{a_2}{4} \nu t \|(k, \partial_y) f_k\|^2\notag\\
      \gtrsim& \nu t(1+t^2)|k|^2\|f_k\|^2.\notag
   \end{align}
  Then summing over the integer $k$, we obtain \eqref{aa1}. 

  In addition, we can also derive for $k\neq 0$,
  \begin{align}
      \|f_{in,k}\|^2=& \Phi_k(0)\geq  \Phi_k(t) \geq \left(1 + \frac{a_1 \nu |k|^2 t^3}{4} \right) \|f_k\|^2 + \frac{a_2}{4} \nu t \|(k, \partial_y) f_k\|^2\notag\\
      \gtrsim& (1+\nu |k|^2 t^3)\|f_k\|^2\gtrsim (1+\nu t^3)\|f_k\|^2.\notag
   \end{align}
   Summing over the non-zero integer $k$ gives \eqref{aa2}.
\end{proof}

Based on the conclusion in Proposition \ref{p4.1}, we can immediately obtain the following corollary.
\begin{coro}\label{c4.5}
    Under the conditions of Proposition \ref{p4.1}, if $f$ is the solution to \eqref{4.2}, then we have for $t\leq \nu^{-\frac{1}{6}}$ that
  \begin{align*}
&\|f(t)\| + \nu^{\frac{1}{2}} \left(\int_0^t \|\nabla f(\tau)\|^2 \dd \tau\right)^{\frac{1}{2}} + \nu^{\frac{1}{2}}\int_0^t \|\partial_x f(\tau)\|\dd \tau+\nu^{\frac{1}{3}}\int_0^t \|f_{\neq}(\tau)\|\dd \tau \\
&\lesssim \|f_{in}\|.
\end{align*}
\end{coro}

With the above preparations, we turn to the estimates for $(\omega,\theta)$ on a short time scale.

 \subsection{Short time scale $0\leq t\leq \nu^{-\frac{1}{6}}$}
We define the space-time norm $\|f\|_{X_1}$, where
\begin{align*}
\|f\|_{X_1} &:= \|f\|_{L^\infty L^2} + \nu^{\frac{1}{2}} \|\nabla f\|_{L^2 L^2}+\nu^{\frac{1}{2}} \| \partial_x f\|_{L^1 L^2} + \nu^{\frac{1}{3}} \|P_{\ne} f\|_{L^1 L^2}.
\end{align*}
\begin{prop}\label{p4.7}
Let $(\omega,\theta)$ be the solution to \eqref{2.9}, then for $0\leq t\leq T_0=\nu^{-\frac{1}{6}}$,
\begin{align}\label{4.9}
  \|\theta\|_{X_1}\lesssim& \|\Theta_{in}\|_{H^4}( \nu^{\frac{2}{3}}+\|\Omega_{in}\|_{H^3}+\nu^{-\frac{1}{3}} \|\Theta_{in}\|_{H^4})+\nu^{-\frac{1}{3}}\|\omega\|_{L^\infty L^2}\|\Theta_{in}\|_{H^4}\notag\\ 
  &+\nu^{-\frac{1}{2}} \|\omega\|_{L^\infty L^2}\|\theta\|_{X_1}+\nu^{-\frac{5}{8}} \|\omega\|_{L^\infty L^2}^{\frac{9}{10}}(\nu^{\frac{1}{2}}\|\nabla \omega\|_{L^2 L^2})^{\frac{1}{10}}\|\theta\|_{X_1}.  
\end{align}

\end{prop}
\begin{proof} Notice that
\begin{equation}\notag
    \partial_t \theta - \nu \Delta \theta + V_1 \partial_x \theta+u_{a}\cdot\nabla\theta=g_1, \ \theta|_{t=0}=0,
\end{equation}
where 
\begin{equation}\notag
    g_1=V_2' \partial_x \phi - u \cdot \nabla \theta_{a}- u \cdot \nabla \theta  -Er^2.
\end{equation}
From Corollary \ref{c4.5} and Duhamel's principle, we have 
\begin{equation}\label{4.10}
\|\theta\|_{X_1} \lesssim \left\| g_1 \right\|_{L^1 L^2}.
\end{equation}
Next, we are going to estimate $\left\| g_1 \right\|_{L^1 L^2}$.  Note that $\phi = \Delta^{-1} \omega$, using \eqref{2.4} gives
\[
\|V_2' \partial_x \phi\| \leq \|V_2' \|_{L^\infty} \|\partial_x \phi\| \lesssim \|\Theta_{in}\|_{H^4} \|\omega\|,
\]
and then,
\begin{align*}
\|V_2' \partial_x \phi\|_{L^1 L^2} \lesssim \|\Theta_{in}\|_{H^4} \|\omega\|_{L^\infty L^2}t \lesssim \nu^{-\frac{1}{6}} \|\Theta_{in}\|_{H^4} \|\omega\|_{L^\infty L^2}.
\end{align*}
It follows from \eqref{2.10} that
\begin{align}
\|u \cdot \nabla \theta_{a}\| \lesssim \|u\| \|\nabla\theta_{a}\|_{L^\infty}\lesssim t\|\omega\|\|\Theta_{in}\|_{H^4},\notag
\end{align}
thus, by $0\leq t\leq \nu^{-\frac{1}{6}}$,
\begin{align}
\|u \cdot \nabla \theta_{a}\|_{L^1 L^2} \lesssim t^2\|\omega\|_{L^\infty L^2}\|\Theta_{in}\|_{H^3}\lesssim \nu^{-\frac{1}{3}}\|\omega\|_{L^\infty L^2}\|\Theta_{in}\|_{H^4}.\notag
\end{align}

Notice that
\[
u \cdot \nabla \theta = P_{\ne} u \cdot \nabla \theta + P_0 u^1 \partial_x \theta.
\]
We have from Lemma~\ref{lem:A.1} and $u = (\partial_y, -\partial_x) \Delta^{-1} \omega$ that
\begin{align*}
\|P_{\ne} u\|_{L^\infty} \lesssim \|\omega\|^{1-\frac{\delta_1}{2}} \|\nabla\omega\|^{\frac{\delta_1}{2}},\ \ \ \|P_0 u\|_{L^\infty} \lesssim \|P_0 u\|_{H^1} \lesssim \|\omega\|,
\end{align*}
and then
\begin{align*}
\|P_{\ne} u\|_{L^2 L^\infty} &\lesssim \|\omega\|_{L^2 L^2}^{1-\frac{\delta_1}{2}} \|\nabla \omega\|_{L^2 L^2}^{\frac{\delta_1}{2}} \lesssim\|\omega\|_{L^\infty L^2}^{1-\frac{\delta_1}{2}}t^{\frac{1}{2}(1-\frac{\delta_1}{2})}\|\nabla \omega\|_{L^2 L^2}^{\frac{\delta_1}{2}}\\
&\lesssim \nu^{-\frac{1}{12}+\frac{\delta_1}{24}}\nu^{-\frac{\delta_1}{4}}\|\omega\|_{L^\infty L^2}^{1-\frac{\delta_1}{2}}(\nu^{\frac{1}{2}}\|\nabla \omega\|_{L^2 L^2})^{\frac{\delta_1}{2}}.
\end{align*}
Let $\delta_1=\frac{1}{5}$, then 
\begin{align*}
\|P_{\ne} u\|_{L^2 L^\infty} \lesssim\nu^{-\frac{1}{8}} \|\omega\|_{L^\infty L^2}^{\frac{9}{10}}(\nu^{\frac{1}{2}}\|\nabla \omega\|_{L^2 L^2})^{\frac{1}{10}}.
\end{align*}
Thus, by H\"older's inequality, we get
\begin{align*}
\|P_{\ne} u \cdot \nabla \theta\|_{L^1 L^2} &\leq \|P_{\ne} u\|_{L^2 L^\infty} \|\nabla \theta\|_{L^2 L^2} \lesssim \nu^{-\frac{5}{8}} \|\omega\|_{L^\infty L^2}^{\frac{9}{10}}(\nu^{\frac{1}{2}}\|\nabla \omega\|_{L^2 L^2})^{\frac{1}{10}}\|\theta\|_{X_1}, \\
\|P_0 u^1 \partial_x \theta\|_{L^1 L^2} &\leq \|\omega\|_{L^\infty L^2} \|\partial_x \theta\|_{L^1 L^2} \lesssim \nu^{-\frac{1}{2}} \|\omega\|_{L^\infty L^2}\|\theta\|_{X_1}.
\end{align*}
Finally, applying \eqref{2.15}and $0 \leq t \leq \nu^{-\frac{1}{6}}$ leads to
\begin{align*}
\|Er^2\|_{L^1 L^2} &\lesssim \|\nu (1+t)\|_{L^1} \|\Theta_{in}\|_{H^4} + \|(1+t)^{-2}\|_{L^1}  \|\Theta_{in}\|_{H^4}(\|\Omega_{in}\|_{H^3}+\nu^{-\frac{1}{3}} \|\Theta_{in}\|_{H^4}) \\
&\lesssim \|\Theta_{in}\|_{H^4}(\nu^{\frac{2}{3}} +\|\Omega_{in}\|_{H^3}+\nu^{-\frac{1}{3}} \|\Theta_{in}\|_{H^4}).
\end{align*}

The proof of this proposition is completed.
\end{proof}

Now, we give the uniform estimate of $(\omega,\theta)(x,t)$ in a short time scale.
\begin{prop}\label{p4.8}
If $0 \leq t \leq T_0 = \nu^{-\frac{1}{6}}$, then
\begin{equation}\label{4.11}
    \|\omega\|_{L^\infty L^2}+\nu^{\frac{1}{2}} \|\nabla \omega\|_{L^2L^2}+\nu^{-\frac{1}{2}}\|\theta\|_{X_1} \lesssim \nu^{\frac{1}{3}}\|\Omega_{in}\|_{H^3}+\nu^{-\frac{1}{6}}\|\Theta_{in}\|_{H^4}.
\end{equation}
\end{prop}

\begin{proof}
Firstly, we multiply  \eqref{2.9}$_1$ and \eqref{2.9}$_2$ by $\omega$ and $\theta$ respectively, and integrate it with respect to $x$ and $y$ to obtain
\[
\frac{1}{2} \frac{\rm d}{{\rm d}t} \|\omega\|^2 + \nu \|\nabla \omega\|^2 = \int_{-1}^1\int_\mathbb{T} (V_1'' \partial_x \phi - u \cdot \nabla \omega_{a} - Er^1-\partial_x \theta) \omega \dd x \dd y,
\]
\[
\frac{1}{2} \frac{\rm d}{{\rm d}t} \|\theta\|^2 + \nu \|\nabla \theta\|^2 = \int_{-1}^1\int_\mathbb{T} (V_2' \partial_x \phi -u \cdot \nabla \theta_{a} - Er^2) \theta \dd x \dd y.
\]
Using $u = (\partial_y, -\partial_x) \phi$ and $\phi = \Delta^{-1} \omega$, we have $\|u\| = \|\nabla \phi\| \leq C \|\omega\|$,
and
\begin{align*}
&\frac{1}{2}\frac{\rm d}{{\rm d}t} \|\omega\|^2_{L^2}+\nu \|\nabla \omega\|^2\notag\\
&\leq \|\omega\|\|V_1'' \partial_x \phi - u \cdot \nabla \omega_{a} - Er^1\|+\|\omega\|\| \partial_x\theta\| \\
&\leq \| \omega\|(\|\partial_x \phi\| \|V_1''\|_{L^\infty} + \|u\| \|\nabla \omega_{a}\|_{L^\infty} + \|Er^1\|) +\|\omega\|\| \partial_x\theta\|\\
&\lesssim\|\omega\|^2( \|V_1''\|_{L^\infty} +\|\nabla \omega_{a}\|_{L^\infty}) + \|\omega\|\|Er^1\| + \|\omega\|\| \partial_x\theta\|,
\end{align*}
\begin{align*}
\frac{1}{2}\frac{\rm d}{{\rm d}t} \|\theta\|^2_{L^2}+\nu \|\nabla \theta\|^2
&\leq \|\theta\|\|V_2' \partial_x \phi - u \cdot \nabla \theta_{a} - Er^2\|\\
&\leq \| \theta\|(\|\partial_x \phi\| \|V_2'\|_{L^\infty} + \|u\| \|\nabla \theta_{a}\|_{L^\infty} + \|Er^2\|) \\
&\lesssim\| \theta\|\|\omega\|( \|V_2'\|_{L^\infty} +\|\nabla \theta_{a}\|_{L^\infty}) + \|\theta\|\|Er^2\| .
\end{align*}
Integrating the above equations with respect to $t$ yields
\begin{align}\label{123}
  &\|\omega\|^2_{L^\infty L^2}+\nu \|\nabla \omega\|_{L^2L^2}^2 \notag\\ 
  &\leq C\|\omega\|_{L^\infty L^2}^2( \|V_1''\|_{L^1L^\infty} +\|\nabla \omega_{a}\|_{L^1 L^\infty}) + C\|\omega\|_{L^\infty L^2}(\|Er^1\|_{L^1L^2}+\|\partial_x\theta\|_{L^1L^2}).
\end{align}
\begin{align}\label{456}
  &\|\theta\|^2_{L^\infty L^2}+\nu \|\nabla \theta\|_{L^2L^2}^2 \notag\\ 
  &\leq C\|\theta\|_{L^\infty L^2}\|\omega\|_{L^\infty L^2}( \|V_2'\|_{L^1L^\infty} +\|\nabla \theta_{a}\|_{L^1 L^\infty}) + C\|\theta\|_{L^\infty L^2}\|Er^2\|_{L^1L^2}. 
\end{align}
From \eqref{2.3}--\eqref{2.4} and Proposition~\ref{p2.1}, we have
\begin{align*}
&\|\nabla \omega_{a}\|_{L^\infty} + \|V_1''\|_{L^\infty}\leq C (1+t) (\|\Omega_{in}\|_{H^3}+\nu^{-\frac{1}{3}}\|\Theta_{in}\|_{H^4}) \leq C\varepsilon_0 (1+t) \nu^{\frac{1}{3}}, \\
&\|\nabla \theta_{a}\|_{L^\infty} + \|V_2'\|_{L^\infty}\leq C (1+t) \|\Theta_{in}\|_{H^4}\leq C\varepsilon_0 (1+t) \nu^{\frac{5}{6}}, \\
&\|Er^1\| \lesssim \nu (1+t) (\|\Omega_{in}\|_{H^3}+\nu^{-\frac{1}{3}}\|\Theta_{in}\|_{H^4}) + (1+t)^{-2} (\|\Omega_{in}\|_{H^3}+\nu^{-\frac{1}{3}}\|\Theta_{in}\|_{H^4})^2,\\
&\|Er^2\|\lesssim  \nu (1+t)\|\Theta_{in}\|_{H^4} +  (1+t)^{-2}\|\Theta_{in}\|_{H^4}(\|\Omega_{in}\|_{H^3}+\nu^{-\frac{1}{3}}\|\Theta_{in}\|_{H^4}).
\end{align*}
 Then, we arrive at 
 \begin{align}
  &\|\omega\|^2_{L^\infty L^2}+\nu \|\nabla \omega\|_{L^2L^2}^2 \notag\\ 
  \leq& C\varepsilon_0(1+t)^2 \nu^{\frac{1}{3}}\|\omega\|_{L^\infty L^2}^2 +\|\omega\|_{L^\infty L^2}\|\partial_x\theta\|_{L^1L^2}\notag\\ 
  &+ C\left[\nu (1+t)^2 (\|\Omega_{in}\|_{H^3}+\nu^{-\frac{1}{3}}\|\Theta_{in}\|_{H^4}) + (\|\Omega_{in}\|_{H^3}+\nu^{-\frac{1}{3}}\|\Theta_{in}\|_{H^4})^2\right]\|\omega\|_{L^\infty L^2} \notag\\ 
  \leq& C\varepsilon_0\|\omega\|_{L^\infty L^2}^2 +C\left[\nu^{-\frac{1}{2}}\|\theta\|_{X_1}+\nu^{\frac{1}{3}}(\|\Omega_{in}\|_{H^3}+\nu^{-\frac{1}{3}}\|\Theta_{in}\|_{H^4})\right]\|\omega\|_{L^\infty L^2}.\notag
\end{align}
\begin{align}
 &\|\theta\|_{L^\infty L^2}^2 + \nu \|\nabla \theta\|_{L^2L^2}^2\notag\\  
 \leq& C\varepsilon_0(1+t)^2 \nu^{\frac{5}{6}}\|\theta\|_{L^\infty L^2}\|\omega\|_{L^\infty L^2}\notag\\ 
 &+ C\left[\nu (1+t)^2\|\Theta_{in}\|_{H^4} + \|\Theta_{in}\|_{H^4}(\|\Omega_{in}\|_{H^3}+\nu^{-\frac{1}{3}}\|\Theta_{in}\|_{H^4})\right]\|\theta\|_{L^\infty L^2}\notag\\
 \leq& C\varepsilon_0\nu^{\frac{1}{2}}\|\theta\|_{L^\infty L^2}\|\omega\|_{L^\infty L^2}+C\|\Theta_{in}\|_{H^4}(\nu^{\frac{2}{3}}+\|\Omega_{in}\|_{H^3}+\nu^{-\frac{1}{3}}\|\Theta_{in}\|_{H^4})\|\theta\|_{L^\infty L^2}.\notag
\end{align}
By using the smallness of $\varepsilon_0$, we obtain 
 \begin{align}\label{4.12}
  \|\omega\|_{L^\infty L^2}+\nu^{\frac{1}{2}} \|\nabla \omega\|_{L^2L^2} 
  \leq C\left[\nu^{-\frac{1}{2}}\|\theta\|_{X_1}+\nu^{\frac{1}{3}}(\|\Omega_{in}\|_{H^3}+\nu^{-\frac{1}{3}}\|\Theta_{in}\|_{H^4})\right],
\end{align}
\begin{align}\label{4.13}
 \|\theta\|_{L^\infty L^2} + \nu^{\frac{1}{2}} \|\nabla \theta\|_{L^2L^2}  
 \leq C\varepsilon_0\nu^{\frac{1}{2}}\|\omega\|_{L^\infty L^2}+C\|\Theta_{in}\|_{H^4}(\nu^{\frac{1}{3}}+\|\Omega_{in}\|_{H^3}+\nu^{-\frac{1}{3}}\|\Theta_{in}\|_{H^4}).
\end{align}
Notice that 
\begin{align}\label{4.14}
  \nu^{-\frac{1}{2}}\|\theta\|_{X_1}\leq& C_1\nu^{-\frac{1}{2}}\|\Theta_{in}\|_{H^4}( \nu^{\frac{2}{3}}+\|\Omega_{in}\|_{H^3}+\nu^{-\frac{1}{3}} \|\Theta_{in}\|_{H^4})+C_1\nu^{-\frac{5}{6}}\|\omega\|_{L^\infty L^2}\|\Theta_{in}\|_{H^4}\notag\\ 
  &+C_1\nu^{-1} \|\omega\|_{L^\infty L^2}\|\theta\|_{X_1}+C_1\nu^{-\frac{9}{8}}\|\omega\|_{L^\infty L^2}^{\frac{9}{10}}(\nu^{\frac{1}{2}}\|\nabla \omega\|_{L^2 L^2})^{\frac{1}{10}}\|\theta\|_{X_1}. \end{align}
  Combining the above three equations gives
\begin{align}\label{4.15a}
&\|\omega\|_{L^\infty L^2}+\nu^{\frac{1}{2}} \|\nabla \omega\|_{L^2L^2}+\nu^{-\frac{1}{2}} \|\theta\|_{L^\infty L^2}+\nu^{-\frac{1}{2}}\|\theta\|_{X_1}\notag\\ 
\leq&C\nu^{\frac{1}{3}}(\|\Omega_{in}\|_{H^3}+\nu^{-\frac{1}{3}}\|\Theta_{in}\|_{H^4})+C\nu^{-\frac{1}{6}}\|\Theta_{in}\|_{H^4}\notag\\ 
&+C\nu^{-1} \|\omega\|_{L^\infty L^2}\|\theta\|_{X_1}+C\nu^{-\frac{9}{8}} \|\omega\|_{L^\infty L^2}^{\frac{9}{10}}(\nu^{\frac{1}{2}}\|\nabla \omega\|_{L^2 L^2})^{\frac{1}{10}}\|\theta\|_{X_1}.
\end{align}
Assume that $\|\omega\|_{L^\infty L^2}+\nu^{\frac{1}{2}} \|\nabla \omega\|_{L^2L^2}+\nu^{-\frac{1}{2}}\|\theta\|_{X_1}\leq \varepsilon_1\nu^{\frac{2}{3}}$ for some small constant $\varepsilon_1$, we conclude
\begin{align}\label{4.15}
\|\omega\|_{L^\infty L^2}+\nu^{\frac{1}{2}} \|\nabla \omega\|_{L^2L^2}+\nu^{-\frac{1}{2}}\|\theta\|_{X_1}\leq C\nu^{\frac{1}{3}}(\|\Omega_{in}\|_{H^3}+\nu^{-\frac{1}{2}}\|\Theta_{in}\|_{H^4}).
\end{align}
Notice that $\|\Omega_{in}\|_{H^3} +\nu^{-\frac{1}{2}}\|\Theta_{in}\|_{H^4}\leq \varepsilon_0 \nu^{\frac{1}{3}},$
then we have
\[
\|\omega\|_{L^\infty L^2}+\nu^{\frac{1}{2}} \|\nabla \omega\|_{L^2L^2}+\nu^{-\frac{1}{2}} \|\theta\|_{L^\infty L^2}+\nu^{-\frac{1}{2}}\|\theta\|_{X_1} \leq C(\nu^{\frac{1}{3}}\|\Omega_{in}\|_{H^3}+\nu^{-\frac{1}{6}}\|\Theta_{in}\|_{H^4}) \leq C\varepsilon_0 \nu^{\frac{2}{3}}.
\]
Taking $\varepsilon_0 = \frac{\varepsilon_1 }{2C}$, then one can get \eqref{4.11} for $0 \leq t \leq\nu^{-\frac{1}{6}}$.
Thus, we get \eqref{4.11}. 
The proof of this proposition is completed.
\end{proof}

\subsubsection{Long time scale $t\geq \nu^{-\frac{1}{6}}=T_0$}

In this subsection, we study the space-time estimate for $(\omega,\theta)$ with $t\geq \nu^{-\frac{1}{6}}$. For this, we first consider the following linearized system:
\begin{equation}\label{4.15aa}
    \partial_t f + V_1 \partial_x f- \nu \Delta f  = g_2, \quad f(t,x,\pm 1) = 0, \quad \text{for }\  t \geq T_0.
\end{equation}
Define the space-time norm $\|f\|_{X_2}$ in $t\geq T_0$, where
\begin{align}\label{4.15aaa}
\|f\|_{X_2} &= \|f\|_{L^\infty L^2} + \nu^{\frac{1}{2}} \|\nabla f\|_{L^2 L^2} + \|\e^{\delta \nu^{\frac{1}{3}} t} P_{\neq} f\|_{L^\infty L^2} + \nu^{\frac{1}{2}} \|\e^{\delta \nu^{\frac{1}{3}} t} \nabla P_{\neq} f\|_{L^2 L^2} \notag\\
&\quad + \|\e^{\delta \nu^{\frac{1}{3}} t} \partial_x \nabla \Delta^{-1} f\|_{L^2 L^2} + \nu^{\frac{1}{2}} \|\e^{\delta \nu^{\frac{1}{3}} t} \partial_x f\|_{L^1 L^2} + \nu^{\frac{1}{3}} \|\e^{\delta \nu^{\frac{1}{3}} t} P_{\neq} f\|_{L^1 L^2}.
\end{align}
We have the following proposition, which provides a bound for the norm $\|\cdot\|_{X_2}$.
\begin{prop}\label{p4.8a}
Suppose that $f$ satisfies \eqref{4.15aa} and $V_1(t,y)$ satisfies \eqref{eq:4.1}, then for $t\geq T_0$, there exists $0 < \delta \ll 1$ such that
\begin{equation}\label{4.15bb}
\|f\|_{X_2} \lesssim \|f(T_0)\| + \|\e^{\delta \nu^{\frac{1}{3}} t} g_2\|_{L^1 L^2}.
\end{equation}
\end{prop}
 \begin{proof}
 The proof of this proposition is similar to Proposition 2.4 in \cite{wei-zhang2026}. For brevity, we
omit the details.
\end{proof}

To obtain the estimates of $(\omega,\theta)$ without explicit expression, we need to use Proposition~\ref{p4.8a}, the main difficulty is the reaction part of $u \cdot \nabla \omega_{a}, \ u \cdot \nabla \theta_{a}$ due to the polynomial growth of $ \nabla \omega_{a}$ and $\nabla \theta_{a}$. Motivated by \cite{wei-zhang2026}, we apply Proposition~\ref{p2.1} to find that
\[
u \cdot \nabla \omega_{a} + t u^{(2)} \partial_x \omega_{a} = u^{(1)} \partial_x \omega_{a} + u^{(2)} (\partial_y + t \partial_x) \omega_{a},
\]
and
\[
u \cdot \nabla \theta_{a} + t u^{(2)} \partial_x \theta_{a} = u^{(1)} \partial_x \theta_{a} + u^{(2)} (\partial_y + t \partial_x) \theta_{a},
\]
are smaller. Thus, we introduce the reaction parts $\omega_{re}$ and $\theta_{re}$, which satisfy
\begin{align}
  &\partial_t \omega_{re} - \nu \Delta \omega_{re} + V_1 \partial_x \omega_{re} - t u^{(2)} \partial_x \omega_{a} = 0, \quad \omega_{re}(t,x,\pm 1) = 0, \quad \omega_{re}|_{t=T_0} = 0,\label{aaa}\\ 
  &\partial_t \theta_{re} - \nu \Delta \theta_{re} + V_1 \partial_x \theta_{re} - t u^{(2)} \partial_x \theta_{a} = 0, \quad \theta_{re}(t,x,\pm 1) = 0, \quad \theta_{re}|_{t=T_0} = 0.\label{bbb}
\end{align}
Let $\omega_* = \omega - \omega_{re}$, $\theta_* = \theta - \theta_{re}$, then we get that
\begin{align}
\partial_t \omega_* - \nu \Delta \omega_* + V_1 \partial_x \omega_* = h_1 , \quad \omega_*(t,x,\pm 1) = 0, \quad \omega_*|_{t=T_0} = \omega|_{t=T_0},\label{aaaa}\\
\partial_t \theta_* - \nu \Delta \theta_* + V_1 \partial_x \theta_*  = h_2, \quad \theta_*(t,x,\pm 1) = 0, \quad \theta_*|_{t=T_0} = \theta|_{t=T_0},\label{bbbb}
\end{align}
where
\begin{align}
h_1 &= V_1'' \partial_x \phi - u_{ob} \cdot \nabla \omega - u^{(1)} \partial_x \omega_{a} - u^{(2)} (\partial_y + t \partial_x) \omega_{a} - Er^1-\partial_x \theta ,\notag\\
    h_2 &=V_2' \partial_x \phi - u_{ob} \cdot \nabla \theta - u^{(1)} \partial_x \theta_{a} - u^{(2)} (\partial_y + t \partial_x)\theta_{a}   - Er^2.\notag
\end{align}

As for $h_1$ and $h_2$, we can deduce the following estimate.

\begin{prop}\label{prop:2.6}
The following estimates hold
\begin{align}\label{h1}
\|\e^{\delta \nu^{\frac{1}{3}} t} h_1\|_{L^1 L^2} \lesssim& \nu^{-\frac{1}{3}} (\|\Omega_{in}\|_{H^3}+\nu^{-\frac{1}{3}}\|\Theta_{in}\|_{H^4}) \|\omega\|_{X_2}+ \nu^{-\frac{5}{8}}\|\omega\|_{X_2}^2 \notag\\
&+\nu^{\frac{1}{3}} \|\Omega_{in}\|_{H^3}+\|\Theta_{in}\|_{H^4}+\nu^{-\frac{1}{2}} \|\theta\|_{X_2},
\end{align}
\begin{align}\label{h2}
\|\e^{\delta \nu^{\frac{1}{3}} t} h_2\|_{L^1 L^2} \lesssim& \nu^{-\frac{1}{3}} \|\Theta_{in}\|_{H^4} \|\omega\|_{X_2}+\nu^{-\frac{1}{3}} (\|\Omega_{in}\|_{H^3}+\nu^{-\frac{1}{3}}\|\Theta_{in}\|_{H^4}) \|\theta\|_{X_2}\notag\\
&+ \nu^{-\frac{5}{8}} \|\omega\|_{X_2}\|\theta\|_{X_2} +\nu^{\frac{1}{3}} \|\Theta_{in}\|_{H^4}.
\end{align}
\end{prop}
\begin{proof}
Let's first prove \eqref{h2}, which will be divided into several steps.

{\it Step 1. The estimate of $V_2' \partial_x \phi$.} By using \eqref{2.4}, we have
\begin{align*}
\|\e^{\delta \nu^{\frac{1}{3}} t} V_2' \partial_x \phi\|_{L^1 L^2} &\lesssim \|\Theta_{in}\|_{H^4} \|\e^{\delta \nu^{\frac{1}{3}} t} P_{\neq} \omega\|_{L^1 L^2} \lesssim \nu^{-\frac{1}{3}} \|\Theta_{in}\|_{H^4} \|\omega\|_{X_2}\\
&\lesssim\nu^{-\frac{1}{3}} \|\Theta_{in}\|_{H^4} \|\omega\|_{X_2}.
\end{align*}

{\it Step 2. The estimate of $u_{ob} \cdot \nabla \theta$.} Note that $u_{ob}=u +u_{a}$, thus,
\[
u_{ob} \cdot \nabla \theta = u_{a} \cdot \nabla \theta + u \cdot \nabla \theta.
\]

For the first term, 
\[
u_{a} \cdot \nabla \theta = u_{a}^{(1)} \partial_x \theta + u_{a}^{(2)} \partial_y \theta.
\]
It follows from \eqref{2.13} that
\begin{align*}
\|u_{a}^{(1)}\|_{L^\infty L^\infty} &\lesssim T_0^{-1} (\|\Omega_{in}\|_{H^3}+\nu^{-\frac{1}{3}}\|\Theta_{in}\|_{H^4}) \lesssim\nu^{\frac{1}{6}} (\|\Omega_{in}\|_{H^3}+\nu^{-\frac{1}{3}}\|\Theta_{in}\|_{H^4}), \\
\|\e^{\delta \nu^{\frac{1}{3}} t} u_{a}^{(2)}\|_{L^2 L^\infty} &\lesssim \|t^{-2}\|_{L^2(T_0,\infty)} (\|\Omega_{in}\|_{H^3}+\nu^{-\frac{1}{3}}\|\Theta_{in}\|_{H^4})\\& \lesssim T_0^{-1} (\|\Omega_{in}\|_{H^3}+\nu^{-\frac{1}{3}}\|\Theta_{in}\|_{H^4}) \lesssim \nu^{\frac{1}{6}} (\|\Omega_{in}\|_{H^3}+\nu^{-\frac{1}{3}}\|\Theta_{in}\|_{H^4}).
\end{align*}
Thus, by Hölder's inequality, we deduce
\begin{align*}
\|\e^{\delta \nu^{\frac{1}{3}} t} u_{a} \cdot \nabla \theta\|_{L^1 L^2} &\lesssim \|u_{a}^{(1)}\|_{L^\infty L^\infty} \|\e^{\delta \nu^{\frac{1}{3}} t} \partial_x \theta\|_{L^1 L^2} + \|\e^{\delta \nu^{\frac{1}{3}} t} u_{a}^{(2)}\|_{L^2 L^\infty} \|\nabla \theta\|_{L^2 L^2} \\
&\lesssim \nu^{\frac{1}{6}} (\|\Omega_{in}\|_{H^3}+\nu^{-\frac{1}{3}}\|\Theta_{in}\|_{H^4})\nu^{-\frac{1}{2}} \|\theta\|_{X_2}\\
&\lesssim \nu^{-\frac{1}{3}} (\|\Omega_{in}\|_{H^3}+\nu^{-\frac{1}{3}}\|\Theta_{in}\|_{H^4}) \|\theta\|_{X_2}.
\end{align*}
For the second term, 
\[
u \cdot \nabla \theta = P_{\neq} u \cdot \nabla \theta + P_0 u^1 \partial_x \theta.
\]
It follows from Lemma~\ref{lem:A.1} and $u = (\partial_y, -\partial_x) \Delta^{-1} \omega$ that
\begin{align*}
\|\e^{\delta \nu^{\frac{1}{3}} t} P_{\neq} u\|_{L^2 L^\infty} &\lesssim \|\e^{\delta \nu^{\frac{1}{3}} t} \partial_x u\|_{L^2 L^2}^{\frac{3}{4}} \|\e^{\delta \nu^{\frac{1}{3}} t} \nabla \omega\|_{L^2 L^2}^{\frac{1}{4}} \\
&\lesssim \|\omega\|_{X_2}^{\frac{3}{4}} (\nu^{-\frac{1}{2}} \|\omega\|_{X_2})^{\frac{1}{4}} = \nu^{-\frac{1}{8}} \|\omega\|_{X_2}.
\end{align*}
One can infer that
\begin{align*}
\|\e^{\delta \nu^{\frac{1}{3}} t} P_{\neq} u \cdot \nabla \theta\|_{L^1 L^2} &\leq \|\e^{\delta \nu^{\frac{1}{3}} t} P_{\neq} u\|_{L^2 L^\infty} \|\nabla \theta\|_{L^2 L^2} \lesssim \nu^{-\frac{1}{8}} \|\omega\|_{X_2} \cdot \nu^{-\frac{1}{2}} \|\theta\|_{X_2} \lesssim \nu^{-\frac{5}{8}}\|\omega\|_{X_2}\|\theta\|_{X_2}, \\
\|\e^{\delta \nu^{\frac{1}{3}} t} P_0 u^1 \partial_x \theta\|_{L^1 L^2} &\leq \|\omega\|_{L^\infty L^2} \|\e^{\delta \nu^{\frac{1}{3}} t} \partial_x \theta\|_{L^1 L^2} \lesssim \nu^{-\frac{1}{2}} \|\omega\|_{X_2}\|\theta\|_{X_2}\lesssim \nu^{-\frac{1}{2}} \|\omega\|_{X_2}\|\theta\|_{X_2},
\end{align*}
which imply
\[
\|\e^{\delta \nu^{\frac{1}{3}} t} u \cdot \nabla \theta\|_{L^1 L^2} \lesssim \nu^{-\frac{5}{8}} \|\omega\|_{X_2}\|\theta\|_{X_2}.
\]

{\it Step 3. The estimate of $u^{(1)} \partial_x \theta_{a} + u^{(2)} (\partial_y + t \partial_x) \theta_{a}$.}  From Proposition~\ref{p2.1},  one gets
\begin{align*}
\|u^{(2)} (\partial_y + t \partial_x) \theta_{a}\| &\leq \|u^{(2)}\| \|(\partial_y + t \partial_x) \theta_{a}\|_{L^\infty} \lesssim \|\omega\| \e^{-2\nu^{\frac{1}{3}} t} \|\Theta_{in}\|_{H^4}, \\
\|u^{(1)} \partial_x \theta_{a}\| &\leq \|u^{(1)}\| \|\partial_x \theta_{a}\|_{L^\infty} \lesssim \|\omega\| \e^{-2\nu^{\frac{1}{3}} t} \|\Theta_{in}\|_{H^4},
\end{align*}
which means that
\begin{align*}
\|\e^{\delta \nu^{\frac{1}{3}} t} (u^{(1)} \partial_x \theta_{a} + u^{(2)} (\partial_y + t \partial_x) \theta_{a})\|_{L^1 L^2}&\lesssim \|\omega\|_{L^\infty L^2} \|\e^{-\nu^{\frac{1}{3}} t}\|_{L^1(T_0,\infty)} \|\Theta_{in}\|_{H^4} \\
&  \lesssim \nu^{-\frac{1}{3}} \|\Theta_{in}\|_{H^4} \|\omega\|_{X_2}.
\end{align*}

{\it Step 4. The estimate of $Er^2$.} By applying \eqref{2.15}, we obtain
\begin{align*}
\|\e^{\delta \nu^{\frac{1}{3}} t} Er^2\|_{L^1 L^2} \lesssim& \|\nu (1+t) \e^{-\nu^{\frac{1}{3}} t}\|_{L^1(T_0,\infty)} \|\Theta_{in}\|_{H^4} \\ 
&+ \|(1+t)^{-2}\|_{L^1(T_0,\infty)} \|\Theta_{in}\|_{H^4}(\|\Omega_{in}\|_{H^3}+\nu^{-\frac{1}{3}}\|\Theta_{in}\|_{H^4}) \\
\lesssim& \nu^{\frac{1}{3}} \|\Theta_{in}\|_{H^4} +\|\Theta_{in}\|_{H^4}(\|\Omega_{in}\|_{H^3}+\nu^{-\frac{1}{3}}\|\Theta_{in}\|_{H^4})\lesssim \nu^{\frac{1}{3}} \|\Theta_{in}\|_{H^4} .
\end{align*}
By combining the above formulas, we get \eqref{h2}.

Similarly, we can also prove \eqref{h1}.
The proof of this proposition is completed.

\end{proof}

Next, we consider the estimate of the reaction term.

\begin{prop}\label{p4.12a}
Let $\omega_{re}$ and $\theta_{re}$ be the solution of \eqref{aaa} and \eqref{bbb} respectively, we have 
\begin{align}
&\|\omega_{re}\|_{X_2} \lesssim \nu^{-\frac{1}{3}} \|\e^{\delta \nu^{\frac{1}{3}} t} \nabla u^{(2)}\|_{L^2 L^2} (\|\Omega_{in}\|_{H^3}+\nu^{-\frac{1}{3}} \|\Theta_{in}\|_{H^4}),\label{4.34a}\\
&\|\theta_{re}\|_{X_2} \lesssim \nu^{-\frac{1}{3}} \|\e^{\delta \nu^{\frac{1}{3}} t} \nabla u^{(2)}\|_{L^2 L^2} \|\Theta_{in}\|_{H^4}.\label{4.35a}
\end{align}
\end{prop}

We only prove \eqref{4.35a}, and \eqref{4.34a} can be proved similarly. 
Similar to \cite{wei-zhang2026}, we first assume that $f_1$ satisfies
\[
\partial_t f_1 + y \partial_x f_1 = t \e^{\nu^{\frac{1}{3}} t} u^{(2)} \partial_x \theta_{a}, \quad f_1|_{t=T_0} = 0.
\]
Due to $\theta_{a}|_{y=\pm 1} = 0$, we have $f_1|_{y=\pm 1} = 0$. 
Then assume that $f_2$ satisfies
\[
\partial_t f_2 + V_1 \partial_x f_2- \nu \Delta f_2  = \nu \e^{-\nu^{\frac{1}{3}} t/2} \Delta f_1, \quad f_2(t,x,\pm 1) = f_2|_{t=T_0} = 0.
\]
As in \cite{wei-zhang2026}, we have
\begin{equation}\label{4.36a}
\|f_2\|_Y \lesssim \|\e^{-\nu^{\frac{1}{3}} t/2} f_1\|_Y,
\end{equation}
here the norm $\|\cdot\|_{Y}$ is defined as
\begin{equation}\label{4.37a}
    \|f\|_{Y} = \|f\|_{L^\infty L^2} + \nu^{\frac{1}{2}} \|\nabla f\|_{L^2 L^2} 
+ \|\partial_x \nabla \Delta^{-1} f\|_{L^2 L^2}. 
\end{equation}
Set $f_3 = \e^{-\nu^{\frac{1}{3}} t/2} f_2 + \e^{-\nu^{\frac{1}{3}} t} f_1$, then it is not difficult to verify that  $f_3|_{y=\pm 1} = f_3|_{t=T_0} = 0$, and
\begin{align*}
\partial_t f_3 &- \nu \Delta f_3 + V_1 \partial_x f_3 + \nu^{\frac{1}{3}} \e^{-\nu^{\frac{1}{3}} t/2} f_2/2 + \nu^{\frac{1}{3}} \e^{-\nu^{\frac{1}{3}} t} f_1 \\
&= t u^2 \partial_x \theta_{a} + \e^{-\nu^{\frac{1}{3}} t} (V_1 - y) \partial_x f_1.
\end{align*}
Set $f_4 = \theta_{re} - f_3$, then $f_4$ satisfies
\[
\left\{
\begin{aligned}
&\partial_t f_4 + V_1\partial_x f_4 - \nu \Delta f_4 = \nu^{\frac{1}{3}} \e^{-\nu^{\frac{1}{3}} t/2} f_2 / 2 + \nu^{\frac{1}{3}} \e^{-\nu^{\frac{1}{3}} t} f_1 - \e^{-\nu^{\frac{1}{3}} t} (V_1 - y) \partial_x f_1, \\
&f_4(t,x,\pm 1) = f_4|_{t=T_0} = 0.
\end{aligned}
\right.
\]

We present the following two lemmas concerning bound of the $X_2$ norm. The proofs can be found in \cite{wei-zhang2026}. For the sake of brevity, we omit the details here.
\begin{lem}\label{l4.13}
Suppose that $0<\delta\ll 1$, $f$ is a given function,  there holds 
\[
\|f\|_{X_2} \lesssim \|\e^{\nu^{\frac{1}{3}} t/2} f\|_Y + \nu^{\frac{1}{2}} \|\e^{\delta \nu^{\frac{1}{3}} t} \partial_x f\|_{L^1 L^2}.
\]
\end{lem}

\begin{lem}\label{l4.14}
For $f_4$, we can deduce
\[
\|f_4\|_{X_2} \lesssim \|\e^{-\nu^{\frac{1}{3}} t/2} f_1\|_Y + \|\partial_x f_1\|_{L^\infty L^2}.
\]
\end{lem}
Based on the above two lemmas, we can now give the bound of $\|\theta_{re}\|_{X_2}$.
\begin{lem}\label{l4.15}
   Suppose $\delta$ is defined in Proposition \ref{p4.8a}, there holds
\begin{equation}\label{4.38}
\|\theta_{re}\|_{X_2} \lesssim \|\e^{-\nu^{\frac{1}{3}}t/2}f_1\|_Y + \|\partial_x f_1\|_{L^\infty L^2} + \nu^{\frac{1}{6}}\|\partial_x \theta_{re}\|_{X_2}.
\end{equation}
\end{lem}
\begin{proof}
 By applying Lemma \ref{l4.13} shows
\[
\|f_3\|_{X_2} \lesssim \|\e^{\nu^{\frac{1}{3}}t/2}f_3\|_Y + \nu^{\frac{1}{2}}\|\e^{\delta\nu^{\frac{1}{3}}t}\partial_x f_3\|_{L^1 L^2}.
\]
It follows from \eqref{4.36a} that
\[
\|\e^{\nu^{\frac{1}{3}}t/2}f_3\|_Y \leq \|f_2\|_Y + \|\e^{-\nu^{\frac{1}{3}}t/2}f_1\|_Y \lesssim \|\e^{-\nu^{\frac{1}{3}}t/2}f_1\|_Y.
\]
Thus, one gets
\[
\nu^{\frac{1}{2}}\|\e^{\delta\nu^{\frac{1}{3}}t}\partial_x f_3\|_{L^1 L^2} \leq \nu^{\frac{1}{2}}\|\e^{\delta\nu^{\frac{1}{3}}t}\partial_x f_4\|_{L^1 L^2} + \nu^{\frac{1}{2}}\|\e^{\delta\nu^{\frac{1}{3}}t}\partial_x \theta_{re}\|_{L^1 L^2} \leq \|f_4\|_{X_2} + \nu^{\frac{1}{6}}\|\partial_x \theta_{re}\|_{X_2},
\]
and then,
\[
\|f_3\|_{X_2} \lesssim \|\e^{-\nu^{\frac{1}{3}}t/2}f_1\|_Y + \|f_4\|_{X_2} + \nu^{\frac{1}{6}}\|\partial_x \theta_{re}\|_{X_2},
\]
which along with Lemma \ref{l4.14} gives \eqref{4.38}.
\end{proof}
\begin{lem}\label{l4.16}
Let $\theta_{re}$ be the solution of \eqref{bbb}, then there holds
\[
\|\partial_x\theta_{re}\|_{X_2} \lesssim \nu^{-\frac{1}{2}} \|\e^{\delta \nu^{\frac{1}{3}} t} \nabla u^{(2)}\|_{L^2 L^2} \|\Theta_{in}\|_{H^4}.
\]
\end{lem}

\begin{proof}
By calculating $\partial_x\eqref{bbb}$, 
\[
\partial_t \partial_x \theta_{re}  + V_1 \partial_x^2 \theta_{re}- \nu \Delta \partial_x \theta_{re} - t \partial_x (u^{(2)} \partial_x \theta_{a}) = 0, \quad \partial_x \theta_{re}|_{y=\pm 1} = \partial_x \theta_{re}|_{t=T_0} = 0.
\]
Using Proposition~\ref{p4.8a} gives
\[
\|\partial_x \theta_{re}\|_{X_2} \lesssim \|\e^{\delta \nu^{\frac{1}{3}} t} t\partial_x (u^{(2)} \partial_x \theta_{a})\|_{L^1 L^2}.
\]
Notice that  $P_0 u^{(2)} = 0$, $\|u^{(2)}\| \leq \|\partial_x u^{(2)}\| \leq \|\nabla u^{(2)}\|$, by Lemma \ref{l3.2}, we deduce
\begin{align*}
\|t \partial_x (u^{(2)} \partial_x \theta_{a})\| &\leq t \|\partial_x^2 \theta_{a}\|_{L^\infty} \|u^{(2)}\| + t \|\partial_x \theta_{a}\|_{L^\infty} \|\partial_x u^{(2)}\| \\
&\lesssim \nu^{-\frac{1}{3}} \e^{-\nu^{\frac{1}{3}} t} \|\Theta_{in}\|_{H^4} (\|u^{(2)}\| + \|\partial_x u^{(2)}\|) \lesssim \nu^{-\frac{1}{3}} \e^{-\nu^{\frac{1}{3}} t} \|\Theta_{in}\|_{H^4} \|\nabla u^{(2)}\|,
\end{align*}
and then,
\begin{align*}
\|\partial_x \theta_{re}\|_{X_2} &\lesssim \|\e^{\delta \nu^{\frac{1}{3}} t} t \partial_x (u^{(2)} \partial_x \theta_{a})\|_{L^1 L^2} \lesssim \nu^{-\frac{1}{3}} \|\Theta_{in}\|_{H^4} \|\e^{(\delta-1)\nu^{\frac{1}{3}} t} \nabla u^{(2)}\|_{L^1 L^2} \\
&\leq \nu^{-\frac{1}{3}} \|\Theta_{in}\|_{H^4} \|\e^{\delta \nu^{\frac{1}{3}} t} \nabla u^{(2)}\|_{L^2 L^2} \|\e^{-\nu^{\frac{1}{3}} t}\|_{L^2(T_0,\infty)} \\
&\leq \nu^{-\frac{1}{2}} \|\Theta_{in}\|_{H^4} \|\e^{\delta \nu^{\frac{1}{3}} t} \nabla u^{(2)}\|_{L^2 L^2}.
\end{align*}
The proof of Lemma \ref{l4.16} is completed.
\end{proof}

\begin{lem}\label{l4.17}
It holds that
\[
\|\e^{-\nu^{\frac{1}{3}} t/2} f_1\|_Y + \|\partial_x f_1\|_{L^\infty L^2} \lesssim \nu^{-\frac{1}{3}} \|\e^{\delta \nu^{\frac{1}{3}} t} \nabla u^{(2)}\|_{L^2 L^2} \|\Theta_{in}\|_{H^4}.
\]
\end{lem}
\begin{proof}
See Lemma 5.5 in \cite{wei-zhang2026}.
\end{proof}

Combining Lemmas \ref{l4.15}--\ref{l4.17}, we can get \eqref{4.35a}, and the proof of Proposition \ref{p4.12a} is finished.

Next, we are going to prove Proposition~\ref{p2.2}.

\begin{proof}[\textbf{Proof of Proposition~\ref{p2.2}}]
Thanks to \eqref{4.11} in Proposition~\ref{p4.8}, we derive
\[
\|\omega_*(T_0)\| = \|\omega(T_0)\| \lesssim \nu^{\frac{1}{3}} \|\Omega_{in}\|_{H^3}+\nu^{-\frac{1}{6}}\|\Theta_{in}\|_{H^4},
\]
\begin{equation}\label{4.43}
   \|\theta_*(T_0)\| = \|\theta(T_0)\| \lesssim \nu^{\frac{5}{6}} \|\Omega_{in}\|_{H^3}+\nu^{\frac{1}{3}}\|\Theta_{in}\|_{H^4}.
\end{equation}
Then it follows from \eqref{4.43}, Proposition~\ref{p4.8a} and Proposition~\ref{prop:2.6} that
\begin{align}\label{4.44}
\|\omega_*\|_{X_2} \lesssim& \|\omega_*(T_0)\| + \|\e^{\delta \nu^{\frac{1}{3}} t} h_1\|_{L^1 L^2} \notag\\
\lesssim& \nu^{-\frac{1}{3}} (\|\Omega_{in}\|_{H^3}+\nu^{-\frac{1}{3}}\|\Theta_{in}\|_{H^4}) \|\omega\|_{X_2}+ \nu^{-\frac{5}{8}}\|\omega\|_{X_2}^2 \notag\\
&+\nu^{-\frac{1}{2}} \|\theta\|_{X_2}+\nu^{\frac{1}{3}} \|\Omega_{in}\|_{H^3}+\nu^{-\frac{1}{6}}\|\Theta_{in}\|_{H^4},
\end{align}
\begin{align}\label{4.45}
\|\theta_*\|_{X_2} \lesssim& \|\theta_*(T_0)\| + \|\e^{\delta \nu^{\frac{1}{3}} t} h_2\|_{L^1 L^2} \notag\\
\lesssim& \nu^{-\frac{1}{3}} \|\Theta_{in}\|_{H^4} \|\omega\|_{X_2}+\nu^{-\frac{1}{3}} (\|\Omega_{in}\|_{H^3}+\nu^{-\frac{1}{3}}\|\Theta_{in}\|_{H^4}) \|\theta\|_{X_2}+ \nu^{-\frac{5}{8}}\|\omega\|_{X_2}\|\theta\|_{X_2} \notag\\
&+\nu^{\frac{5}{6}} \|\Omega_{in}\|_{H^3}+\nu^{\frac{1}{3}} \|\Theta_{in}\|_{H^4}.
\end{align}
Moreover, thanks to 
\[
\|\e^{\delta \nu^{\frac{1}{3}} t} \nabla u^2\|_{L^2 L^2} = \|\e^{\delta \nu^{\frac{1}{3}} t} \partial_x \nabla \Delta^{-1} \omega\|_{L^2 L^2} \leq \|\omega\|_{X_2},
\]
then using $\omega=\omega_*+\omega_{re}$, Proposition~\ref{p4.12a} and \eqref{4.44}--\eqref{4.45} gives
\begin{align}\label{4.46}
\|\omega\|_{X_2} \leq \|\omega_*\|_{X_2} + \|\omega_{re}\|_{X_2}\lesssim&\nu^{-\frac{1}{3}} (\|\Omega_{in}\|_{H^3}+\nu^{-\frac{1}{3}}\|\Theta_{in}\|_{H^4}) \|\omega\|_{X_2}+ \nu^{-\frac{5}{8}}\|\omega\|_{X_2}^2 \notag\\
&+\nu^{-\frac{1}{2}} \|\theta\|_{X_2}+\nu^{\frac{1}{3}} \|\Omega_{in}\|_{H^3}+\nu^{-\frac{1}{6}}\|\Theta_{in}\|_{H^4}, 
\end{align}
\begin{align}\label{4.47}
\|\theta\|_{X_2} \leq \|\theta_*\|_{X_2} + \|\theta_{re}\|_{X_2} \lesssim& \nu^{-\frac{1}{3}} \|\Theta_{in}\|_{H^4} \|\omega\|_{X_2}+\nu^{-\frac{1}{3}} (\|\Omega_{in}\|_{H^3}+\nu^{-\frac{1}{3}}\|\Theta_{in}\|_{H^4}) \|\theta\|_{X_2}\notag\\
&+ \nu^{-\frac{5}{8}} \|\omega\|_{X_2}\|\theta\|_{X_2} +\nu^{\frac{5}{6}} \|\Omega_{in}\|_{H^3}+\nu^{\frac{1}{3}} \|\Theta_{in}\|_{H^4}. 
\end{align}
By calculating $\nu^{-\frac{1}{2}}\cdot$\eqref{4.47}+$\bar{\varepsilon}\cdot$\eqref{4.46}~($0<\bar{\varepsilon}\ll 1$), and noting that the linear term $\nu^{-\frac{1}{2}} \|\theta\|_{X_2}$ is absorbed due to the smallness of $\bar{\varepsilon}$, we get
\begin{align}
  \|\omega\|_{X_2}+\nu^{-\frac{1}{2}}\|\theta\|_{X_2} \lesssim&\nu^{-\frac{1}{3}} (\|\Omega_{in}\|_{H^3}+\nu^{-\frac{1}{3}}\|\Theta_{in}\|_{H^4}) \|\omega\|_{X_2}+ \nu^{-\frac{5}{8}}\|\omega\|_{X_2}^2 \notag\\
&+\nu^{-\frac{5}{6}} \|\Theta_{in}\|_{H^4} \|\omega\|_{X_2}+\nu^{-\frac{5}{6}} (\|\Omega_{in}\|_{H^3}+\nu^{-\frac{1}{3}}\|\Theta_{in}\|_{H^4}) \|\theta\|_{X_2}\notag\\
&+ \nu^{-\frac{9}{8}} \|\omega\|_{X_2}\|\theta\|_{X_2}+\nu^{\frac{1}{3}} \|\Omega_{in}\|_{H^3}+\nu^{-\frac{1}{6}}\|\Theta_{in}\|_{H^4}  .\notag
\end{align}
Then using the condition \eqref{1.4}, we arrive at
\begin{align}\label{4.48}
  \|\omega\|_{X_2}+\nu^{-\frac{1}{2}}\|\theta\|_{X_2} \leq&C(\nu^{-\frac{5}{8}}\|\omega\|_{X_2}^2 + \nu^{-\frac{9}{8}} \|\omega\|_{X_2}\|\theta\|_{X_2}+\nu^{\frac{1}{3}} \|\Omega_{in}\|_{H^3}+\nu^{-\frac{1}{6}}\|\Theta_{in}\|_{H^4})  .
\end{align}
With the energy estimate \eqref{4.48} at hand, we are ready to set up the bootstrap argument.
For $T_0\leq t\leq T$, we first assume that $\|\omega\|_{X_2} \leq \varepsilon_1 \nu^{\frac{2}{3}}, \|\theta\|_{X_2} \leq \varepsilon_1 \nu^{\frac{7}{6}}$ for a constant $0<\varepsilon_1\ll 1$ which will be determined later. Then by \eqref{4.48}, we have
\begin{equation}\label{4.49}
\|\omega\|_{X_2}+\nu^{-\frac{1}{2}}\|\theta\|_{X_2} \leq C \nu^{\frac{1}{3}} \|\Omega_{in}\|_{H^3}+C\nu^{-\frac{1}{6}} \|\Theta_{in}\|_{H^4} \leq C \varepsilon_0 \nu^{\frac{2}{3}}.
\end{equation}
Taking $\varepsilon_0 = \frac{\varepsilon_1 }{2C}$, then one can get \eqref{4.49} for all $t\geq T_0$. Combining Proposition~\ref{p4.8} and \eqref{4.49}, and noticing the definition of $\|\cdot\|_{X_2} $ in \eqref{4.15aaa}, we obtain \eqref{222} and hence complete the proof of
Proposition \ref{p2.2}.
\end{proof}

\section{Proof of Proposition \ref{ap3.2}}\label{s5}
In this section, we give the proof of the Proposition \ref{ap3.2}. First, let's consider some basic estimates.

\subsection{Space-time estimates for linearized equation}\label{s3.3.1}

\begin{prop}\label{ap3.5}
Suppose that $u_{a}$ is defined as in \eqref{a3.3}, and $f$ solves
\begin{equation}\label{a3.50}
\left\{
\begin{aligned}
&\partial_t f - \nu \Delta f + y \partial_x f+u_{a}\cdot\nabla f = 0, \ \ \ \  t>0,\\
&f|_{t=0}=f_{in}(x,y),\quad f(t,x,\pm 1) = 0, \quad \nabla\cdot u_{a}=0,
\end{aligned}
\right.
\end{equation}
then for $0<t\leq \nu^{-\frac{1}{6}}$, it holds that 
\begin{align*}
&\|k f_k(t)\|_{L^2_y} \les (\nu t)^{-\frac{1}{2}} (1+t)^{-1} \|f_{in,k}\|_{L^2_y}, \\
&\|f_{k}(t)\|_{L^2_y} \les (1 +\nu t+ \nu k^2t^3)^{-\frac{1}{2}} \|f_{in,k}\|_{L^2_y}.
\end{align*}
Here $f_k$ is the Fourier transform of $f$ with respect to the variable $x$ and $k\in\mathbb{R}$.
\end{prop}
\begin{proof}
  The proof is similar to that of Proposition \ref{p4.1},  we omit the details here.
\end{proof}
Based on the above proposition, we immediately get the following corollary.
\begin{coro}\label{ac3.6}
    Under the conditions of Proposition \ref{ap3.5}, we have for $t\leq \nu^{-\frac{1}{6}}$ that
  \begin{align*}
\|f(t)\| + \nu^{\frac{1}{2}} \left(\int_0^t \|\nabla f(s)\|^2 \dd s\right)^{\frac{1}{2}} + \nu^{\frac{1}{2}}\int_0^t \|\partial_x f(s)\|\dd s\lesssim \|f_{in}\|.
\end{align*}
\end{coro}

\subsubsection{Short time scale $0\leq t\leq \nu^{-\frac{1}{6}}$}
We define the space-time norm $\|f\|_{X_1} $, where
\begin{align}\label{x1}
\|f\|_{X_1} &= \|f\|_{L^\infty L^2} + \nu^{\frac{1}{2}} \|\nabla f\|_{L^2 L^2}+\nu^{\frac{1}{2}} \| \partial_x f\|_{L^1 L^2} .
\end{align}
\begin{prop}\label{ap3.7}
Suppose that $(\omega,\theta)$ is the solution to \eqref{a3.5}, then for $0\leq t\leq T_0=\nu^{-\frac{1}{6}}$,
\begin{align}\label{a3.51}
  \|\theta\|_{X_1}\lesssim& \nu^{-\frac{1}{3}}\|\omega\|_{L^\infty L^2}\mathcal{E}_{2,in}+\nu^{-\frac{7}{12}}\|\omega\|_{L^\infty L^2}\|\theta\|_{X_1}\notag \\ 
  &+\nu^{-\frac{5}{8}}\|\omega\|_{L^\infty L^2}^{\frac{9}{10}}(\nu^{\frac{1}{2}}\|\nabla \omega\|_{L^2 L^2})^{\frac{1}{10}}\|\theta\|_{X_1}+\nu^{\frac13}\ln(1+\nu^{-1})\mathcal{E}_{2,in}.  
\end{align}

\end{prop}
\begin{proof} Notice that
\begin{equation}\notag
    \partial_t \theta - \nu \Delta \theta + y \partial_x \theta+u_{a}\cdot\nabla\theta=g_1, \ \theta|_{t=0}=0,
\end{equation}
where 
\begin{equation}\notag
    g_1=- u \cdot \nabla \theta_{a}- u \cdot \nabla \theta  -Er^2.
\end{equation}
From Corollary \ref{ac3.6} and Duhamel's principle, we have 
\begin{equation}\label{a3.52}
\|\theta\|_{X_1} \lesssim \left\| g_1 \right\|_{L^1 L^2}.
\end{equation}
Now we begin to estimate $\left\| g_1 \right\|_{L^1 L^2}$. 
Firstly, it follows from \eqref{a3.8}--\eqref{a3.9} that
\begin{align}
\|u \cdot \nabla \theta_{a}\| \lesssim \|u\| \|\nabla\theta_{a}\|_{L^\infty}\lesssim t\|\omega\|\mathcal{E}_{2,in},\notag
\end{align}
and then,
\begin{align}\label{a3.53}
\|u \cdot \nabla \theta_{a}\|_{L^1 L^2} \lesssim t^2\|\omega\|_{L^\infty L^2}\mathcal{E}_{2,in}\lesssim \nu^{-\frac{1}{3}}\|\omega\|_{L^\infty L^2}\mathcal{E}_{2,in}.
\end{align}
Next, notice that
\[
u \cdot \nabla \theta = u^{(1)} \partial_x \theta + u^{(2)} \partial_y \theta.
\]
Then 
\begin{align}
  \|u^{(1)} \partial_x \theta\|\lesssim \|u^{(1)}\|_{L^\infty}\|\partial_x \theta\|\lesssim \int_{\R}\|u^{(1)}_k\|_{L^\infty_y}\dd k \|\partial_x \theta\|.
\end{align}
For $|k|\leq 1$, we use Sobolev inequality to obtain
\begin{equation}\notag
  \int_{|k|\leq 1}\|u^{(1)}_k\|_{L^\infty_y}\dd k\leq \int_{|k|\leq 1}\|\omega_k\|_{L^2_y}\dd k\lesssim \|\omega_k\|_{L^2_kL^2_y}\lesssim \|\omega\|_{L^2}.  
\end{equation}
For $|k|\geq 1$, 
\begin{align}
  \int_{|k|\geq 1}\|u^{(1)}_k\|_{L^\infty_y}\dd k\leq& \int_{|k|\geq 1}\|u^{(1)}_k\|^{\frac{1}{2}}_{L^2_y}\|\partial_yu^{(1)}_k\|^{\frac{1}{2}}_{L^2_y}\dd k=\int_{|k|\geq 1}|k|^{-\frac{1}{2}-\frac{\delta_1}{2}}\|k^{1+\delta_1}u^{(1)}_k\|^{\frac{1}{2}}_{L^2_y}\|\partial_yu^{(1)}_k\|^{\frac{1}{2}}_{L^2_y}\dd k\notag \\ 
  \lesssim& \||k|^{-\frac{1}{2}-\frac{\delta_1}{2}}\|_{L^2_k(|k|\geq 1)}\|k^{1+\delta_1}u^{(1)}_k\|^{\frac{1}{2}}_{L^2_kL^2_y}\|\partial_yu^{(1)}_k\|^{\frac{1}{2}}_{L^2_kL^2_y}\notag \\ \lesssim& \|ku^{(1)}_k\|_{L^2_kL^2_y}^{\frac{1-\delta_1}{2}}\|k^2u^{(1)}_k\|_{L^2_kL^2_y}^{\frac{\delta_1}{2}}\|(k,\partial_y)u^{(1)}_k\|_{L^2_kL^2_y}^{\frac{1}{2}} 
  \lesssim \|\omega\|_{L^2}^{1-\frac{\delta_1}{2}}\|\nabla \omega\|_{L^2}^{\frac{\delta_1}{2}}. \notag 
\end{align}
Thus, we deduce
\begin{align}
  \|u^{(1)} \partial_x \theta\|_{ L^1L^2}\lesssim& \|\omega\|_{L^\infty L^2}\|\partial_x \theta\|_{L^1 L^2}+\left\|\|\omega\|_{L^2}^{1-\frac{\delta_1}{2}}\|\nabla \omega\|_{L^2}^{\frac{\delta_1}{2}}\right\|_{L^2_t}\|\partial_x \theta\|_{L^2 L^2}\notag\\
  \lesssim&\|\omega\|_{L^\infty L^2}\|\partial_x \theta\|_{L^1 L^2}+\|\omega\|_{L^2 L^2}^{1-\frac{\delta_1}{2}} \|\nabla \omega\|_{L^2 L^2}^{\frac{\delta_1}{2}}\|\nabla \theta\|_{L^2 L^2}\notag\\
  \lesssim&\|\omega\|_{L^\infty L^2}\|\partial_x \theta\|_{L^1 L^2}+\|\omega\|_{L^\infty L^2}^{1-\frac{\delta_1}{2}}t^{\frac{1}{2}(1-\frac{\delta_1}{2})}\|\nabla \omega\|_{L^2 L^2}^{\frac{\delta_1}{2}}\|\nabla \theta\|_{L^2 L^2}\notag\\ 
  \lesssim&\|\omega\|_{L^\infty L^2}\|\partial_x \theta\|_{L^1 L^2}+\nu^{-\frac{1}{12}+\frac{\delta_1}{24}}\|\omega\|_{L^\infty L^2}^{1-\frac{\delta_1}{2}}\|\nabla \omega\|_{L^2 L^2}^{\frac{\delta_1}{2}}\|\nabla \theta\|_{L^2 L^2}.\notag
\end{align}
Let $\delta_1=\frac{1}{5}$, then
\begin{align}\label{a3.55}
  \|u^{(1)} \partial_x \theta\|_{ L^1L^2}\lesssim& \nu^{-\frac{1}{2}}\|\omega\|_{L^\infty L^2}(\nu^{\frac{1}{2}}\|\partial_x \theta\|_{L^1 L^2})+\nu^{-\frac{5}{8}}\|\omega\|_{L^\infty L^2}^{\frac{9}{10}}(\nu^{\frac{1}{2}}\|\nabla \omega\|_{L^2 L^2})^{\frac{1}{10}}(\nu^{\frac{1}{2}}\|\nabla \theta\|_{L^2 L^2})\notag\\ 
  \lesssim& \nu^{-\frac{1}{2}}\|\omega\|_{L^\infty L^2}\|\theta\|_{X_1}+\nu^{-\frac{5}{8}}\|\omega\|_{L^\infty L^2}^{\frac{9}{10}}(\nu^{\frac{1}{2}}\|\nabla \omega\|_{L^2 L^2})^{\frac{1}{10}}\|\theta\|_{X_1}. 
\end{align}
For $\|u^{(2)} \partial_y \theta\|_{ L^1L^2}$, note that when $|k|\leq 1$,
\begin{equation}\notag
   \|\omega\|_{L^\infty L^2}\|\partial_y \theta\|_{L^1 L^2}\lesssim \|\omega\|_{L^\infty L^2} t^{\frac{1}{2}}\|\nabla \theta\|_{L^2 L^2}\lesssim \nu^{-\frac{7}{12}}\|\omega\|_{L^\infty L^2}(\nu^{\frac{1}{2}}\|\nabla \theta\|_{L^2 L^2}), 
\end{equation}
and when $|k|\geq 1$, we use a similar approach as before to calculate, and ultimately obtain
\begin{align}\label{a3.56}
  \|u^{(2)} \partial_y \theta\|_{ L^1L^2}\lesssim& \nu^{-\frac{7}{12}}\|\omega\|_{L^\infty L^2}(\nu^{\frac{1}{2}}\|\nabla \theta\|_{L^2 L^2})+\nu^{-\frac{5}{8}}\|\omega\|_{L^\infty L^2}^{\frac{9}{10}}(\nu^{\frac{1}{2}}\|\nabla \omega\|_{L^2 L^2})^{\frac{1}{10}}(\nu^{\frac{1}{2}}\|\nabla \theta\|_{L^2 L^2})\notag\\ 
  \lesssim& \nu^{-\frac{7}{12}}\|\omega\|_{L^\infty L^2}\|\theta\|_{X_1}+\nu^{-\frac{5}{8}}\|\omega\|_{L^\infty L^2}^{\frac{9}{10}}(\nu^{\frac{1}{2}}\|\nabla \omega\|_{L^2 L^2})^{\frac{1}{10}}\|\theta\|_{X_1}. 
\end{align}
For $\|Er^2\|_{L^1 L^2}$,  
thanks to \eqref{a3.18},
\begin{align}\label{estim Er}
&\int^{T_{0}}_0\|Er^2\|_{L^2}\dd t\lesssim \mathcal{E}_{2,in}\int^{T_{0}}_0(1+t)^{-1}\mathcal{E}_{2,in}+\nu+\nu^{\frac23} (1+\nu t^3)^{-\frac{5}{12}}\dd t\lesssim \nu^{\frac13}\ln(1+\nu^{-1})\mathcal{E}_{2,in}.
\end{align}
Substituting \eqref{a3.53}--\eqref{estim Er} into \eqref{a3.52} yields \eqref{a3.51}, which completes the proof of this proposition.

\end{proof}

Now, we give the uniform estimate of $(\omega,\theta)(x,t)$ in a short time scale.
\begin{prop}\label{ap3.8}
If $0 \leq t \leq T_0 = \nu^{-\frac{1}{6}}$, then
\begin{equation}\label{a3.58}
\|\omega(t)\|_{L^\infty L^2}+\nu^{\frac{1}{2}} \|\nabla \omega(t)\|_{L^2L^2}+\nu^{-\frac{1}{2}}\|\theta(t)\|_{X_1} \lesssim\nu^{\frac13}\ln (1+\nu^{-1})\mathcal{E}_{in},
\end{equation}
\begin{equation}\label{aa3.58}
\|(1+\nu^{\frac{1}{3}}t\mathcal{M}_1)^\delta\omega(t)\|_{L^\infty L^2}+\nu^{-\frac{1}{2}}\|(1+\nu^{\frac{1}{3}}t\mathcal{M}_1)^\delta\theta(t)\|_{L^\infty L^2} \lesssim\nu^{\frac13}\ln (1+\nu^{-1})\mathcal{E}_{in},
\end{equation}
\end{prop}

\begin{proof}
Firstly, we multiply \eqref{a3.5}$_1$ and \eqref{a3.5}$_2$ by $\omega$ and $\theta$ respectively, and integrate it with respect to $x$ and $y$ to obtain
\[
\frac{1}{2} \frac{\rm d}{{\rm d}t} \|\omega\|^2 + \nu \|\nabla \omega\|^2 = \int_{-1}^1\int_\mathbb{R} (- u \cdot \nabla \omega_{a} - Er^1-\partial_x \theta) \omega \dd x \dd y,
\]
\[
\frac{1}{2} \frac{\rm d}{{\rm d}t} \|\theta\|^2 + \nu \|\nabla \theta\|^2 = \int_{-1}^1\int_\mathbb{R} (-u \cdot \nabla \theta_{a} - Er^2) \theta \dd x \dd y.
\]
Using $u = (\partial_y, -\partial_x) \phi$ and $\phi = \Delta^{-1} \omega$, we have $\|u\| = \|\nabla \phi\| \leq C \|\omega\|$,
and
\begin{align*}
\frac{1}{2}\frac{\rm d}{{\rm d}t} \|\omega\|^2+\nu \|\nabla \omega\|^2
&\leq \|\omega\|\|u \cdot \nabla \omega_{a}+Er^1\|+\|\omega\|\| \partial_x\theta\| \\
&\leq \| \omega\|(\|u\|\|\nabla \omega_{a}\|_{L^\infty} + \|Er^1\|) +\|\omega\|\| \partial_x\theta\|\\
&\lesssim\|\omega\|^2\|\nabla \omega_{a}\|_{L^\infty} + \|\omega\|\|Er^1\| + \|\omega\|\| \partial_x\theta\|,
\end{align*}
\begin{align*}
\frac{1}{2}\frac{\rm d}{{\rm d}t} \|\theta\|^2+\nu \|\nabla \theta\|^2
&\leq \|\omega\|\|u \cdot \nabla \theta_{a} +Er^2\| \\
&\leq \| \theta\|(\|u\| \|\nabla \theta_{a}\|_{L^\infty} + \|Er^2\|) \\
&\lesssim\| \theta\|\|\omega\|\|\nabla \theta_{a}\|_{L^\infty} + \|\theta\|\|Er^2\| .
\end{align*}
Integrating the above equations with respect to $t$ yields
\begin{align}
  &\|\omega\|^2_{L^\infty L^2}+\nu \|\nabla \omega\|_{L^2L^2}^2 \notag\\ 
  &\leq C\|\omega\|_{L^\infty L^2}^2\|\nabla \omega_{a}\|_{L^1 L^\infty} + C\|\omega\|_{L^\infty L^2}(\|Er^1\|_{L^1L^2}+\|\partial_x\theta\|_{L^1L^2}),\notag  
\end{align}
\begin{align}
  &\|\theta\|^2_{L^\infty L^2}+\nu \|\nabla \theta\|_{L^2L^2}^2 \notag\\ 
  &\leq C\|\theta\|_{L^\infty L^2}\|\omega\|_{L^\infty L^2}\|\nabla \theta_{a}\|_{L^1 L^\infty} + C\|\theta\|_{L^\infty L^2}\|Er^2\|_{L^1L^2}.\notag  
\end{align}
From \eqref{a3.8}--\eqref{a3.9}, \eqref{omL Linfty}--\eqref{paxomL Linfty} and \eqref{a3.18}--\eqref{ErL2}, we arrive at 
 \begin{align}
  &\|\omega\|^2_{L^\infty L^2}+\nu \|\nabla \omega\|_{L^2L^2}^2 \notag\\ 
  \leq &C\nu^{-\frac13}\mathcal{E}_{in}\|\omega\|_{L^\infty L^2}^2 + C\nu^{\frac13}\ln (1+\nu^{-1}) \mathcal{E}_{in}\|\omega\|_{L^\infty L^2} +\|\omega\|_{L^\infty L^2}\|\partial_x\theta\|_{L^1L^2}   \notag\\
  \leq& \left(C\varepsilon_0+\frac{1}{2}\right)\|\omega\|_{L^\infty L^2}^2 +C\nu^{\frac23}\ln^2 (1+\nu^{-1}) \mathcal{E}_{in}^2+C\|\partial_x\theta\|_{L^1L^2}^2,\notag
\end{align}
\begin{align}
 \|\theta\|_{L^\infty L^2}^2 + \nu \|\nabla \theta\|_{L^2L^2}^2  
 \leq& C \nu^{-\frac13}\mathcal{E}_{2,in}\|\theta\|_{L^\infty L^2}\|\omega\|_{L^\infty L^2}+\nu^{\frac13}\ln (1+\nu^{-1})\mathcal{E}_{2,in}\|\theta\|_{L^\infty L^2}\notag\\ 
 \leq& \left(C\varepsilon_0+\frac{1}{2}\right)\|\theta\|_{L^\infty L^2}^2+C\varepsilon_0\nu\|\omega\|_{L^\infty L^2}^2+C\nu^{\frac23}\ln ^2(1+\nu^{-1})\mathcal{E}_{2,in}^2.\notag
\end{align}
By the smallness of $\varepsilon_0$, we obtain 
 \begin{align}\notag
  \|\omega\|_{L^\infty L^2}+\nu^{\frac{1}{2}} \|\nabla \omega\|_{L^2L^2} 
  \leq C\left[\nu^{-\frac{1}{2}}\|\theta\|_{X_1}+\nu^{\frac13}\ln (1+\nu^{-1}) \mathcal{E}_{in}\right],
\end{align}
\begin{align}\notag
 \|\theta\|_{L^\infty L^2} + \nu^{\frac{1}{2}} \|\nabla \theta\|_{L^2L^2}  
 \leq C\varepsilon_0^{\frac{1}{2}}\nu^{\frac{1}{2}}\|\omega\|_{L^\infty L^2}+C\nu^{\frac13}\ln (1+\nu^{-1})\mathcal{E}_{2,in}.
\end{align}
Notice that 
\begin{align}\notag
  \nu^{-\frac{1}{2}}\|\theta\|_{X_1}\leq& C\nu^{-\frac{5}{6}}\|\omega\|_{L^\infty L^2}\mathcal{E}_{2,in}+C\nu^{-\frac{13}{12}}\|\omega\|_{L^\infty L^2}\|\theta\|_{X_1}\notag \\ 
  &+C\nu^{-\frac{9}{8}}\|\omega\|_{L^\infty L^2}^{\frac{9}{10}}(\nu^{\frac{1}{2}}\|\nabla \omega\|_{L^2 L^2})^{\frac{1}{10}}\|\theta\|_{X_1}+C\nu^{-\frac16}\ln(1+\nu^{-1})\mathcal{E}_{2,in},\end{align}
then  combining the above three equations shows
\begin{align}\notag
&\|\omega\|_{L^\infty L^2}+\nu^{\frac{1}{2}} \|\nabla \omega\|_{L^2L^2}+\nu^{-\frac{1}{2}} \|\theta\|_{L^\infty L^2}+\nu^{-\frac{1}{2}}\|\theta\|_{X_1}\notag\\ 
\lesssim&\nu^{-\frac{13}{12}}\|\omega\|_{L^\infty L^2}\|\theta\|_{X_1}+\nu^{-\frac{9}{8}}\|\omega\|_{L^\infty L^2}^{\frac{9}{10}}(\nu^{\frac{1}{2}}\|\nabla \omega\|_{L^2 L^2})^{\frac{1}{10}}\|\theta\|_{X_1}+\nu^{\frac13}\ln (1+\nu^{-1}) \mathcal{E}_{in}.\notag
\end{align}
Assume that $\|\omega\|_{L^\infty L^2}+\nu^{\frac{1}{2}} \|\nabla \omega\|_{L^2L^2}+\nu^{-\frac{1}{2}}\|\theta\|_{X_1}\leq \varepsilon_1\nu^{\frac{2}{3}}\ln (1+\nu^{-1})$ for some small constant $\varepsilon_1$, then we conclude
\begin{align}\notag
\|\omega\|_{L^\infty L^2}+\nu^{\frac{1}{2}} \|\nabla \omega\|_{L^2L^2}+\nu^{-\frac{1}{2}}\|\theta\|_{X_1}\leq C\nu^{\frac13}\ln (1+\nu^{-1}) \mathcal{E}_{in}.
\end{align}
Notice that $\mathcal{E}_{in}\leq \varepsilon_0 \nu^{\frac{1}{3}},$
Then for $0 \leq t \leq\nu^{-\frac{1}{6}}$,  we have
\[
\|\omega\|_{L^\infty L^2}+\nu^{\frac{1}{2}} \|\nabla \omega\|_{L^2L^2}+\nu^{-\frac{1}{2}}\|\theta\|_{X_1} \leq C\nu^{\frac13}\ln (1+\nu^{-1})\mathcal{E}_{in} \leq C\varepsilon_0\nu^{\frac23}\ln (1+\nu^{-1}).
\]
Assume $\varepsilon_0\leq \frac{\varepsilon_1}{2C}$, we get \eqref{a3.58}. 
Then by $(1+\nu^{\frac{1}{3}}t|\mathcal{M}_1|)^\delta\lesssim 1$ and the definition of $\|\cdot\|$ in \eqref{x1}, we can also get \eqref{aa3.58}
The proof of this proposition is completed.
\end{proof}

\subsubsection{Long time scale $t\geq \nu^{-\frac{1}{6}}$}

With the help of estimates in a short time scale, we only need to establish the uniform estimates of the solution in a long time scale $[T_0, T]$ for every $T > T_0=\nu^{-\frac{1}{6}}$. Firstly, we consider the space-time estimates of the following system:
\begin{equation}\label{a3.61}
    \partial_t f + y \partial_x f- \nu \Delta f  = g_2, \quad f(t,x,\pm 1) = 0, \quad \text{for }\  t \geq T_0.
\end{equation}
Define the space-time norms $\|f\|_{X_2}$ and $\|f\|_{X_{2,\delta}}$, namely
\begin{align}\label{a3.62}
\|f\|_{X_2} := \|f\|_{L^\infty L^2} + \nu^{\frac{1}{2}} \|\nabla f\|_{L^2 L^2} + \|\mathcal{M}_2 \nabla \Delta^{-1} f\|_{L^2 L^2} + \nu^{\frac{1}{2}} \|\partial_x f\|_{L^1 L^2},
\end{align}
\begin{align}\label{a3.63}
\|f\|_{X_{2,\delta}}:= \|(1+\nu^{\frac{1}{3}}t\mathcal{M}_1)^{\delta}f\|_{X_2}.
\end{align}
\begin{prop}\label{ap3.9}
Assume that $f$ is the solution to \eqref{a3.61} for $t\in [t_0,t_1]$, then, there exists a suitable small constant $\varepsilon$ such that
\begin{align}
\|f\|_{X_2}\lesssim &\|f(t_0)\|_{L^2}+\| g_2\|_{L^1L^2},\label{esti: ftildeX} \\
\nu^{\frac13+\varepsilon}\|\mathcal{M}_1 f\|_{L^1 L^2}\lesssim_{\varepsilon} &\|f(t_0)\|_{L^2}+\|g_2\|_{L^1 L^2}.\label{f enhance}
\end{align}
Moreover, if $t_0\geq \nu^{-1-\varepsilon}$, we have
\begin{align}
\|f\|_{X_{2,\delta}}\lesssim &\|(1+\nu^{\frac13} t_0 \mathcal{M}_1)^{\delta}f(t_0)\|_{L^2}+\|(1+ \nu^{\frac13}t\mathcal{M}_1)^{\delta}g_2\|_{L^1 L^2},\label{fX large time}\\
\nu^{\frac13}\|\mathcal{M}_1(1+ \nu^{\frac13}t\mathcal{M}_1)^{\delta}f\|_{L^1L^2}\lesssim &\|(1+\nu^{\frac13} t_0 \mathcal{M}_1)^{\delta}f(t_0)\|_{L^2}+\|(1+ \nu^{\frac13}t\mathcal{M}_1)^{\delta}g_2\|_{L^1 L^2}.\label{esti: nu1/3M f}
\end{align}
\end{prop}
\begin{proof}
See \cite{miao2025} for details.

\end{proof}

To apply Proposition~\ref{ap3.9}, the main difficulty lies in the reaction part of $u \cdot \nabla \theta_{a}$. Motivated by \cite{wei-zhang2026, miao2025}, we introduce the reaction part $\omega_{re}$, which solves
\[
\partial_t \omega_{re} - \nu \Delta \omega_{re} + y \partial_x \omega_{re} - t u^{(2)} \partial_x \omega_{a} = 0, \quad \omega_{re}(t,x,\pm 1) = 0, \quad \omega_{re}|_{t=T_0} = 0,
\]
\[
\partial_t \theta_{re} - \nu \Delta \theta_{re} + y \partial_x \theta_{re} - t u^{(2)} \partial_x \theta_{a} = 0, \quad \theta_{re}(t,x,\pm 1) = 0, \quad \theta_{re}|_{t=T_0} = 0.
\]
Then let $\omega_* = \omega - \omega_{re}$, $\theta_* = \theta - \theta_{re}$, which satisfies
\[
\partial_t \omega_* - \nu \Delta \omega_* + y \partial_x \omega_* = h_1 , \quad \omega_*(t,x,\pm 1) = 0, \quad \omega_*|_{t=T_0} = \omega|_{t=T_0},
\]
\[
\partial_t \theta_* - \nu \Delta \theta_* + y \partial_x \theta_*  = h_2, \quad \theta_*(t,x,\pm 1) = 0, \quad \theta_*|_{t=T_0} = \theta|_{t=T_0},
\]
where
\begin{align}
&h_1 = - U \cdot \nabla \omega - u^{(1)} \partial_x \omega_{a} - u^{(2)} (\partial_y + t \partial_x) \omega_{a} - Er^1-\partial_x \theta ,\label{a3.75}\\
    &h_2 =- U \cdot \nabla \theta - u^{(1)} \partial_x \theta_{a} - u^{(2)} (\partial_y + t \partial_x)\theta_{a}   - Er^2.\label{a3.76}
\end{align}

For $(\omega_*,\theta_*)$, we establish the following uniform estimate.
\begin{prop}\label{ap3.12}
Under the conditions of Theorem \ref{thm1.2}, there exists a constant $0<\varepsilon\ll 1$ such that for $j_0=\log_2 \nu^{-\frac23-\varepsilon}$, 
\begin{align}
\|\omega_*\|_{X_2(T_0,T_{j_0})}\lesssim&  \nu^{-\frac13}\mathcal{E}_{in}\|\omega\|_{X_2(T_0,T_{j_0})}+\nu^{-\frac58} \|\omega\|^2_{X_2(T_0,T_{j_0})} +\nu^{\frac13-\varepsilon}\mathcal{E}_{in}+\nu^{-\frac12}\|\theta\|_{X_2(T_0,T_{j_0})}, \label{esti om1 tildeX}\\
\|\omega_*\|_{X_{2,\delta}(T_{j_0},T)}\lesssim& \|(1+ \nu^{\frac13}T_{j_0} \mathcal{M}_1)^{\delta}\omega_*(T_{j_0})\|_{L^2}+\mathcal{E}_{in}\|\omega\|_{X_{2,\delta}(T_{j_0},T)} \notag\\ 
&+\nu^{-\frac58}\|\omega\|^2_{X_{2,\delta}(T_{j_0},T)}+\nu^{\frac13}\mathcal{E}_{in}+\nu^{-\frac12}\|\theta\|_{X_{2,\delta}(T_{j_0},T)},\label{esti om1 X large}\\
\|\theta_*\|_{X_2(T_0,T_{j_0})}\lesssim&  \nu^{-\frac13}\mathcal{E}_{2,in}\|\omega\|_{X_2(T_0,T_{j_0})}+\nu^{-\frac13}\mathcal{E}_{in}\|\theta\|_{X_2(T_{0},T_{j_0})}+\nu^{-\frac58}\|\omega\|_{X_2(T_0,T_{j_0})}\|\theta\|_{X_2(T_0,T_{j_0})} \notag\\ 
&+\nu^{\frac56-\varepsilon}\mathcal{E}_{in}, \label{a3.79}\\
\|\theta_*\|_{X_{2,\delta}(T_{j_0},T)}\lesssim &\|(1+ \nu^{\frac13}T_{j_0} \mathcal{M}_1)^{\delta}\theta_*(T_{j_0})\|_{L^2}+\mathcal{E}_{2,in}\|\omega\|_{X_{2,\delta}(T_{j_0},T)}+\nu^{\frac12}\mathcal{E}_{in}\|\theta\|_{X_{2,\delta}(T_{j_0},T)}\notag\\ 
&+\nu^{-\frac58}\|\omega\|_{X_{2,\delta}(T_{j_0},T)}\|\theta\|_{X_{2,\delta}(T_{j_0},T)}+\nu^{\frac13}\mathcal{E}_{2,in}.\label{a3.80}
\end{align}
\end{prop}
\begin{proof}
We only give the proof of \eqref{a3.79}--\eqref{a3.80}, and the proof of \eqref{esti om1 tildeX}--\eqref{esti om1 X large} can be given similarly. By using \eqref{esti: ftildeX} with $t_0=T_0$ and Proposition \ref{ap3.8}, we infer 
\begin{align*}
\|\theta_*\|_{X_2(T_0,T_{j_0})}\lesssim \|\theta_*(T_0)\|_{L^2}+\|h_2\|_{L^1(T_0,T_{j_0};L^2)}\lesssim  \nu^{\frac56-\varepsilon}\mathcal{E}_{in}+\|h_2\|_{L^1(T_0,T_{j_0};L^2)}.
\end{align*}
Then we use \eqref{fX large time} with $t_0=T_{j_0}=2^{j_0}\nu^{-\frac13}=\nu^{-1-\varepsilon}$ to get
\begin{align*}
\|\theta_*\|_{X_{2,\delta}(T_{j_0},T)}\lesssim \|(1+ \nu^{\frac13}T_{j_0} \mathcal{M}_1)^{\delta}\theta_*(T_{j_0})\|_{L^2}+ \|(1+\nu^{\frac13}t\mathcal{M}_1)^{\delta}h_2\|_{L^1(T_{j_0},T; L^2)}.
\end{align*}
In the following, we only need to show
\begin{align}
\|h_2\|_{L^1(T_0,T_{j_0};L^2)}\lesssim & \nu^{-\frac13}\mathcal{E}_{2,in}\|\omega\|_{X_2(T_0,T_{j_0})}+\nu^{-\frac13}\mathcal{E}_{in}\|\theta\|_{X_2(T_{0},T_{j_0})}\notag\\ 
&+\nu^{-\frac58}\|\omega\|_{X_2(T_0,T_{j_0})}\|\theta\|_{X_2(T_0,T_{j_0})}+\nu^{\frac56-\varepsilon}\mathcal{E}_{in},\label{a3.81}\\
\|(1+\nu^{\frac13}t\mathcal{M}_1)^{\delta}h_2\|_{L^1 (T_{j_0},T; L^2) }\lesssim &\mathcal{E}_{2,in}\|\omega\|_{X_{2,\delta}(T_{j_0},T)}+\nu^{\frac12}\mathcal{E}_{in}\|\theta\|_{X_{2,\delta}(T_{j_0},T)}\notag\\ 
&+\nu^{-\frac58}\|\omega\|_{X_{2,\delta}(T_{j_0},T)}\|\theta\|_{X_{2,\delta}(T_{j_0},T)}+\nu^{\frac13}\mathcal{E}_{2,in}.\label{a3.82}
\end{align}

{\it Step 1. The estimate of $u^{(1)}\partial_x\theta_{a}+u^{(2)}(\partial_y+t\partial_x)\theta_{a}$}.
For $t\in (T_{j_0}, T)$, we use Proposition \ref{linear} and $\|u\|\lesssim \|\omega\|$ to obtain
\begin{equation*}
\begin{aligned}
&\|(1+\nu^{\frac13}t\mathcal{M}_1)^{\delta}(u^{(1)}\partial_x\theta_{a}+u^{(2)}(\partial_y+t\partial_x)\theta_{a})\| \\
\lesssim&\|(1+\nu^{\frac13}t\mathcal{M}_1)^{\delta}u\|(\|(1+\nu^{\frac13}t\mathcal{M}_1)^{\delta}\partial_x\theta_{a}\|_{L^\infty} +\|(1+\nu^{\frac13}t\mathcal{M}_1)^{\delta}(\partial_y+t\partial_x)\theta_{a}\|_{L^\infty}).
\end{aligned}
\end{equation*}
Integrating in time, we deduce
\begin{align}\label{weight-u1paxomL}
&\|(1+\nu^{\frac13}t\mathcal{M}_1)^{\delta}(u^{(1)}\partial_x\theta_{a}+u^{(2)}(\partial_y+t\partial_x)\theta_{a})\|_{L^1(T_{j_0},T; L^2) } \notag\\
\lesssim& \|(1+\nu^{\frac13}t\mathcal{M}_1)^{\delta}\omega\|_{L^\infty(T_{j_0},T; L^2)}\|(1+\nu t^3)^{-\frac12}\|_{L^1(T_{j_0},T)}\mathcal{E}_{2,in}\notag\\
\lesssim& \mathcal{E}_{2,in}\|\omega\|_{X_{2,\delta}(T_{j_0},T)}.
\end{align}

For $t\in (T_0, T_{j_0})$, note that
\begin{align*}
\|(1+\nu t^3)^{-\frac12}\|_{L^1(T_0,T_{j_0})}\lesssim \nu^{-\frac13}.
\end{align*}
Then we get that
\begin{align}\label{u1epaxom+u2payom}
\|u^{(1)}\partial_x\theta_{a}+u^{(2)}(\partial_y+t\partial_x)\theta_{a}\|_{L^1(T_{0},T_{j_0}; L^2)}\lesssim \nu^{-\frac13}\mathcal{E}_{2,in}\|\omega\|_{X_2(T_{0},T_{j_0})}.
\end{align}

{\it Step 2. The estimate of $U\cdot\nabla \theta$}.
Notice that
$$U\cdot\nabla \theta=u\cdot\nabla \theta+u_{a}\cdot\nabla \theta.$$
For $u_{a}\cdot\nabla \theta$, we infer
\begin{align}\label{uLomeL1L2}
&\|(1+\nu^{\frac13}t\mathcal{M}_1)^{\delta}(u_{a}\cdot\nabla \theta)\|_{ L^1(T_{j_0},T; L^2)}\notag\\
\lesssim &\|(1+\nu^{\frac13}t\mathcal{M}_1)^{\delta}(u_{a}^{(1)}\partial_x \theta)\|_{ L^1(T_{j_0},T; L^2)}+\|(1+\nu^{\frac13}t\mathcal{M}_1)^{\delta}(u_{a}^{(2)}\partial_y \theta)\|_{L^1(T_{j_0},T; L^2)}\notag\\
\lesssim &\|(1+\nu^{\frac13}t\mathcal{M}_1)^{\delta}u_{a}^{(1)}\|_{L^\infty(T_{j_0},T; L^\infty)}\|(1+\nu^{\frac13}t\mathcal{M}_1)^{\delta}\partial_x \theta\|_{L^1(T_{j_0},T; L^2)}\notag\\
&+\|(1+\nu^{\frac13}t\mathcal{M}_1)^{\delta}u_{a}^{(2)}\|_{L^2(T_{j_0},T;L^\infty)}\|(1+\nu^{\frac13}t\mathcal{M}_1)^{\delta}\partial_y \theta\|_{L^2(T_{j_0},T; L^2)}.
\end{align}
Using \eqref{u1LLinfty}--\eqref{u12L22}, we obtain
\begin{equation*}
\begin{aligned}
&\|(1+\nu^{\frac13}t\mathcal{M}_1)^{\delta}u_{a}^{(1)}\|_{L^\infty(T_{j_0},T; L^\infty)}\lesssim T^{-1}_{j_0}\mathcal{E}_{in}\lesssim \nu \mathcal{E}_{in},\\
&\|(1+\nu^{\frac13}t\mathcal{M}_1)^{\delta}u_{a}^{(2)}\|_{L^2(T_{j_0},T;L^\infty)}\lesssim \|(1+t)^{-2}\ln(1+t)\|_{L^2(T_{j_0}, T)} \mathcal{E}_{in}\lesssim \nu^{\frac32}\mathcal{E}_{in},
\end{aligned}
\end{equation*}
which together with \eqref{uLomeL1L2} gives
\begin{align}\label{weight-uLnaome}
\|(1+\nu^{\frac13}t\mathcal{M}_1)^{\delta}(u_{a}\cdot\nabla \theta)\|_{ L^1(T_{j_0},T; L^2)}
\lesssim \nu^{\frac12}\mathcal{E}_{in}\|\theta\|_{X_{2,\delta}(T_{j_0},T)}.
\end{align}
When $t\in (T_0,T_{j_0})$, 
\begin{equation*}
\begin{aligned}
&\|u_{a}^{(1)}\|_{L^\infty(T_0,T_{j_0}; L^\infty)}\lesssim T^{-1}_{0}\mathcal{E}_{in}\lesssim \nu^{\frac16} \mathcal{E}_{in},\\
&\|u_{a}^{(2)}\|_{L^2(T_0,T_{j_0};L^\infty)}\lesssim \|(1+t)^{-2}\ln(1+t)\|_{L^2(T_{0}, T_{j_0})} \mathcal{E}_{in}\lesssim \nu^{\frac16}\mathcal{E}_{in},
\end{aligned}
\end{equation*}
and then,
\begin{align}\label{uLnaome}
\|u_{a}\cdot\nabla \theta\|_{ L^1(T_{0},T_{j_0}; L^2)}
\lesssim \nu^{-\frac13}\mathcal{E}_{in}\|\theta\|_{X_2(T_{0},T_{j_0})}.
\end{align}
For the estimate of  $u\cdot\nabla \theta$, we get
\begin{align*}
&\|(1+\nu^{\frac13}t\mathcal{M}_1)^{\delta}(u\cdot\nabla \theta)\|_{L^1(T_{j_0},T; L^2)}\\
\lesssim &\|(1+\nu^{\frac13}t\mathcal{M}_1)^{\delta}(u^{(1)} \partial_x\theta)\|_{L^1(T_{j_0},T; L^2)}+\|(1+\nu^{\frac13}t\mathcal{M}_1)^{\delta}(u^{(2)}\partial_y \theta)\|_{L^1(T_{j_0},T; L^2)}\\
\lesssim &\big\|\|(1+\nu^{\frac13}t\mathcal{M}_1)^{\delta}u^{(1)}\|_{L^\infty} \|(1+\nu^{\frac13}t\mathcal{M}_1)^{\delta}\partial_x\theta\|_{L^2}\big\|_{L^1(T_{j_0},T)}\\
&+\|(1+\nu^{\frac13}t\mathcal{M}_1)^{\delta}u^{(2)}\|_{L^2(T_{j_0},T; L^\infty)}\|(1+\nu^{\frac13}t\mathcal{M}_1)^{\delta}\partial_y \theta\|_{L^2(T_{j_0},T;L^2)}
:=I_1+I_2.
\end{align*}
Notice that
\begin{align*}
\|(1+\nu^{\frac13}t\mathcal{M}_1)^{\delta}u^{(1)}\|_{L^\infty}\lesssim \int_{\R}(1+\nu^{\frac13}t\mathcal{M}_1(k) )^{\delta}\|u^{(1)}_{k}\|_{L^\infty_y} \dd k.
\end{align*}
If $|k|\leq 1$, we apply $\|u^{(1)}_{k}\|_{L^\infty_y}\lesssim \|\omega_{k}\|_{L^2_y}$ to obtain 
\begin{equation}\notag
\begin{aligned}
\int_{|k|\leq1}(1+\nu^{\frac13} t\mathcal{M}_1(k))^{\delta}\|u^{(1)}_{k}\|_{L^\infty_y} \dd k\lesssim \|(1+\nu^{\frac13}t\mathcal{M}_1)^{\delta}\omega\|.
\end{aligned}
\end{equation}
If $|k|\geq 1$, as in \cite{miao2025}, we have 
\begin{equation}\notag
\begin{aligned}
\int_{|k|\geq1}(1+\nu^{\frac13}t\mathcal{M}_1(k))^{\delta}\|u^{(1)}_{k}\|_{L^\infty_y}\dd k
\lesssim  \|(1+\nu^{\frac13}t\mathcal{M}_1(k))^{\delta}k u^{(1)}_{k}\|^{\frac34}\|(1+\nu^{\frac13}t\mathcal{M}_1(k))^{\delta}\partial_y \omega_{k}\|^{\frac14}.
\end{aligned}
\end{equation}
By combining the above three formulas, we can get
\begin{align}\label{esti I}
I_1\lesssim& \|(1+\nu^{\frac13}t\mathcal{M}_1)^{\delta}\omega\|_{L^\infty(T_{j_0},T;L^2)}\|(1+\nu^{\frac13}t\mathcal{M}_1)^{\delta}\partial_x\theta\|_{L^1(T_{j_0},T;L^2)}\notag\\
&+\|(1+\nu^{\frac13}t\mathcal{M}_1)^{\delta}\mathcal{M}_2\nabla \Delta^{-1}\omega\|^{\frac34}_{L^2(T_{j_0},T;L^2)}\|(1+\nu^{\frac13}t\mathcal{M}_1)^{\delta}\nabla  \omega\|^{\frac14}_{L^2(T_{j_0},T;L^2)}\notag\\ 
&\ \ \ \ \times\|(1+\nu^{\frac13}t\mathcal{M}_1)^{\delta}\nabla  \theta\|_{L^2(T_{j_0},T;L^2)}\notag\\
\lesssim &\nu^{-\frac58}\|\omega\|_{X_{2,\delta}(T_{j_0},T)}\|\theta\|_{X_{2,\delta}(T_{j_0},T)}.
\end{align}
For $I_2$, note that $\phi_k=(\partial^2_y-k^2)^{-1}\omega_k$ and
\begin{align*}
\|(1+\nu^{\frac13}t\mathcal{M}_1)^{\delta}u^{(2)}\|_{L^\infty}\lesssim \int_{\R}(1+\nu^{\frac13}t\mathcal{M}_1(k))^{\delta}\|u^{(2)}_{k}\|_{L^\infty_y} \dd k.
\end{align*}
If $|k|\leq 1$,  it holds
\begin{equation}\notag
\begin{aligned}
&\int_{|k|\leq1}(1+\nu^{\frac13}t\mathcal{M}_1(k))^{\delta}\|u^{(2)}_{k}\|_{L^\infty_y}\dd k\\
\lesssim &\int_{|k|\leq 1} |k|(1+\nu^{\frac13}t\mathcal{M}_1(k))^{\delta}\|u_k\|^{\frac34}_{L^2_y}\|\partial_y^2u_k\|^{\frac14}_{L^2_y} \dd k\\
\lesssim &\|(1+\nu^{\frac13}t\mathcal{M}_1(k))^{\delta} |k|^{\frac12} u_k\|^{\frac34}_{L^2_kL^2_y}\|(1+\nu^{\frac13}t\mathcal{M}_1(k))^{\delta}\partial_y \omega_{k}\|^{\frac14}_{L^2_kL^2_y}.
\end{aligned}
\end{equation}
If $|k|\geq 1$, we can also prove that
\begin{equation}
\begin{aligned}\notag
&\int_{|k|\geq1}(1+\nu^{\frac13}t\mathcal{M}_1(k))^{\delta}\|u^{(2)}_{k}\|_{L^\infty_y}dk\\
\lesssim &\|(1+\nu^{\frac13}t\mathcal{M}_1(k))^{\delta}k u^{(2)}_{k}\|^{\frac34}_{L^2_kL^2_y}\|(1+\nu^{\frac13}t\mathcal{M}_1(k))^{\delta}\partial_y \omega_{k}\|^{\frac14}_{L^2_kL^2_y},
\end{aligned}
\end{equation}
then we arrive at
\begin{align}\label{esti II}
I_2\lesssim \nu^{-\frac58}\|\omega\|_{X_{2,\delta}(T_{j_0},T)}\|\theta\|_{X_{2,\delta}(T_{j_0},T)}.
\end{align}
Combining  \eqref{esti I}--\eqref{esti II}, we have
\begin{align}\label{weight-uenaome}
\|(1+\nu^{\frac13}t\mathcal{M}_1)^{\delta}(u\cdot\nabla \theta)\|_{L^1(T_{j_0},T;L^2)}\lesssim\nu^{-\frac58}\|\omega\|_{X_{2,\delta}(T_{j_0},T)}\|\theta\|_{X_{2,\delta}(T_{j_0},T)} .
\end{align}
When $t\in (T_0,T_{j_0})$, similar to \eqref{esti I}--\eqref{esti II}, we derive
\begin{align}\label{uenaome}
\|u\cdot\nabla \theta\|_{L^1(T_0,T_{j_0};L^2)}\lesssim \nu^{-\frac58}\|\omega\|_{X_2(T_0,T_{j_0})}\|\theta\|_{X_2(T_0,T_{j_0})}.
\end{align}
{\it Step 3. The estimate of $E_r^2$.}
We use $(1+\nu t^3)^{-\frac14}\lesssim (\nu^{\frac13}(1+t))^{-\varepsilon}$ and $(\nu(1+t))^{\frac13}{\rm e}^{-\nu t}\lesssim 1$ to get
\begin{align*}
&\|{\rm e}^{-\nu t}(1+t)^{-1}(1+\nu t^3)^{-\frac14}\mathcal{E}_{2,in}\|_{L^1(T_{j_0},T)}+\|\nu {\rm e}^{-\nu t}(1+\nu t^3)^{-\frac14}\|_{L^1(T_{j_0},T)}\\
\lesssim  & \|\nu^{\frac13}(1+t)^{-1}(\nu^{\frac13}(1+t))^{-\varepsilon}+\nu^{\frac23}(1+\nu t^3)^{-\frac14}(1+t)^{-\frac13}\|_{L^1(T_{j_0},T)}
\lesssim \nu^{\frac13},
\end{align*}
\begin{align*}
\|{\rm e}^{-\nu t}(1+t)^{-1}(1+\nu t^3)^{-\frac14}\mathcal{E}_{2,in}\|_{L^1(T_0, T_{j_0})}+\|\nu {\rm e}^{-\nu t}(1+\nu t^3)^{-\frac14}\|_{L^1(T_0, T_{j_0})}\lesssim \nu^{\frac13-\varepsilon},
\end{align*}
and
\begin{align*}
&\int_{T_{0}}^{T} \nu^{\frac23} {\rm e}^{-\nu t}(1+\nu t^3)^{-\frac{5}{12}}\dd t\lesssim \nu^{\frac13}.
\end{align*}
Then by \eqref{a3.18}, we obtain
\begin{align}
&\|(1+\nu^{\frac13}t\mathcal{M}_1)^{\delta}E_r^2\|_{ L^1(T_{j_0},T;L^2)}\lesssim \nu^{\frac13}\mathcal{E}_{2,in},\label{weight-Er}\\
&\|E_r^2\|_{ L^1(T_0,T_{j_0};L^2)}\lesssim \nu^{\frac13-\varepsilon}\mathcal{E}_{2,in}\lesssim\nu^{\frac56-\varepsilon}\mathcal{E}_{in}.\label{estim Er}
\end{align}
By combining \eqref{weight-u1paxomL}--\eqref{estim Er}, we get \eqref{a3.81}--\eqref{a3.82}, and then finish the proof of Proposition \ref{ap3.12}.

\end{proof}

Next, we study the estimates of reaction term.
\begin{prop}\label{pro: wre}
Under the conditions of Theorem \ref{thm1.2}, it holds that for $j_0=\log_2 \nu^{-\frac23-\varepsilon}$, 
\begin{align}
\|\omega_{re}\|_{X_2(T_0,T_{j_0})}\lesssim &\nu^{-\frac13}\|\omega\|_{X_2(T_0,T_{j_0})}\mathcal{E}_{in},\label{main result}\\
\|\omega_{re}\|_{X_{2,\delta}(T_{j_0},T)}\lesssim &\nu^{-\frac13}\|\omega\|_{X_{2,\delta}(T_{j_0},T)}\mathcal{E}_{in}+\|(1+\nu^{\frac13}T_{j_0} \mathcal{M}_1)^{\delta} \omega_{re}(T_{j_0})\|,\label{om2 Xtj}\\ 
\|\theta_{re}\|_{X_2(T_0,T_{j_0})}\lesssim &\nu^{-\frac13}\|\omega\|_{X_2(T_0,T_{j_0})}\mathcal{E}_{2,in},\label{a3.96}\\
\|\theta_{re}\|_{X_{2,\delta}(T_{j_0},T)}\lesssim &\nu^{-\frac13}\|\omega\|_{X_{2,\delta}(T_{j_0},T)}\mathcal{E}_{2,in}+\|(1+\nu^{\frac13}T_{j_0} \mathcal{M}_1)^{\delta} \theta_{re}(T_{j_0})\|.\label{a3.97}
\end{align}
\end{prop}
We only prove \eqref{a3.96}--\eqref{a3.97}, and the proof of \eqref{main result}--\eqref{om2 Xtj} is similar. 
Motivated by \cite{miao2025}, we define $\theta_{re}=\sum_{j\geq 1}\theta_{re,j}$, where $\theta_{re,j}$ satisfies
\begin{equation}\label{equ: om2j}
\left\{
\begin{aligned}
&\partial_t \theta_{re,j}-\nu\Delta\theta_{re,j}+y\partial_x\theta_{re,j}=tu^{(2)}\partial_x\theta_{a}\chi_{j}, \quad t>T_{j-1}, \\
& \theta_{re,j}(t,x,\pm1)=0, \quad \theta_{re,j}|_{t=T_{j-1}}=0,
\end{aligned}
\right.
\end{equation}
with $T_0=\nu^{-\frac16}$, $T_j=\nu^{-\frac13} 2^j$ with $j\geq 1$, and $\chi_j$ solves
\begin{equation}\label{def: chij}
\chi_j(t)=\left\{
\begin{aligned}
&1,\quad  t\in(T_{j-1}, T_j],\\
&0,\quad t\notin (T_{j-1}, T_j].
\end{aligned}
\right.
\end{equation}
We have the following lemma.
\begin{lem}\label{al3.14}
Under the conditions of Theorem \ref{thm1.2}, the solution $\theta_{re,j}$ to \eqref{equ: om2j} satisfies
\begin{align}
\|\theta_{re,j}\|_{X_{2,\delta}(T_{j-1},T_j)}\lesssim2^{-\frac{j}{4}}\nu^{-\frac13}\|(1+\nu^{\frac13}t\mathcal{M}_1)^{\delta}\mathcal{M}_2\nabla\Delta^{-1}\omega\|_{L^2(T_{j-1},T_j;L^2)} \mathcal{E}_{2,in}, \quad j\geq 1. \label{om2j large}
\end{align}
In particular, for $\delta=0$, it holds that
\begin{align}
\|\theta_{re,j}\|_{X_2(T_{j-1},T_j)}\lesssim2^{-\frac{j}{4}}\nu^{-\frac13}\|\mathcal{M}_2\nabla\Delta^{-1}\omega\|_{L^2(T_{j-1},T_j;L^2)} \mathcal{E}_{2,in}, \quad j\geq 1. \label{om2j small}
\end{align}
\end{lem}
\begin{proof}
See Proposition 3.7 in \cite{miao2025}.
\end{proof}

Now, we prove \eqref{a3.96}--\eqref{a3.97} in Proposition \ref{pro: wre}.
\begin{proof}
For \eqref{a3.96}, we use $\theta_{re}=\sum_{j}\theta_{re,j}$ to get 
\begin{align}\label{om2XT0Tj0}
\|\theta_{re}\|_{X_2(T_0,T_{j_0})}=&\Big\|\sum_{1\leq j\leq j_0} \theta_{re,j}\Big\|_{X_2(T_0,T_{j_0})}\leq \sum_{1\leq j\leq j_0 }\|\theta_{re,j}\|_{X_2(T_{j-1},T_{j_0})}\notag\\
= &\sum_{1\leq j\leq j_0 }\|\theta_{re,j}\|_{X_2(T_{j-1},T_j)}+\sum_{1\leq j\leq j_0-1}\|\theta_{re,j}\|_{X_2(T_{j},T_{j_0})}.
\end{align}
Then by \eqref{om2j small}, we infer
\begin{align}\label{om2jtildeXj}
\sum_{1\leq j\leq j_0 }\|\theta_{re,j}\|_{X_2(T_{j-1},T_j)}&\lesssim \nu^{-\frac13}\sum_{1\leq j\leq j_0}2^{-\frac{j}{4}}\|\mathcal{M}_2 \nabla\Delta^{-1}\omega\|_{L^2(T_{j-1},T_j;L^2)} \mathcal{E}_{2,in}\notag\\ 
&\lesssim \nu^{-\frac13}\|\omega\|_{X_2(T_0,T_{j_0})}\mathcal{E}_{2,in}.
\end{align}
For $t>T_j$, the system \eqref{equ: om2j} becomes
\begin{equation}\label{equ: om2j free}
\left\{
\begin{aligned}
&\partial_t \theta_{re,j}+y\partial_x\theta_{re,j}-\nu\Delta\theta_{re,j}=0, \quad t>T_{j}, \\
& \theta_{re,j}(t,x,\pm1)=0, \quad \theta_{re,j}|_{t=T_{j}}=\theta_{re,j}(T_j).
\end{aligned}
\right.
\end{equation}
By \eqref{esti: ftildeX} and \eqref{om2j small}, we obtain
\begin{align}\label{om2j Tj-Tj0}
\|\theta_{re,j}\|_{X_2(T_{j},T_{j_0})}\lesssim&\|\theta_{re,j}(T_j)\|_{L^2}\lesssim\|\theta_{re,j}\|_{X_2(T_{j-1},T_j)}\notag\\ 
\lesssim&2^{-\frac{j}{4}}\nu^{-\frac13}\|\mathcal{M}_2 \nabla\Delta^{-1}\omega\|_{L^2(T_{j-1},T_j;L^2)} \mathcal{E}_{2,in}.
\end{align}
Putting \eqref{om2jtildeXj} and \eqref{om2j Tj-Tj0} into \eqref{om2XT0Tj0}, we have
\begin{align}
\|\theta_{re}\|_{X_2(T_0,T_{j_0})}\lesssim\nu^{-\frac13}\|\omega\|_{X_2(T_0,T_{j_0})}\mathcal{E}_{2,in}.
\end{align}

For \eqref{a3.97}, define
$$\theta_{re}=\sum_{1\leq j\leq j_0}\theta_{re,j}+\sum_{j\geq j_0+1}\theta_{re,j}:=\widetilde{\theta}_{re}+\sum_{j\geq j_0+1}\theta_{re,j}.$$
Here $\widetilde{\theta}_{re}$ satisfies
\begin{equation*}
\left\{
\begin{aligned}
&\partial_t \widetilde{\theta}_{re}-\nu\Delta\widetilde{\theta}_{re}+y\partial_x\widetilde{\theta}_{re}=0, \quad t>T_{j_0}, \\
& \widetilde{\theta}_{re}(t,x,\pm1)=0, \quad \widetilde{\theta}_{re}|_{t=T_{j_0}}=\theta_{re}(T_{j_0}).
\end{aligned}
\right.
\end{equation*}
It follows from \eqref{fX large time} that
\begin{align}\label{om2,j small}
\|\widetilde{\theta}_{re}\|_{X_{2,\delta}(T_{j_0},T)}\lesssim &\|(1+\nu^{\frac13}T_{j_0} \mathcal{M}_1)^{\delta} \widetilde{\theta}_{re}(T_{j_0})\|\lesssim\|(1+\nu^{\frac13}T_{j_0} \mathcal{M}_1)^{\delta} \theta_{re}(T_{j_0})\|.
\end{align}
For $j\geq j_0+1$, similar to \eqref{om2jtildeXj}--\eqref{om2j Tj-Tj0}, we apply \eqref{om2j large} to obtain
\begin{align*}
\|\theta_{re,j}\|_{X_{2,\delta}(T_{j-1},T)}\leq &\|\theta_{re,j}\|_{X_{2,\delta}(T_{j-1},T_j)}+\|\theta_{re,j}\|_{X_{2,\delta}(T_j,T)}\\
\lesssim &2^{-\frac{j}{4}}\nu^{-\frac13}\|(1+\nu^{\frac13}t\mathcal{M}_1)^{\delta}\mathcal{M}_2 \nabla\Delta^{-1}\omega\|_{L^2(T_{j-1},T_j;L^2)}\mathcal{E}_{2,in},
\end{align*}
which implies
\begin{align}\label{om2,j large}
\sum_{j\geq j_0+1}\|\theta_{re,j}\|_{X_{2,\delta}(T_{j},T)}\lesssim\nu^{-\frac13}\|(1+\nu^{\frac13}t\mathcal{M}_1)^{\delta}\mathcal{M}_2 \nabla\Delta^{-1}\omega\|_{L^2(T_{j_0},T;L^2)} \mathcal{E}_{2,in}.
\end{align}
Combining \eqref{om2,j small}--\eqref{om2,j large}, we infer
\begin{align*}
\|\theta_{re}\|_{X_{2,\delta}(T_{j_0},T)}\leq &\|\widetilde{\theta}_{re}\|_{X_{2,\delta}(T_{j_0},T)}+\sum_{j\geq j_0+1}\|\theta_{re,j}\|_{X_{2,\delta}(T_{j-1},T)}\\
\lesssim &\nu^{-\frac13}\|\omega\|_{X_{2,\delta}(T_{j_0}, T)}\mathcal{E}_{2,in}+\|(1+\nu^{\frac13}T_{j_0} \mathcal{M}_1)^{\delta} \theta_{re}(T_{j_0})\|.
\end{align*}
Therefore, we have \eqref{a3.96}--\eqref{a3.97}, and the proof of \eqref{main result}--\eqref{om2 Xtj} is similar. The proof of Proposition \ref{pro: wre} is completed.
\end{proof}

Based on these estimates, we are ready to prove Proposition \ref{ap3.2}.
\begin{proof}[Proof of Proposition \ref{ap3.2}]
Using \eqref{esti om1 tildeX}--\eqref{a3.80} and \eqref{main result}--\eqref{a3.97}, we infer
\begin{align*}
\|\omega\|_{X_2(T_0, T_{j_0})}\leq& \|\omega_*\|_{X_2(T_0, T_{j_0})}+\|\omega_{re}\|_{X_2(T_0, T_{j_0})}\\
\lesssim& \nu^{-\frac13}\mathcal{E}_{in}\|\omega\|_{X_2(T_0, T_{j_0})}+\nu^{-\frac58}\|\omega\|^2_{X_2(T_0, T_{j_0})}+\nu^{\frac13-\varepsilon}\mathcal{E}_{in}+\nu^{-\frac12}\|\theta\|_{X_2(T_0,T_{j_0})},\\ 
\|\theta\|_{X_2(T_0, T_{j_0})}\leq& \|\theta_*\|_{X_2(T_0, T_{j_0})}+\|\theta_{re}\|_{X_2(T_0, T_{j_0})}\\
\lesssim& \nu^{-\frac13}\mathcal{E}_{2,in}\|\omega\|_{X_2(T_0,T_{j_0})}+\nu^{-\frac13}\mathcal{E}_{in}\|\theta\|_{X_2(T_{0},T_{j_0})}+\nu^{-\frac58}\|\omega\|_{X_2(T_0,T_{j_0})}\|\theta\|_{X_2(T_0,T_{j_0})} \notag\\ 
& +\nu^{\frac56-\varepsilon}\mathcal{E}_{in}.
\end{align*}
Notice that $\nu^{-\frac12}\mathcal{E}_{2,in}+\mathcal{E}_{in}\leq \varepsilon_0\nu^{\frac13}$ and $\varepsilon_0\ll 1$, 
\begin{align*}
\|\omega\|_{X_2(T_0, T_{j_0})}\leq& C(\nu^{-\frac58} \|\omega\|^2_{X_2(T_0, T_{j_0})}+\nu^{\frac13-\varepsilon}\mathcal{E}_{in}+\nu^{-\frac12}\|\theta\|_{X_2(T_0,T_{j_0})}),\\ 
\|\theta\|_{X_2(T_0, T_{j_0})}\leq&C(\nu^{-\frac13}\mathcal{E}_{2,in}\|\omega\|_{X_2(T_0,T_{j_0})}+\nu^{-\frac58}\|\omega\|_{X_2(T_0,T_{j_0})}\|\theta\|_{X_2(T_0,T_{j_0})}+\nu^{\frac56-\varepsilon}\mathcal{E}_{in}),
\end{align*}
which implies that
\begin{align*}
\|\omega\|_{X_2(T_0, T_{j_0})}+\nu^{-\frac12}\|\theta\|_{X_2(T_0, T_{j_0})}\lesssim& \nu^{-\frac58} \|\omega\|^2_{X_2(T_0, T_{j_0})}+\nu^{-\frac98}\|\omega\|_{X_2(T_0,T_{j_0})}\|\theta\|_{X_2(T_0,T_{j_0})}+\nu^{\frac13-\varepsilon}\mathcal{E}_{in}.
\end{align*}
We use $\mathcal{E}_{in}\leq \varepsilon_0\nu^{\frac13}$ and a priori assumption $\|\omega\|_{X_2(T_0, T_{j_0})}+\nu^{-\frac12}\|\theta\|_{X_2(T_0, T_{j_0})}\leq \varepsilon_1\nu^{\frac23-\varepsilon}$ to get
\begin{align*}
\|\omega\|_{X_2(T_0, T_{j_0})}+\nu^{-\frac12}\|\theta\|_{X_2(T_0, T_{j_0})}\lesssim\nu^{\frac13-\varepsilon}\mathcal{E}_{in}.
\end{align*}
Putting it into \eqref{esti om1 tildeX}, \eqref{a3.79}, \eqref{main result} and \eqref{a3.96} gives
\begin{align*}
\|\omega_*\|_{X_2(T_0, T_{j_0})}+\|\omega_{re}\|_{X_2(T_0, T_{j_0})}+\nu^{-\frac12}(\|\theta_*\|_{X_2(T_0, T_{j_0})}+\|\theta_{re}\|_{X_2(T_0, T_{j_0})})\lesssim\nu^{\frac13-\varepsilon}\mathcal{E}_{in}.
\end{align*}
By taking $0< \varepsilon<(\frac{1}{24}-\frac23\delta)(1+\delta)^{-1}$ and using $T_{j_0}=\nu^{-1-\varepsilon}$, we have
\begin{align}\label{ome small}
&\|(1+ \nu^{\frac13}T_{j_0} \mathcal{M}_1)^{\delta}\theta_*(T_{j_0})\|\leq  \|(1+\nu^{\frac13}t\mathcal{M}_1)^{\delta}\theta_* \|_{X_2(T_0, T_{j_0})}\notag\\
\leq & \|(1+\nu^{\frac13}t)^{\delta}\|_{L^\infty(T_0,T_{j_0})}\|\theta_*\|_{X_2(T_0, T_{j_0})}
\leq  C \nu^{\frac13-\frac23\delta-\varepsilon(1+\delta)}\mathcal{E}_{2,in}\lesssim\nu^{\frac{19}{24}}\mathcal{E}_{in},
\end{align}
\begin{align}\label{om2 small}
&\|(1+ \nu^{\frac13}T_{j_0} \mathcal{M}_1)^{\delta}\theta_{re}(T_{j_0})\|\leq  \|(1+\nu^{\frac13}t\mathcal{M}_1)^{\delta}\theta_{re} \|_{X_2(T_0, T_{j_0})}\notag\\
\leq & \|(1+\nu^{\frac13}t)^{\delta}\|_{L^\infty(T_0,T_{j_0})}\|\theta_{re}\|_{X_2(T_0, T_{j_0})}
\leq  C \nu^{\frac13-\frac23\delta-\varepsilon(1+\delta)}\mathcal{E}_{2,in}\lesssim\nu^{\frac{19}{24}}\mathcal{E}_{in},
\end{align}
and
\begin{align}\label{a3.125}
  &\|(1+ \nu^{\frac13}T_{j_0} \mathcal{M}_1)^{\delta}\omega_*(T_{j_0})\|+\|(1+ \nu^{\frac13}T_{j_0} \mathcal{M}_1)^{\delta}\omega_{re}(T_{j_0})\|\notag\\ 
  \lesssim&\|(1+\nu^{\frac13}t\mathcal{M}_1)^{\delta}\omega_* \|_{X_2(T_0, T_{j_0})}+\|(1+\nu^{\frac13}t\mathcal{M}_1)^{\delta}\omega_{re} \|_{X_2(T_0, T_{j_0})}\lesssim\nu^{\frac{7}{24}}\mathcal{E}_{in}.
\end{align}
Then by \eqref{1.8}, \eqref{esti om1 X large}, \eqref{a3.80}, \eqref{om2 Xtj}, \eqref{a3.97}, \eqref{ome small}--\eqref{a3.125}, we conclude
\begin{equation}\notag
\begin{aligned}
\|\omega\|_{X_{2,\delta}(T_{j_0},T)}\leq& \|\omega_*\|_{X_{2,\delta}(T_{j_0},T)}+\|\omega_{re}\|_{X_{2,\delta}(T_{j_0},T)}\\
\leq &C(\nu^{-\frac13}\mathcal{E}_{in}\|\omega\|_{X_{2,\delta}(T_{j_0},T)}+\nu^{-\frac58}\|\omega\|^2_{X_{2,\delta}(T_{j_0},T)}+\nu^{-\frac12}\|\theta\|_{X_{2,\delta}(T_{j_0},T)}+\nu^{\frac{7}{24}}\mathcal{E}_{in})\\
\leq &C(\nu^{-\frac58}\|\omega\|^2_{X_{2,\delta}(T_{j_0},T)}+\nu^{-\frac12}\|\theta\|_{X_{2,\delta}(T_{j_0},T)}+\nu^{\frac{7}{24}}\mathcal{E}_{in}),\\
\|\theta\|_{X_{2,\delta}(T_{j_0},T)}\leq& \|\theta_*\|_{X_{2,\delta}(T_{j_0},T)}+\|\theta_{re}\|_{X_{2,\delta}(T_{j_0},T)}\\
\leq &C(\nu^{-\frac13}\mathcal{E}_{in}\|\theta\|_{X_{2,\delta}(T_{j_0},T)}+\nu^{-\frac13}\mathcal{E}_{2,in}\|\omega\|_{X_{2,\delta}(T_{j_0},T)}\\ 
&+\nu^{-\frac58}\|\omega\|_{X_{2,\delta}(T_{j_0},T)}\|\theta\|_{X_{2,\delta}(T_{j_0},T)}+\nu^{\frac{19}{24}}\mathcal{E}_{in}).
\end{aligned}
\end{equation}
Therefore,
\begin{align}
 \|\omega\|_{X_{2,\delta}(T_{j_0},T)}+\nu^{-\frac12}\|\theta\|_{X_{2,\delta}(T_{j_0},T)}\lesssim&\nu^{-\frac58}\|\omega\|^2_{X_{2,\delta}(T_{j_0},T)}+\nu^{-\frac98}\|\omega\|_{X_{2,\delta}(T_{j_0},T)}\|\theta\|_{X_{2,\delta}(T_{j_0},T)}\notag\\ 
 &+\nu^{\frac{7}{24}}\mathcal{E}_{in}.   
\end{align}
By assuming $\|\omega\|_{X_{2,\delta}(T_{j_0},T)}+\nu^{-\frac12}\|\theta\|_{X_{2,\delta}(T_{j_0},T)} \leq \varepsilon_1\nu^{\frac58}$, and using $\mathcal{E}_{in}\leq \varepsilon_0\nu^{\frac13}~(\varepsilon_0\ll 1)$, we arrive at
\begin{align}\label{ome X}
\|\omega\|_{X_{2,\delta}(T_{j_0},T)}+\nu^{-\frac12}\|\theta\|_{X_{2,\delta}(T_{j_0},T)}\leq C \nu^{\frac{7}{24}}\mathcal{E}_{in},
\end{align}
which together with Proposition \ref{ap3.8} and \eqref{ome small}--\eqref{a3.125} gives \eqref{a3.20}--\eqref{a3.20a}. This completes the proof of Proposition \ref{ap3.2}.
\end{proof}

\section{Further discussions}\label{s6}
In this work, we have investigated the stability of Boussinesq equations in $\T\times[-1,1]$ and $\R\times[-1,1]$. We impose the Navier-slip boundary conditions on the velocity here. 
For the vorticity, the threshold $\frac{1}{3}$ seems to be optimal \cite{wei-zhang2026}. 
However, regarding the temperature, it remains unknown whether the threshold $\frac{5}{6}$ is optimal. 
It is worth noting that the authors in \cite{Niu} obtained the threshold  $\frac{2}{3}$ for the temperature in $\T\times \R$.
 Through careful analysis of the result in \cite{Niu}, we find that the key technique also lies in the quasi-linear method, i.e., the main system \eqref{eq:1.4} in $\T\times \R$ is decomposed into the following two equations. 
 Let

\[
\omega = \omega^i + \omega^e, \quad \theta = \theta^i + \theta^e,
\]
where \((\omega^i, \theta^i)\) solves

\begin{equation}\label{i}
\begin{cases}
\partial_t \omega^i + y \partial_x \omega^i - \nu \Delta \omega^i + u^i \cdot \nabla \omega^i = -\partial_x \theta^i, \\
\partial_t \theta^i + y \partial_x \theta^i + u_0^{i,(1)} \partial_x \theta^i - \nu \Delta \theta^i = 0, \\
\partial_t u_0^{i,(1)} - \nu \partial_y^2 u_0^{i,(1)} = P_0(u_{\neq}^i \cdot \nabla u_{\neq}^{i,(1)}), \\
\omega^i(0, x, y) = \Omega_{in}(x, y), \quad \theta^i(0, x, y) = \Theta_{in}(x, y),
\end{cases}
\end{equation}
and \((\omega^e, \theta^e)\) solves

\begin{equation}\label{e}
\begin{cases}
\partial_t \omega^e + y \partial_x \omega^e - \nu \Delta \omega^e + u^i \cdot \nabla \omega^e + u^e \cdot \nabla \omega^i + u^e \cdot \nabla \omega^e = -\partial_x \theta^e, \\
\partial_t \theta^e + y \partial_x \theta^e - \nu \Delta \theta^e + u_{\neq}^i \cdot \nabla \theta^e + u^i \cdot \nabla \theta^e + u^e \cdot \nabla \theta^e + u^e \cdot \nabla \theta^i = 0, \\
\omega^e(0, x, y) = 0, \quad \theta^e(0, x, y) = 0.
\end{cases}
\end{equation}
 Here \((\omega^i, \theta^i)\) represents the good part and \((\omega^e, \theta^e)\) denotes the quasi-linear and nonlinear part. The stability result in $\T\times\R$ is achieved by handling these two parts separately and combining their estimates.

An interesting observation from our subsequent research is that, for the good part \((\omega^i, \theta^i)\) in the above decomposition, we can extend the stability result from $\T\times \R$ to $\T\times[-1,1]$. 
This indicates that the boundary effects do not affect the threshold of the good component, which provides a crucial clue for further improving the stability results in bounded domains. 
However, we encounter an significant difficulty:
In estimating \( \om^e \), we need to obtain higher-order estimates for \( \om^i \), rather than just zeroth-order estimates. 
This poses some problems with boundaries, because after taking derivatives with respect to $y$, some boundary information seems to be lost.
Moreover, at this stage we cannot use Fourier multipliers and Plancherel's theorem, which makes boundary value problems significantly different from the whole-space case.
Based on the above analysis, we propose the following problem: Can we get higher-order estimate of \( \om^i \) under Navier-slip boundary conditions, and thereby extend the stability result from the $\T\times \R$ case to the $\T\times[-1,1]$ case for the original equation? 
We believe that solving this open problem will not only improve the stability theory of the considered partial differential equations in bounded domains but also provide new ideas and techniques for handling boundary effects in related PDE problems.

In the following, we present the stability estimates for the good part \( \om^i \) under boundary conditions. This may be the first step to extend the results in the whole space to the problems with boundaries.
\begin{thm}\label{thm4.1}
    There exist $0<\delta, \eps_0\ll 1$ such that for all $0<\nu \le 1$, $(x,y)\in \T\times[-1,1]$, if the initial data $\Om_{in}, \Theta_{in}$ satisfies $\|\Om_{in}\|_{H^4}+\nu^{-\frac{1}{3}}\|\Theta_{in}\|_{H^4}\le \eps_0\nu^{\frac{1}{3}}$, $(\om^i,\theta^i)(t,x,\pm1)=0$, then the solution $(\om^i,\theta^i)(t,x,y)$ of \eqref{i} satisfies the following stability estimates:
    \[\|\om^i_{\neq}\|+\nu^{-\frac{1}{3}}\|\theta^i_{\neq}\|\lesssim \e^{-\delta\nu^{\frac{1}{3}}t}(\|\Om_{in}\|_{H^4}+\nu^{-\frac{1}{3}}\|\Theta_{in}\|_{H^4}),\]
    \[\|u^i_{\neq}\|\lesssim (1+t)^{-1} \e^{-\delta\nu^{\frac{1}{3}}t/2}(\|\Om_{in}\|_{H^4}+\nu^{-\frac{1}{3}}\|\Theta_{in}\|_{H^4}).\]
\end{thm}

\section{Appendix}
\subsection{An elementary embedding inequality}
Here we present a crucial embedding inequality used in the nonlinear estimates for the case $D=\mathbb{T}\times\mathbb{R}$.
\begin{lem}\label{lem:A.1}
Suppose that $k\in \Z$ and $P_0 f = 0$, then for a constant $0<\delta_1< 1$, it holds that
\[
\|f\|_{L^\infty} \lesssim \min\left\{\frac{1}{\delta_1}\|\nabla f\|^{1-\frac{\delta_1}{2}} \|\partial_x^2f\|^{\frac{\delta_1}{2}},\ \|\partial_x f\|^{\frac{3}{4}} \|f\|_{H^2}^{\frac{1}{4}}\right\}.
\]
\end{lem}

\begin{proof}
Denote $f_k(y) = \frac{1}{2\pi} \int_{\mathbb{T}} f(x,y) \e^{-ikx} \dd x, I=[-1,1]$, then for $k \in \mathbb{Z}$ and $j=1,2$,
\[
f(x,y) = \sum_{k \neq 0} f_k(y) \e^{ikx},
\]
and
\[
\|\partial_x^j f\|_{L^2(\mathbb{T}\times\mathbb{R})}^2 = 2\pi \sum_{k \neq 0} \|(ik)^j f_k\|_{L^2(I)}^2, \quad
\|\nabla f\|_{L^2(\mathbb{T}\times\mathbb{R})}^2 = 2\pi \sum_{k \neq 0} \|(k, \partial_y) f_k\|_{L^2(I)}^2.
\]
Using Sobolev inequality and Hölder inequality gives
\begin{align*}
\|f\|_{L^\infty(\mathbb{T}\times\mathbb{R})} &\leq \sum_{k \neq 0} \|f_k\|_{L^\infty(I)} \lesssim \sum_{k \neq 0} \|f_k\|_{L^2(I)}^{\frac{1}{2}} \|f_k\|_{H^1(I)}^{\frac{1}{2}}\lesssim \sum_{k \neq 0} |k|^{-\frac{1}{2}-\frac{\delta_1}{2}}\||k|^{1+\delta_1}f_k\|_{L^2(I)}^{\frac{1}{2}} \|f_k\|_{H^1(I)}^{\frac{1}{2}} \\
&\lesssim \left( \sum_{k \neq 0} |k|^{-1-\delta_1} \right)^{\frac{1}{2}} \left( \sum_{k \neq 0} \||k|^{1+\delta_1} f_k\|_{L^2(I)}^2 \right)^{\frac{1}{4}} \left( \sum_{k \neq 0} \|(k, \partial_y) f_k\|_{L^2(I)}^2 \right)^{\frac{1}{4}} \\
&\lesssim \frac{1}{\delta_1} \left( \sum_{k \neq 0} \||k|^{2} f_k\|_{L^2(I)}^{2\delta_1}\||k|f_k\|_{L^2(I)}^{2(1-\delta_1)} \right)^{\frac{1}{4}} \left( \sum_{k \neq 0} \|(k, \partial_y) f_k\|_{L^2(I)}^2 \right)^{\frac{1}{4}} \\
&\lesssim \frac{1}{\delta_1} \|\nabla f\|^{1-\frac{\delta_1}{2}} \|\partial_x^2f\|^{\frac{\delta_1}{2}}.
\end{align*}
The proof of $\|f\|_{L^\infty} \lesssim \|\partial_x f\|^{\frac{3}{4}} \|f\|_{H^2}^{\frac{1}{4}}$ was shown in \cite{wei-zhang2026}. The proof of this lemma is completed.
\end{proof}
\subsection{Proof of Theorem \ref{thm4.1}}
 Similar to Theorem \ref{thm1.1}, we also use quasi-linearization method. Let 
\[W_{1,in}=\mathbb{P}_0 \Om_{in},\ W_{2,in}=\mathbb{P}_0 \Theta_{in},\ W_1=\e^{\nu t\pr_y^2}W_{1,in},\ W_2=\e^{\nu t \pr_y^2} W_{2,in},\ V_1=y-\pr_y(-\pr_y^2)^{-1}W_1.\]
Thanks to $W_{1,in}(\pm 1) =W_{2,in}(\pm 1) =0$, it holds for $t > 0$ that
\begin{align}\label{background flow}
& \|V_1 - y\|_{C^4} = \|\partial_y(-\partial_y^2)^{-1} W_1\|_{C^4} \lesssim \|W_1\|_{C^3} \leq \|W_{1,in}\|_{C^3} \lesssim \|W_{1,in}\|_{H^4} \leq \|\Omega_{in}\|_{H^4}, \\
& \|V_{1}' - 1\|_{C^3} + \|V_{1}''\|_{C^2} + \|V_{1}'''\|_{C^1}+\|V_1^{(4)}\|_{L^\infty} \lesssim \|V_{1} - y\|_{C^4} \lesssim \|\Omega_{in}\|_{H^4}.
\end{align}
Denote $\om_{ob}=\om^i-W_1,\ \theta_{ob}=\theta^i-W_2.$ Then we have
\be
\begin{cases}
    \pr_t \om_{ob}+V_1\pr_x\om_{ob}-\nu\Delta\om_{ob}-V_1'' \pr_x\phi+u_{ob}\cdot \nabla \om_{ob}=-\pr_x \theta_{ob},\\[2mm]
    \pr_t \theta_{ob}+V_1\pr_x\theta_{ob}-\nu\Delta\theta_{ob}+u_{ob,0}^1\pr_x\theta_{ob}=0,\\[2mm]
    u_{ob}=(\pr_y,-\pr_x)\Delta^{-1}\om_{ob},\ \phi=\Delta^{-1}\om_{ob},\\[2mm]
    \om_{ob}|_{t=0}=\Om_{in,\neq},\ \theta_{ob}|_{t=0}=\Theta_{in,\neq}.
\end{cases}
\ee
Next, we construct approximate solutions $\omega_a$ and $\theta_a$ of $\omega_{ob}$ and $\theta_{ob}$ respectively:
\be \label{linear-om}
\begin{cases}
    \pr_t \om_a-\nu t^2\pr_x^2 \om_a+V_1\pr_x \om_a=-\pr_x\theta_a,\\[2mm]
    \pr_t\theta_a-\nu t^2\pr_x^2\theta_a+V_1\pr_x\theta_a=0,\\[2mm]
    \om_a|_{t=0}=\Om_{in,\neq},\ \om_a(t,x,\pm 1)=0,\\[2mm]
    \theta_a|_{t=0}=\Theta_{in,\neq},\ \theta_a(t,x,\pm 1)=0.
\end{cases}
\ee
Let 
\[\om=\om_{ob}-\om_a,\ \theta=\theta_{ob}-\theta_a, \phi=\Delta^{-1}\om_a, \ \phi=\phi-\phi_a,\ u=(\pr_y,-\pr_x)\phi_a, u=u_{ob}-u_a.\]
Then
\be
\begin{cases}
    \pr_t\om+V_1\pr_x\om-\nu\Delta\om-V_1''\pr_x\phi+u_{ob}\cdot \nabla\om+u\cdot \nabla \om_a +Er^1=-\pr_x \theta,\\[2mm]
    \pr_t \theta+V_1\pr_x\theta-\nu\Delta\theta+u_{0}^{(1)}\pr_x\theta+u_{0}^{(1)}\pr_x\theta_a+Er^2=0,\\[2mm]
    Er^1=Er_\om^1-V_1''\pr_x\phi_a+u_a\cdot \nabla \om_a,\ Er_\om^1=\pr_t\om_a+V_1\pr_x\om_a-\nu\Delta\om_a+\pr_x\theta_a,\\[2mm]
    Er^2=\pr_t\theta_a+V_1\pr_x\theta_a-\nu\Delta\theta_a,\\[2mm]
    \om(0,x,y)|=\theta(0,x,y)=\om(t,x,\pm 1)=\theta(t,x,\pm 1)=0.
\end{cases}
\ee
In addition to \eqref{2.10}--\eqref{2.15}, we can also prove the following estimates:
\begin{prop}\label{prop:linear}
 the solution $(\om_a,\theta_a)$ of \eqref{linear-om}  satisfies that for $t \geq 0$,
\begin{align}
& \|\partial^2_x \theta_a(t)\|_{L^\infty} + \|(\partial_y + t\partial_x)\pr_x\theta_a(t)\|_{L^\infty} \lesssim {\rm e}^{-2\nu^{\frac{1}{3}} t} \|\Theta_{in}\|_{H^4}, \\
&\|\pr_x Er^2(t)\|\les\ \nu (1+t) \e^{-2\nu^{\frac{1}{3}} t} \|\Theta_{in}\|_{H^4}. 
\end{align}
\end{prop}
\begin{proof}
    The proof is similar to that of Proposition \ref{p2.1} and Lemma \ref{l3.2}, and the specific details are omitted here.
\end{proof}
For $(\om, \theta)$, we have the following estimates.
\begin{prop} \label{thm:error-term}
    Under the conditions of Theorem \ref{thm4.1}, it holds for $t\ge 0$,
    \[\|\om_{\neq}\|+\nu^{-\frac{1}{3}}(\|\theta_{\neq}\|+\|\pr_x\theta_{\neq}\|)\lesssim \e^{-\delta\nu^{\frac{1}{3}}t}\nu^{\frac{1}{3}}(\|\Om_{in}\|_{H^4}+\nu^{-\frac{1}{3}}\|\Theta_{in}\|_{H^4}).\]
\end{prop} 
For Proposition \ref{thm:error-term}, we still divide into two time periods to analyze. For short time, we have
\begin{lem} \label{prop:short-time}
    If $0\le t\le \nu^{-\frac{1}{6}}$, then
    \[\|\om\|+\nu^{-\frac{1}{3}}\|\theta\|+\nu^{-\frac{1}{3}}\|\pr_x\theta\|\lesssim\nu^{\frac{1}{3}}(\|\Om_{in}\|_{H^4}+\nu^{-\frac{1}{3}}\|\Theta_{in}\|_{H^4}).\]
\end{lem}
\begin{proof}
    By energy estimate, we have 
    \be\notag
\|\om\|_{L^\infty L^2}+\nu^{\frac{1}{2}}\|\nabla \om\|_{L^2 L^2}\lesssim (\|V_1''\|_{L^1L^\infty}+\|\nabla \om_a\|_{L^1 L^\infty})+\|Er^1\|_{L^1L^2}+\nu^{-\frac{1}{6}}\|\pr_x \theta\|_{L^\infty L^2},
    \ee
    \be\notag
\|\theta\|_{L^\infty L^2}+\nu^{\frac{1}{2}}\|\nabla\theta\|_{L^2L^2}\lesssim\|\om\|_{L^\infty L^2}\|\pr_x\theta_a\|_{L^1L^\infty}+\|Er^2\|_{L^1L^2} ,
    \ee
    \be\notag
\|\pr_x\theta\|_{L^\infty L^2}+\nu^{\frac{1}{2}}\|\nabla\pr_x\theta\|_{L^2L^2}\lesssim \|\om\|_{L^\infty L^2}\|\pr^2_x\theta_a\|_{L^1L^\infty}+\|\pr_x Er^2\|_{L^1L^2} .
    \ee
By Propositions \ref{p2.1} and \ref{prop:linear}, and $0\le t\le \nu^{-\frac{1}{6}}$, we have
    \[\|V_1''\|_{L^1 L^\infty}\les \nu^{-\frac{1}{6}}\|\Om_{in}\|_{H^4},\ \|\pr_x\theta_a\|_{L^1 L^\infty}+\|\pr_x^2\theta_a\|_{L^1 L^\infty}\les \nu^{-\frac{1}{6}}\|\Theta_{in}\|_{H^4},\]
    \[\|\nabla\om_a\|_{L^1 L^\infty}\les \nu^{-\frac{1}{3}}(\|\Om_{in}\|_{H^4}+\nu^{-\frac{1}{3}}\|\Theta\|_{H^4}),\  \|Er^1\|_{L^1L^2}\les \nu^{\frac{1}{3}}(\|\Om_{in}\|_{H^4}+\nu^{-\frac{1}{3}}\|\Theta\|_{H^4}),\]
    \[\|Er^2\|_{L^1 L^2}+\|\pr_x Er^2\|_{L^1 L^2}\les \nu^{\frac{1}{3}}\|\Theta_{in}\|_{H^4}.\]
    Then by assuming that
    \be\notag
\|\om\|_{L^\infty L^2}+\nu^{\frac{1}{2}}\|\nabla \om\|_{L^2 L^2}+\nu^{-\frac{1}{3}}\Big( \|\theta\|_{L^\infty L^2}+\nu^{\frac{1}{2}}\|\nabla\theta\|_{L^2L^2} +\|\pr_x\theta\|_{L^\infty L^2}+\nu^{\frac{1}{2}}\|\nabla\pr_x\theta\|_{L^2L^2}\Big)\les \eps_1 \nu^{\frac{2}{3}}
    \ee
    for some small $\eps_1$ , we have for $0<t\le \nu^{-\frac{1}{6}}$ that
    \[\|\om\|_{L^2}+\nu^{-\frac{1}{3}}\|\theta\|_{L^2}+\nu^{-\frac{1}{3}}\|\pr_x\theta\|_{L^2}\le C\nu^{\frac{1}{3}}(\|\Om_{in}\|_{H^4}+\nu^{-\frac{1}{3}}\|\Theta_{in}\|_{H^4}).\]
\end{proof}

 For the long time scale, consider the reaction part $\om_{re}$, which solves 
\be
\begin{cases}
    \pr_t\om_{re}-\nu\Delta\om_{re}+V_1\pr_x \om_{re}-tu_e^2\pr_x\om_{a}=0,\\[2mm]
    \om_{re}(t,x,\pm 1)=0,\ \om_{re}|_{t=\nu^{-\frac{1}{6}}}=0.
\end{cases}
\ee
Then let $\om_*=\om-\om_{re}$, which solves 
\be
\begin{cases}
    \pr_t\om_*+V_1\pr_x\om_*-\nu\Delta\om_*+g=0,\\[2mm]
    g=-V_1''\pr_x\phi+u_{ob}\cdot\nabla\om+u^{(1)}\pr_x\om_a+u^{(2)}(\pr_y+t\pr_x)\om_a+Er^1+\pr_x\theta,\\[2mm]
    \om_*|_{t=\nu^{-\frac{1}{6}}}=\om|_{t=\nu^{-\frac{1}{6}}}.
\end{cases}
\ee
\begin{lem} \label{prop:reaction}
    It holds that
    \[\|\om_{re}\|_{X_2}\les\nu^{-\frac{1}{3}}\|\e^{\delta\nu^{\frac{1}{3}}t}\nabla u^{(2)}\|_{L^2 L^2}(\|\Om_{in}\|_{H^4}+\nu^{-\frac{1}{3}}\|\Theta_{in}\|_{H^4}).\]
   Here the norm $\|\cdot \|_{X_2}$ is defined in \eqref{4.15aaa}.
\end{lem}
\begin{proof}
    The proof is similar to Proposition \ref{p4.12a}, and the specific details are omitted here.
\end{proof}
To estimate $\om_*$, we first have
\begin{lem} \label{prop:source-estimate}
     It holds that 
     \be\notag
\begin{aligned}
\|e^{\delta\nu^{\frac{1}{3}}t}g\|_{L^1 L^2}\les&\nu^{-\frac{1}{3}}(\|\Om_{in}\|_{H^4}+\nu^{-\frac{1}{3}}\|\Theta_{in}\|_{H^4})\|\omega\|_{X_2}+\nu^{-\frac{2}{3}}\|\om\|_{X_2}^2+\nu^{-\frac{5}{6}}\|\om\|_{X_2}\|\pr_x\theta\|_{X_2}\\ 
&+\nu^{\frac{1}{3}}\|\Omega_{in}\|_{H^4}+\|\Theta_{in}\|_{H^4}.
\end{aligned}
\ee
\end{lem}
\begin{proof}
    We only estimate $\|\e^{\delta\nu^{\frac{1}{3}}t}\pr_x\theta\|_{L^1L^2}$; the estimates of other terms are similar to Proposition \ref{prop:2.6}.
    By Proposition \ref{p4.8a}, we have
\[\|\e^{\delta\nu^{\frac{1}{3}}t}\pr_x\theta\|_{L^1L^2}\le \nu^{-\frac{1}{3}}\|\pr_x\theta\|_{X_2}\les \nu^{-\frac{1}{3}}\Big(\|\pr_x\theta(T_0)\|_{L^2}+\|e^{\delta\nu^{\frac{1}{3}}t}g_1\|_{L^1 L^2}\Big),\]
in which
\[g_1=u_{0}^{(1)}\pr^2_x\theta+u_{0}^{(1)}\pr_x^2\theta_a+\pr_x Er^2.\]
We can prove
\be\notag
\begin{aligned}
    \|\e^{\delta\nu^{\frac{1}{3}}t}u_{0}^{(1)}\pr_x^2\theta\|_{L^1 L^2}&\le\int_{T_0}^T \|u^{(1)}_{0}\|_{L^\infty}\|\pr_x^2\theta\|\dd t\les \int_{T_0}^T \|u_{0}^{(1)}\|_{H^1}\|\e^{\delta\nu^{\frac{1}{3}}t}\pr^2_x\theta\|_{L^2}\dd t\\ &\les \int_{T_0}^T \|u_{0}^{(1)}\|_{H^1}\|\e^{\delta\nu^{\frac{1}{3}}t}\pr^2_x\theta\|_{L^2}\dd t\les\|\om\|_{L^\infty L^2}\|e^{\delta\nu^{\frac{1}{3}}t}\pr_x^2\theta\|_{L^1 L^2}\les \nu^{-\frac{1}{2}}\|\om\|_{X_2}\|\pr_x\theta\|_{X_2}.
\end{aligned}
\ee
By Proposition \ref{p2.1} and Lemma \ref{l3.2}, we get 
\be\notag
\begin{aligned}
    \|\e^{\delta\nu^{\frac{1}{3}}t}u_{0}^{(1)}\pr_x^2\theta_a\|_{L^1 L^2}&\le \int_{T_0}^T\e^{\delta\nu^{\frac{1}{3}}t}\|u_{0}^{(1)}\|\|\pr^2_x\theta_a\|_{L^\infty}\dd t\les \|\om\|_{L^\infty L^2}\int_{T_0}^T \e^{\delta\nu^{\frac{1}{3}}t}\e^{-\nu^{\frac{1}{3}}t}\|\Theta_{in}\|_{H^4} \dd t\\&\les \nu^{-\frac{1}{3}}\|\om\|_{X_2}\|\Theta_{in}\|_{H^4},
\end{aligned}
\ee
and
\be\notag
\begin{aligned}
    \|\e^{\delta\nu^{\frac{1}{3}}t}\pr_x Er^2\|_{L^1 L^2}=\int_{T_0}^T \e^{\delta\nu^{\frac{1}{3}}t}\|\pr_x Er^2\|\dd t\les \nu\int_{T_0}^T(1+t)\e^{\delta\nu^{\frac{1}{3}}t}\e^{-2\nu^{\frac{1}{3}}t}\|\Theta_{in}\|_{H^4}\dd t\le \nu^{\frac{1}{3}}\|\Theta_{in}\|_{H^4}.
\end{aligned}
\ee
Above all,
\[\|\e^{\delta\nu^{\frac{1}{3}}t}\pr_x\theta\|_{L^1 L^2}\les \nu^{-\frac{5}{6}}\|\om\|_{X_2}\|\pr_x\theta\|_{X_2}+\nu^{-\frac{2}{3}}\|\om\|_{X_2}\|\Theta_{in}\|_{H^4}+\|\Theta_{in}\|_{H^4}.\]
\end{proof}
Combining Proposition \ref{p4.8a} and Lemma \ref{prop:source-estimate}, we can get the estimate of $\om_*$. 
Then by using Lemma \ref{prop:reaction}, we can get the estimate of $\|\om\|_{X_2}$.
In addition, we can also conclude that
    \[\|\theta\|_{X_2}\les \|\theta(T_0)\|+\nu^{-\frac{2}{3}}\|\om\|_{X_2}\|\theta\|_{X_2}+\nu^{-\frac{1}{3}}\|\om\|_{X_2}\|\Theta_{in}\|_{H^4}+\nu^{\frac{1}{3}}\|\Theta_{in}\|_{H^4},\]
    and
    \[\|\pr_x\theta\|_{X_2}\les \|\pr_x\theta(T_0)\|+\nu^{-\frac{2}{3}}\|\om\|_{X_2}\|\pr_x\theta\|_{X_2}+\nu^{-\frac{1}{3}}\|\om\|_{X_2}\|\Theta_{in}\|_{H^4}+\nu^{\frac{1}{3}}\|\Theta_{in}\|_{H^4}.\]
\begin{proof}[Proof of Proposition \ref{thm:error-term} ]
    The proof is similar to that of Proposition \ref{p2.2}, and the specific details are omitted here.
\end{proof}
\begin{proof}[Proof of Theorem \ref{thm4.1}] 
    Combining Propositions \ref{prop:linear} and \ref{thm:error-term}, we can get the proof of Theorem \ref{thm4.1} immediately.
\end{proof}

\vspace{1cm}

\noindent{\bf Conflict of Interest} The authors declare that they have no conflict of interest.

\vspace{0.5cm}
\noindent{\bf Acknowledgments:} Nangao Zhang is supported by the National Natural Science Foundation of China under  Grants 12401283. Ruizhao Zi is partially  supported by the National Natural Science Foundation of China under  Grants 12571235.


\begin{thebibliography}{999}

\bibitem{Arbon2025}
 Arbon R., Bedrossian J., Quantitative hydro\dd ynamic stability for Couette flow on unbounded domains with
 Navier boundary conditions, {\it Commun. Math. Phys.}, {\bf406}, 57 pp, 2025.

\bibitem{Bedrossian20231}
Bedrossian J., Bianchini R., Coti Zelati M., Dolce M., Nonlinear inviscid damping and shear-buoyancy
instability in the two-dimensional Boussinesq equations, {\it Commun. Pure Appl. Math.}, {\bf76}, 3685--3768, 2023.

\bibitem{BPM1}
Bedrossian J., Germain P., Masmoudi N., \dd ynamics near the subcritical transition of the 3D Couette
 flow I: Below threshold case, \textit{Mem. Amer. Math. Soc.}, \textbf{266}, no. 1294, 2020.

 \bibitem{BPM2}
Bedrossian J.,Germain P., Masmoudi N., \dd ynamics near the subcritical transition of the 3D Couette
 flow II: Above threshold case, \textit{Mem. Amer. Math. Soc.}, \textbf{279}, no. 1377, 2022.

\bibitem{BPM}
Bedrossian J., Germain P., Masmoudi N., On the stability threshold for the 3D Couette flow in Sobolev regularity, \textit{Ann.
Math.}, \textbf{185}, 541--608, 2017.

\bibitem{BHW}
Bedrossian J., He S., Iyer S., Wang F., Stability threshold of nearly-Couette shear flows with Navier boundary conditions in
2D, \textit{Commun.Math.Phys.}, 406:\textbf{28}, 2025.

\bibitem{BMV}
Bedrossian J., Masmoudi N., Vicol V., Enhanced dissipation and inviscid damping in the inviscid limit of the Navier-Stokes
equations near the two dimensional Couette flow, \textit{Arch. Ration. Mech. Anal.}, \textbf{219}, 1087--1159, 2016.

\bibitem{BVW}
Bedrossian J.,Vicol V., Wang F., The Sobolev stability threshold for 2D shear flows near Couette, \textit{J. Nonlinear Sci.}, \textbf{28}, 2051--2075, 2018.


\bibitem{Bianchini2025}
Bianchini, R., Coti Zelati M., Ertzbischoff L., Ill-posedness of the hydrostatic Euler-Boussinesq equations and failure of hydrostatic limit, {\it Commun. Math. Phys.}, {\bf406}, 50 pp, 2025.

\bibitem{Cao2013}
Cao C., Wu J., Global regularity for the two-dimensional anisotropic Boussinesq equations with vertical dissipation,
 {\it Arch. Ration. Mech. Anal.}, {\bf208}, 985--1004, 2013.

\bibitem{Chae}
Chae D., Global regularity for the 2D Boussinesq equations with partial viscosity terms, {\it Adv. Math.}, {\bf 203},
 497--513, 2006.

\bibitem{Chen2020}
Chen Q., Li T., Wei D., Zhang Z., Transition threshold for the 2-D Couette flow in a finite channel,
{\it Arch. Ration. Mech. Anal.}, {\bf 238}, 125--183, 2020.

\bibitem{Chen2026}
Chen Q., Li Z., Quantitative stability of the 2D Monotone shear flow for Boussinesq equation in a finite channel, preprint (arXiv:2603.16268v1), 2026.

\bibitem{miao2025}
Chen Q., Li Z., Miao C.,
The optimal transition threshold for the 2D Couette flow in the infinite channel, preprint (arXiv:2510.18365), 2025.

\bibitem{Chen-Wang}
Chen Y., Wang W., Yang G., Stability threshold of Couette flow for Boussinesq equations in  $\mathbb{R}^2$,
preprint (arXiv:2508.11908), 2025.

\bibitem{Coti}
Coti Zelati M., Del Zotto A.,
Suppression of lift-up effect in the 3D Boussinesq equations around a stably stratified Couette flow,
{\it Q. Appl. Math.}, {\bf 83}, 389--401, 2025.

\bibitem{Danchin}
 Danchin R., Paicu M., Global well-posedness issues for the inviscid Boussinesq system with Yudovich's type data,
 {\it Commun. Math. Phys.}, {\bf290}, 1--14, 2009.

\bibitem{Deng2020}
Deng W., Wu J., Zhang P., Stability of Couette flow for 2D Boussinesq system with vertical dissipation,
{\it J. Funct. Anal.}, {\bf281}, 109255, 2021.

\bibitem{Doering1996}
Doering  C., Constantin P., Variational bounds on energy dissipation in incompressible flows, III Convection, {\it Phys.
 Rev. E}, {\bf 53}, 5957--5981, 1996.

\bibitem{Doering2018}
Doering C., Wu J., Zhao K., Zheng X., Long time behavior of the two-dimensional Boussinesq equations without
 buoyancy diffusion, {\it Phys. D}, {\bf 376}, 144--159, 2018.

\bibitem{Kelvin}
Kelvin L., Stability of fluid motion-rectilinear motion of viscous fluid between two parallel plates, 
{\it Philos. Mag.}, {\bf 24}, 188--196, 1887.

\bibitem{Knobel}
Knobel N., Suppression of Fluid Echoes and Sobolev Stability Threshold for 2D Dissipative Fluid Equations Around Couette Flow, 
preprint (arXiv:2505.23391), 2025.

\bibitem{Liang2025}
Liang T., Li Y., Zhai X., Stability threshold for two-dimensional Boussinesq systems near-Couette shear flow in a finite channel, 
preprint (arXiv:2504.02669), 2025.

\bibitem{Liang2026}
Liang T., Wu J., Zhai X., Stability threshold of the 2D Boussinesq system near Couette flow in an infinite channel, 
preprint (arXiv:2512.01159v1), 2025.

\bibitem{Lihui2026}
Li H., Liu N., Zhao W.,
Stability threshold of the two-dimensional Couette flow in the whole plane,
{\it J. Funct. Anal.},  {\bf 290}, 40 pp. 2026.

\bibitem{Li2016}
 Li J., Titi E., Global well-posedness of the 2D Boussinesq equations with vertical dissipation, {\it Arch. Ration. Mech.
 Anal.}, {\bf 220}, 983--1001, 2016.

\bibitem{MSZ}
Masmoudi N., Said-Houari B., Zhao W., Stability of the Couette flow for a 2D Boussinesq system without thermal
diffusivity, {\it Arch. Ration. Mech. Anal.}, \textbf{245}, 645--752, 2022.

\bibitem{Majda2002}
Majda A., Bertozzi A., Vorticity and Incompressible Flow, Cambridge University Press, Cambridge, 2002.

\bibitem{Masmoudi2021}
Masmoudi N., Zhai C., Zhao W., Asymptotic stability for two-dimensional Boussinesq systems around the Couette flow in a finite channel, {\it J. Funct. Anal.}, {\bf284}, 109736, 2023.



\bibitem{Niu}
Niu B., Zhao W., Improved stability threshold of the Two-Dimensional Couette flow for Navier-Stokes-
Boussinesq Systems via quasi-linearization, preprint (arXiv:2409.16216), 2024.

\bibitem{Ren}
Ren X., Wei D., Transition threshold of Couette flow for 2D 
Boussinesq equations, {\it J. Funct. Anal.}, {\bf290}, 111383, 2026.

\bibitem{Tao2019}
Tao L. , Wu J., The 2D Boussinesq equations with vertical dissipation and linear stability of shear flows, {\it J. Differential
 Equations}, {\bf267}, 1731--1747, 2019.

\bibitem{Tao2020}
Tao L. , Wu J., Zhao K., Zheng X., Stability near hydrostatic equilibrium to the 2D Boussinesq equations without
 thermal diffusion, {\it Arch. Ration. Mech. Anal.}, {\bf 237}, 585--630, 2020.

\bibitem{Villani}
Villani C., Hypocoercivity, {\it Mem. Am. Math. Soc.}, {\bf 202}, iv+141, 2009.


\bibitem{wei-zhang2023}
Wei D., Zhang Z., Nonlinear enhanced dissipation and inviscid damping for the 2D Couette flow, {\it Tunis. J. Math.}, {\bf5}, 573--592, 2023.


\bibitem{wei-zhang2026}
Wei D., Zhang Z., Stability threshold of the 2D Couette flow in a finite channel, {\it Ann. Inst. Henri Poincaré C, Anal. Non Linéaire}, {\bf 43}, 501--540, 2023.

\bibitem{zhai2022}
Zhai C., Zhao W., Stability threshold of the Couette flow for Navier-Stokes Boussinesq system with large Richardson number $\gamma^2>\frac{1}{4}$, {\it SIAM J. Math. Anal.}, {\bf55}, 1284--1318, 2023.

\bibitem{ZhangZi}
Zhang Z., Zi R., Stability threshold of Couette flow for 2D Boussinesq equations in Sobolev spaces, {\it J.
Math. Pures Appl.}, {\bf179}, 123--182, 2023.
    \end{thebibliography}
\end{document}